\documentclass[10pt,reqno]{amsart}
\usepackage[top=2.5cm, bottom=2.5cm, left=2.5cm , right=2.5cm]{geometry}
\usepackage{mathtools}
\usepackage{enumitem}
\usepackage{hyperref}
\usepackage{amsthm,amsfonts,amssymb,euscript,mathrsfs}

\usepackage[english]{babel}

\newcommand{\R}{\mathbb{R}}

\newcommand{\N}{\mathbb{N}}
\newcommand{\M}{\mathcal{M}}

\renewcommand{\d}{\mathrm{d}}

\theoremstyle{plain}

\newtheorem{lem}{Lemma}[section]
\newtheorem{thm}{Theorem}[section]
\newtheorem{prop}{Proposition}[section]
\newtheorem{coro}{Corollary}[section]
\newtheorem{mydef}{Definition}[section]

\numberwithin{equation}{section}

\newcommand{\ffi}{\varphi}
\newcommand{\fir}{\mathring{\varphi}}
\newcommand{\omegar}{\mathring{\omega}}
\newcommand{\e}{\varepsilon}
\newcommand{\dr}{\partial}
\newcommand{\dive}{\mathrm{div}}
\newcommand{\tb}{\Bar{\otimes}}
\newcommand{\Ll}{\mathbf{T}}
\newcommand{\tr}{\mathrm{tr}}

\renewcommand{\l}{\left\|}
\renewcommand{\r}{\right\|}

\newcommand{\enstq}[2]{\left\{#1~\middle|~#2\right\}}

\title[Einstein vacuum equations with $\mathbb{U}(1)$ symmetry in an elliptic gauge]{Einstein vacuum equations with $\mathbb{U}(1)$ symmetry in an elliptic gauge: local well-posedness and blow-up criterium}
\author{Arthur Touati}
\date{}
\address{Centre de Mathématiques Laurent Schwartz, Ecole Polytechnique, France}
\email{arthur.touati@polytechnique.edu}

\begin{document}

\begin{abstract}
In this article, we are interested in the Einstein vacuum equations on a Lorentzian manifold displaying $\mathbb{U}(1)$ symmetry. We identify some freely prescribable initial data, solve the constraint equations and prove the existence of a unique and local in time solution at the $H^3$ level. In addition, we prove a blow-up criterium at the $H^2$ level. By doing so, we improve a result of Huneau and Luk in \cite{hunluk18} on a similar system, and our main motivation is to provide a framework adapted to the study of high-frequency solutions to the Einstein vacuum equations done in a forthcoming paper by Huneau and Luk. As a consequence we work in an elliptic gauge, particularly adapted to the handling of high-frequency solutions, which have large high-order norms. 
\end{abstract}
\maketitle
\tableofcontents

\section{Introduction}

\subsection{Presentation of the results}

In this article, we are interesting in solving the Einstein vacuum equations 
\begin{equation*}
R_{\mu\nu}-\frac{1}{2}Rg_{\mu\nu}=0
\end{equation*}
on a four-dimensional lorentzian manifold $(\M,  ^{(4)}g)$, where $R_{\mu\nu}$ and $R$ are the Ricci tensor and the scalar curvature associated to $^{(4)}g$. We assume that the manifold $\M$ amits a translation Killing field, this symmetry being called the $\mathbb{U}(1)$ symmetry. Thanks to this symmetry, the $3+1$ Einstein vacuum equations reduce to the $2+1$ Einstein equations coupled with two scalar fields satisfying a wave map system: 
\begin{equation}\label{notre système}
    \left\{
\begin{array}{l}
\Box_g \ffi  =-\frac{1}{2}e^{-4\ffi}\dr^\rho\omega\dr_\rho\omega\\
\Box_g \omega  =4\dr^\rho\omega\dr_\rho\ffi\\
R_{\mu\nu}(g)  =2\partial_{\mu}\ffi\partial_{\nu}\ffi+\frac{1}{2}e^{-4\ffi}\dr_\mu\omega\dr_\nu\omega
\end{array}
\right. 
\end{equation}
where $\ffi$ and $\omega$ are the two scalar fields and $g$ is a $2+1$ lorentzian metric appearing in the decomposition of $^{(4)}g$ (see Section \ref{section U(1) symmetry} for more details).  

The goal of this paper is to solve the previous system in an elliptic gauge. This particular choice of gauge for the $2+1$ spacetime will be precisely defined in Section \ref{section geometrie},  but let us just say for now that it allows us to recast the Einstein equations as a system of \textit{semilinear elliptic equations} for the metric coefficients. This gauge is therefore especially useful for low-regularity problems, since it offers additional regularity for the metric.

More precisely, we obtain two results on this system: local well-posedness with some precise smallness assumptions and a blow-up criterium. Both these results can be roughly stated as follows (see Theorems \ref{theoreme principal} and \ref{theo 2} for some precise statements):

\begin{thm}[Rough version of Theorem \ref{theoreme principal}]\label{rough theo 1}
Given admissible initial data for $(\ffi,\omega)$ large in $H^3$ and small enough in $W^{1,4}$ (the smallness threshold being independent of the potentially large $H^3$-norms), there exists a unique solution to \eqref{notre système} in the elliptic gauge on $[0,T]\times \R^2$ for some $T>0$ depending on the initial $H^3$-norms.
\end{thm}

\begin{thm}[Rough version of Theorem \ref{theo 2}]\label{rough theo 2}
If the time of existence $T$ of the solution obtained in Theorem \ref{rough theo 1} is finite,  then the $H^2$ norm of $(\ffi,\omega)$ diverges at $T$ or the smallness in $W^{1,4}$ no longer holds.
\end{thm}

\subsection{Strategy of proof and main challenges}

Let us briefly discuss the strategy employed to prove the two previous theorems and point out the main challenges we face.  We adopt the same global strategy as in the work of Huneau and Luk \cite{hunluk18}, and we will discuss the differences and similarities with this article in Section \ref{section comparaison}.

\subsubsection{Theorem \ref{rough theo 1}}

In order to prove Theorem \ref{rough theo 1}, we need to solve the Einstein equations in the elliptic gauge.  As the name of the gauge suggests, the system \eqref{notre système} then reads
\begin{equation}\label{système avec jauge}
    \left\{
\begin{array}{l}
\Box_g U  =(\dr U)^2 \\
\Delta g  = (\dr U)^2 + (\dr g)^2
\end{array}
\right. 
\end{equation}
where $U$ denotes either $\ffi$ or $\omega$ and in the second equation $g$ denotes any metric coefficient. One of the main challenges of solving such a system is therefore the inversion of the Laplacian operator on a unbounded set, here $\R^2$. Indeed this will imply that some of the metric coefficients, the lapse $N$ and the conformal factor $\gamma$ (see Section \ref{section geometrie} for their definitions), presents some logarithmic growth at spacelike infinity. To counteract these growth, we work in the whole paper with weighted Sobolev spaces (see Definition \ref{wss}).

One major aspect of Theorem \ref{rough theo 1} is that the smallness assumed for the initial data is only at the $W^{1,4}$ level, while their higher order norms can be arbitrarily large.  It is quite unusual to require some smallness on the initial data to only prove \textit{local} existence, usually one would only ask for smallness on the time of existence. Here however, the smallness of the time of existence can only be of help when performing energy estimates for the hyperbolic part of \eqref{système avec jauge}. When dealing with the non-linearities in the elliptic part of \eqref{système avec jauge}, we rely on the smallness of the solution to close the hierarchy of estimates we introduce.

However, one of the strength of our result is that the smallness of the initial data is only assumed for their first derivatives in $L^4$ topology.  The higher order norm, i.e the $L^2$ norm of their second and third derivatives can be large, and this largeness is \textit{not} compensated by the smallness ot the initial data, which concretely means that the smallness threshold in Theorem \ref{rough theo 1} doesn't depend on the $H^3$ norm of the initial data.  This initial data regime (largeness in $H^3$ not compensated by smallness in $W^{1,4}$) is motivated by the main application of this article, namely to the construction of high-frequency spacetimes in the context of the Burnett conjecture in general relativity. See Huneau and Luk's article \cite{hunluk} for the application of the present article and \cite{bur89} for the original paper of Burnett.

\par\leavevmode\par

Despite the particularities of the elliptic gauge we just discussed, the global strategy to solve the Einstein vacuum equations is standard:
\begin{itemize}
\item we first solve the \textit{constraint equations} and by doing so construct initial data for the metric on the slice $\{ t=0\}$ which in particular satisfies the gauge conditions,
\item then, we solve a \textit{reduced system}, which in our case is a coupled system of elliptic, wave and transport equations,
\item finally, we prove using the Bianchi identity that solving the reduced system actually implies the full Einstein vacuum equations and the propagation of the gauge conditions.
\end{itemize}

As a final comment, note that wave map structure of the hyperbolic part of \eqref{système avec jauge} plays no role in the proof of Theorem \ref{rough theo 1}.

\subsubsection{Theorem \ref{rough theo 2}}

Inversely, the wave map structure of the coupling between the wave equations for $\ffi$ and $\omega$ is at the heart of Theorem's \ref{rough theo 1} proof. This result basically means that the $H^2$ norm of the initial data controls the time of existence of the solution (as long as the smallness in $W^{1,4}$ holds), whereas we need $H^3$ regularity to prove local existence.  This is not a consequence of the standard energy method for the wave equation, since in dimension 2 it only allows for $H^{2+\e}$ regularity. Reaching $H^2$ requires therefore to use another structure, in our case the wave map structure of the hyperbolic part of \eqref{système avec jauge}.  Since the work of Choquet-Bruhat in \cite{CBwavemaps} it is well-known that we can associate to any wave map systems a \textit{third order energy estimate}, which we crucially use to reach $H^2$.

As explain above, we rely on the smallness of the initial data to prove local-existence. This requirement has the following consequence: we are unable to prove local well-posedness at the $H^2$ level. Indeed, in order to obtain such a result, we would need (in addition to the third order energy estimate) to propagate the smallness in $W^{1,4}$ through the wave map system using only $H^2$ norms. This is not possible in dimension 2 using only energy estimates. Therefore, we need to assume that the  $W^{1,4}$ smallness is propagated,  which explains why we "only" prove a blow-up criterium and not local well-posedness at the $H^2$ level.

\subsection{Relation to previous works}

In this section we discuss the link of our work with the litterature. To say it briefly, the proof of Theorem \ref{rough theo 1} draws from \cite{hunluk18} and the proof of Theorem \ref{rough theo 2} uses tehniques from Choquet-Bruhat.

\subsubsection{An improvement of \cite{hunluk18}}\label{section comparaison}

This work has a lot of common points with the work of Huneau and Luk in \cite{hunluk18}, where they also study the system \eqref{notre système}. In this section, let us detail the similarities and differences between these two works.

\par\leavevmode\par

The system actually solved in \cite{hunluk18} is the Einstein \textit{null dust} system in \textit{polarized} $\mathbb{U}(1)$ symmetry. The polarized assumption implies $\omega=0$, and thus simplifies the hyperbolic part of the Einstein equations: a classical linear wave equation replaces our wave map system and its non-linear coupling associated to the non-polarized case we study here.

The Einstein null dust system is a particular case of Einstein Vlasov system and is translated as follows: the system studied in \cite{hunluk18} is coupled with some transport equations for massless particles along null geodesics. This involves the solving of the eikonal equation and thus requires the use of the null structure in $2+1$ dimension to avoid a loss of derivatives. Since we solve the Einstein \textit{vacuum} equations, this difficulty disappears in our work.

\par\leavevmode\par

As explained earlier in this introduction, the actual structure of the hyperbolic part of \eqref{notre système} doesn't influence the proof of the local existence of solutions. The proof given here nevertheless differs from the one of \cite{hunluk18} because of the differences in terms of regularity of the initial data. In \cite{hunluk18}, the initial data enjoy $H^4$ regularity and are small in $W^{1,\infty}$. This should be compared to our assumptions: $H^3$ regularity with smallness in $W^{1,4}$ only.  Because of this fact, the hierarchy of estimates we introduce during the bootstrap argument differs from the one introduced in \cite{hunluk18}. 

\subsubsection{Symmetry and wave maps}

As explained in the seventh chapter of the appendix of \cite{cho09}, the presence of a symmetry group acting on the spacetime generically implies the reduction of the Einstein vacuum equations into a coupled system between some Einstein-type equations and a wave map system. This is in particular the case for the $\mathbb{U}(1)$ symmetry. In \cite{Malone}, Moncrief performed this reduction and in \cite{CBM} Choquet-Bruhat and Moncrief prove local-existence at the $H^2$ level for a manifold of the form $\R_t \times \Sigma \times \mathbb{U}(1)$ where $\Sigma$ is a compact two-dimensional manifold. The compactness of $\Sigma$ allows them to use Schauder fixed point theorem, thus avoiding the need for some initial smallness. This has to be compared to the present work, where we need some initial smallness to solve the PDE system.

As explained earlier in this introduction, the wave map structure is particularly important for the proof of Theorem \ref{rough theo 2}. Indeed, as noted by Choquet-Bruhat in \cite{CBwavemaps} in the most general case, it is always possible to associate to any wave map system a third order energy estimate (see also \cite{MR2387237}).

\section{Geometrical setting}

In this section, we first introduce our notations, and then we present the $\mathbb{U}(1)$ symmetry and the elliptic gauge.

\subsection{Notations}

In this section we introduce the notations of this article. We will be working on $\mathcal{M}\vcentcolon=I\times\R^2$, where $I\subset\R$ is an interval. This space will be given a coordinate system $(t,x^1,x^2)$. We will use $x^i$ with lower case Latin index $i=1,2$ to denote the spatial coordinates.

\paragraph{Convention with indices :} 
\begin{itemize}
    \item Lower case Latin indices run through the spatial indices 1, 2, while lower case Greek indices run through all the spacetime indices. Moreover, repeat indices are always summed over their natural range.
    \item Lower case Latin indices are always raised and lowered with respect to the standard Euclidean metric $\delta_{ij}$, while lower case Greek indices are raised and lowered with respect to the spacetime metric $g$. 
\end{itemize}

\paragraph{Differential operators :} 
\begin{itemize}
    \item For a function $f$ defined on $\R^{2+1}$, we set $\dr f=(\dr_t f,\nabla f)$, where $\nabla f$ is the usual spatial gradient on $\R^2$. Samewise, $\Delta$ denotes the standard Laplacian on $\R^2$. If $A=(a_1,a_2)$ and $B=(b_1,b_2)$ are two vectors of $\R^2$, we use the dot notation for their scalar product 
    \begin{equation*}
    A\cdot B=a_1b_1+a_2b_2= \delta^{ij}a_ib_j.
    \end{equation*}
    The notation $|\cdot|$ is reserved for the norm associated to this scalar product, meaning $|A|^2=A\cdot A$.
    \item $\mathcal{L}$ denotes the Lie derivatives, $D$ denotes the Levi-Civita connection associated to the spacetime metric $g$, and $\Box_g$ denotes the d'Alembertian operator on functions :
    \begin{equation*}
        \Box_gf=\frac{1}{\sqrt{|\det(g)|}}\dr_{\mu}\left( \left(g^{-1}\right)^{\mu\nu}\sqrt{|\det(g)|}\dr_{\nu}f\right).
    \end{equation*}
    \item $L$ denotes the euclidean conformal Killing operator acting on vectors on $\R^2$ to give a symmetric traceless (with respect to $\delta$) covariant 2-tensor :
    \begin{equation*}
        (L\xi)_{ij}\vcentcolon=\delta_{j\ell}\dr_i\xi^{\ell}+\delta_{i\ell}\dr_j\xi^{\ell}-\delta_{ij}\dr_k\xi^k.
    \end{equation*}
\end{itemize}

\paragraph{Functions spaces :} We will work with standard functions spaces $L^p$, $H^k$, $C^m$, $C^{\infty}_c$, etc., and assume their standard definitions. We use the following convention :
\begin{itemize}
    \item All function spaces will be taken on $\R^2$ and the measures will be the 2D-Lebesgue measure.
    \item When applied to quantities defined on a spacetime $I\times\R^2$, the norms $L^p$, $H^k$, $C^m$ denote fixed-time norms. In particular, if in an estimate the time $t\in I$ in question is not explicitly stated, then it means that the estimate holds for all $t\in I$ for the time interval $I$ that is appropriate for the context.
\end{itemize}
We will also work in weighted Sobolev spaces, which are well-suited to elliptic equations. We recall here their definition, together with the definition of weighted Hölder space. The properties of these spaces that
we need are listed in Appendix \ref{appendix B}. We use the standard notation $\langle x \rangle=\left( 1+|x|^2\right)^{\frac{1}{2}}$ for $x\in\R^2$.

\begin{mydef}\label{wss}
Let $m\in\N$, $1<p<\infty$, $\delta\in\R$. The weighted Sobolev space $W^{m,p}_{\delta}$ is the completion of $C^{\infty}_0$ under the norm
\begin{equation*}
    \| u\|_{W^{m,p}_{\delta}}=\sum_{|\beta|\leq m} \left\| \langle x \rangle^{\delta+\beta}\nabla^{\beta}u \right\|_{L^p}.
\end{equation*}
We will use the notation $H^m_{\delta}=W^{m,2}_{\delta}$, $L^p_{\delta}=W^{0,p}_{\delta}$ and $W^{m,p}$ denotes the standard Sobolev spaces on $\R^2$.

The weighted Hölder space $C^m_{\delta}$ is the completion of $C^m_c$ under the norm
\begin{equation*}
    \| u\|_{C^m_{\delta}}=\sum_{|\beta|\leq m} \left\| \langle x \rangle^{\delta+\beta}\nabla^{\beta}u \right\|_{L^{\infty}}.
\end{equation*}
For a covariant 2-tensor $A_{ij}$ tangential to $\R^2$, we use the convention :
\begin{equation*}
    \| A\|_{X}=\sum_{i,j=1,2}  \| A_{ij}\|_{X},
\end{equation*}
where $X$ stands for any function spaces defined above.
\end{mydef}
We denote by $B_r$ the ball in $\R^2$ of radius $r$ centered at 0.

\subsection{Einstein vacuum equations with a translation Killing field}\label{section U(1) symmetry}

In this section, we present the $\mathbb{U}(1)$ symmetry. From now on, we consider a Lorentzian manifold $(I\times\R^3,^{(4)}g)$, where $I\subset\R$ is an interval, and $^{(4)}g$ is a Lorentzian metric, for which $\dr_3$ is a Killing field. Following the Appendix VII of \cite{cho09}, this is equivalent to say that $^{(4)}g$ has the following form :
\begin{equation}
    ^{(4)}g=e^{-2\varphi}g+e^{2\varphi}(\d x^3+A_\alpha\d x^\alpha)^2,
\end{equation}
where $\ffi:I\times\R^2\longrightarrow\R$ is a scalar function, $g$ is a Lorentzian metric on $I\times\R^2$ and $A$ is a 1-form on $I\times\R^2$. The \textit{polarized} $\mathbb{U}(1)$ symmetry is the case where $A=0$. We extend $\ffi$ to a function on $I\times\R^3$ in such a way that $\ffi$ does not depend on $x^3$. Given this ansatz of the metric, the vector field $\dr_3$ is Killing and hypersurface orthogonal. Assuming that the metric $^{(4)}g$ satisfies the Einstein vacuum equations, i.e $R_{\mu\nu}(^{(4)}g)=0$, one can prove that there exists a function $\omega$ such that 
\begin{equation*}
F=-e^{-3\ffi}*\d\omega
\end{equation*}
where $F_{\alpha\beta}=\dr_\alpha A_\beta-\dr_\beta A_\alpha$.
\par\leavevmode\par
The Einstein vacuum equations for $(I\times\R^3,^{(4)}g)$ are thus equivalent to the following system of equations :
\begin{equation}\label{EVE}
    \left\{
\begin{array}{l}
\Box_g \ffi  =-\frac{1}{2}e^{-4\ffi}\dr^\rho\omega\dr_\rho\omega\\
\Box_g \omega  =4\dr^\rho\omega\dr_\rho\ffi\\
R_{\mu\nu}(g)  =2\partial_{\mu}\ffi\partial_{\nu}\ffi+\frac{1}{2}e^{-4\ffi}\dr_\mu\omega\dr_\nu\omega
\end{array}.
\right. 
\end{equation}
Solving the system \eqref{EVE} is the goal of this article. Note that the last equation of \eqref{EVE} is actually the Einstein equation $G_{\mu\nu}(g)=T_{\mu\nu}$ with the following stress-energy-momentum tensor :
\begin{equation}
    T_{\mu\nu}=2\dr_{\mu}\ffi\dr_{\nu}\ffi-g_{\mu\nu}g^{\alpha\beta}\dr_{\alpha}\ffi\dr_{\beta}\ffi+\frac{1}{2}e^{-4\ffi}\left( 2\dr_{\mu}\omega\dr_{\nu}\omega-g_{\mu\nu}g^{\alpha\beta}\dr_{\alpha}\omega\dr_{\beta}\omega\right).\label{tenseur energie impulsion }
\end{equation}

\subsection{The elliptic gauge}\label{section geometrie}

In this section, we present the elliptic gauge. We first write the $(2+1)$-dimensional metric $g$ on $\mathcal{M}\vcentcolon=I\times\R^2$ in the usual form :
\begin{equation}
    g=-N^2\d t^2+\Bar{g}_{ij}\left(\d x^i+\beta^i\d t \right)\left(\d x^j+\beta^j\d t \right).
\end{equation}
Let $\Sigma_t\vcentcolon=\enstq{(s,x)\in\mathcal{M}}{s=t}$ and $e_0\vcentcolon=\dr_t-\beta^i\dr_i$, which is a future directed normal to $\Sigma_t$.  The function $N$ is called the lapse and the vector field $\beta$ is the shift. We introduce $\Ll\vcentcolon=\frac{e_0}{N}$, the unit future directed normal to $\Sigma_t$.
We introduce the second fundamental form of the embedding $\Sigma_t\xhookrightarrow{}\mathcal{M}$
\begin{equation}
    K_{ij}\vcentcolon=-\frac{1}{2N}\mathcal{L}_{e_0}\Bar{g}_{ij}.\label{Kij}
\end{equation}
We decompose $K$ into its trace and traceless part :
\begin{equation}
    K_{ij}=H_{ij}+\frac{1}{2}\Bar{g}_{ij}\tau,\label{def H}
\end{equation}
where $\tau=\tr_{\Bar{g}}K$. We introduce the following gauge conditions, which define the elliptic gauge :
\begin{itemize}
    \item $\Bar{g}$ is conformally flat, i.e there exists a function $\gamma$ such that
    \begin{equation}\label{g bar}
        \Bar{g}_{ij}=e^{2\gamma}\delta_{ij}.
    \end{equation}
    \item the hypersurfaces $\Sigma_t$ are maximal, which means that $K$ is traceless, i.e 
    \begin{equation}
        \tau=0.
    \end{equation}
\end{itemize}
Thus, the metric takes the following form :
\begin{equation}
    g=-N^2\d t^2+e^{2\gamma}\delta_{ij}\left(\d x^i+\beta^i\d t \right)\left(\d x^j+\beta^j\d t \right).\label{metrique elliptique}
\end{equation}
The main computations in the elliptic gauge are performed in Appendix \ref{appendix A}. They show that \eqref{EVE} is schematically of the form 
\begin{equation}\label{EVE2}
    \left\{
\begin{array}{l}
\Box_g U  =(\dr U)^2 \\
\Delta g  = (\dr U)^2+ (\dr g)^2
\end{array}
\right. 
\end{equation}
where $U$ denotes either $\ffi$ or $\omega$ and in the second equation $g$ denotes any metric coefficient.

\section{Main results}

\subsection{Initial data}\label{subsection initial data}

We now describe our choice of initial data for the system \eqref{EVE}. We distinguish the \textit{admissible} initial data, and the \textit{admissible free} initial data. 

For the rest of this paper, we choose a fixed smooth cutoff function $\chi:\R\to\R$ such that $\chi_{|[-1,1]}=0$ and $\chi_{|[-2,2]}=0$. The notation $\chi\ln$ stands for the function $x\in\R^2\longmapsto \chi(|x|)\ln(|x|)$. 

\begin{mydef}[Admissible initial data]
For $-1<\delta<0$ and $R>0$, an admissible initial data set with respect to the elliptic gauge for \eqref{EVE} consists of
\begin{enumerate}
    \item A conformally flat intrinsic metric $(e^{2\gamma}\delta_{ij})_{|\Sigma_0}$ which admits a decomposition 
    \begin{equation*}
        \gamma=-\alpha\chi\ln+\Tilde{\gamma},
    \end{equation*}
    where $\alpha\geq 0$ is a constant and $\Tilde{\gamma}\in H^{4}_{\delta}$.
    \item A second fundamental form $(H_{ij})_{|\Sigma_0}\in H^{3}_{\delta+1}$ which is traceless.
    \item $\left( \Ll\ffi,\nabla\ffi \right)_{|\Sigma_0}\in H^2$, compactly supported in $B_R$.
    \item $\left( \Ll\omega,\nabla\omega \right)_{|\Sigma_0}\in H^2$, compactly supported in $B_R$.
    \item $\gamma$ and $H$ are required to satisfy the following constraint equations :
    \begin{align*}
    \dr^iH_{ij}&=-2e^{2\gamma}\Ll\ffi\dr_j\ffi-\frac{1}{2}e^{-4\ffi+2\gamma}\Ll\omega\dr_j\omega, \\
    \Delta\gamma&=-\frac{e^{-2\gamma}}{2}\vert H\vert^2-e^{2\gamma}\left(\Ll\ffi\right)^2-\vert\nabla\ffi\vert^2-\frac{e^{-4\ffi}}{4}\left( e^{-2\gamma}\left(\Ll\omega\right)^2+|\nabla\omega|^2 \right).
\end{align*}
\end{enumerate}
\end{mydef}

We recall that the constraint equations for the Einstein Vacuum Equations are $G_{00}=T_{00}$ and $G_{0i}=T_{0i}$, where $G_{\mu\nu}=R_{\mu\nu}-\frac{1}{2}Rg_{\mu\nu}$ is the Einstein tensor, and $T_{\mu\nu}$ is the stress-energy-momentum tensor associated to the matter fields $\ffi$ and $\omega$, according to the RHS of system \eqref{EVE}. The fact that these equations reduces in the elliptic gauge to the previous equations on $H$ and $\gamma$ can be proved using the computations done in Appendix \ref{appendix A}.
\par\leavevmode\par
We define the notion of admissible free initial data as follows :

\begin{mydef}[Admissible free initial  data]
We set $\fir=e^{2\gamma}\Ll\ffi$ and $\omegar=e^{2\gamma}\Ll\omega$, where $\gamma$ is as in \eqref{metrique elliptique}. For $-1<\delta<0$ and $R>0$, an admissible free initial  data set with respect to the elliptic gauge for \eqref{EVE} is given by $(\fir,\nabla\ffi)_{|\Sigma_0}\in H^2$ and $(\omegar,\nabla\omega)_{|\Sigma_0}\in H^2$, all compactly supported in $B_R$, satisfying
\begin{equation}\label{orthogonality condition}
    \int_{\R^2}\left(2\fir\dr_j\ffi+\frac{1}{2}e^{-4\ffi}\omegar\dr_j\omega\right)\d x=0, \quad j=1,2.
\end{equation}
\end{mydef}

The interest of the admissible free initial data is that we can construct from them a set of admissible initial data, which in particular satisfies the constraint equations. Note that instead of prescribing $\Ll \ffi$ and $\Ll\omega$, we prescribe a suitable rescaled version of them, which allows the decoupling of the two constraint equations : we will first solve for $H$ and then for $\gamma$.

\subsection{Statement of the theorems}

The following is our main result on local well-posedness for \eqref{EVE}.

\begin{thm}\label{theoreme principal}
Let $-1<\delta<0$ and $R>0$. Given an admissible free initial data set such that
\begin{equation*}
    \left\| \fir \right\|_{L^4}+\left\| \nabla\ffi \right\|_{L^4}+\left\| \omegar \right\|_{L^4}+\left\| \nabla\omega \right\|_{L^4}\leq \e,
\end{equation*}
and 
\begin{equation*}
    C_{high}\vcentcolon=\left\| \fir \right\|_{H^2}+\left\| \nabla\ffi \right\|_{H^2}+\left\| \omegar \right\|_{H^2}+\left\| \nabla\omega \right\|_{H^2}<\infty,
\end{equation*}
for any $C_{high}$, there exists a constant $\e_{0}=\e_{0}(\delta,R)>0$ independent of $C_{high}$ and a time $T=T(C_{high},\delta,R)>0$ such that, if $0<\e\leq\e_{0}$, there exists a unique solution to \eqref{EVE} in elliptic gauge on $[0,T]\times\R^2$. 

Moreover, defining $\delta'=\delta-\e$, there exists a constant $C_h=C_h(C_{high},\delta,R)>0$ such that
\begin{itemize}
    \item The fields $\ffi$ and $\omega$ satisfy for all $t\in [0,T]$ 
    \begin{align*}
        \left\| \dr_t^2\ffi \right\|_{H^1}+\left\| \dr_t\ffi \right\|_{H^2}+\left\| \nabla\ffi \right\|_{H^2}&\leq C_h,\\
        \left\| \dr_t^2\omega \right\|_{H^1}+\left\| \dr_t\omega \right\|_{H^2}+\left\| \nabla\omega \right\|_{H^2}&\leq C_h,
    \end{align*}
    and their supports are both included in $ J^+\left(\Sigma_0\cap B_R\right)$, where $J^+$ denotes the causal future.
    \item The metric components $\gamma$ and $N$ can be decomposed as
    \begin{equation*}
        \gamma=-\alpha\chi\ln+\Tilde{\gamma},\quad N=1+N_a\chi\ln+\Tilde{N},
    \end{equation*}
    with $\alpha\geq 0$ and $N_a(t)\geq 0$ a function of $t$ alone.
    \item $\gamma$, $N$ and $\beta$ satisfy the following estimates for all $t\in[0,T]$ :
    \begin{align*}
        |\alpha|+\left\|\Tilde{\gamma}\right\|_{H^4_{\delta}}+\left\|\dr_t\Tilde{\gamma}\right\|_{H^3_{\delta}}+\left\|\dr_t^2\Tilde{\gamma}\right\|_{H^2_{\delta}} & \leq C_h,\\
        \left|N_a\right|+\left|\dr_tN_a\right|+\left|\dr_t^2N_a\right| & \leq C_h,\\
        \left\|\Tilde{N}\right\|_{H^4_{\delta}}+\left\|\dr_t\Tilde{N}\right\|_{H^3_{\delta}}+\left\|\dr_t^2\Tilde{N}\right\|_{H^2_{\delta}} & \leq C_h,\\
        \left\| \beta \right\|_{H^4_{\delta'}} +\left\| \dr_t\beta \right\|_{H^3_{\delta'}}+\left\| \dr_t^2\beta \right\|_{H^2_{\delta'}}&\leq C_h.
    \end{align*}
    \item The following conservation laws hold :
    \begin{align}
        \int_{\R^2}\left(4e^{2\gamma}\Ll\ffi\dr_j\ffi+e^{-4\ffi+2\gamma}\Ll\omega\dr_j\omega\right)\d x&=0,\label{CL1}\\
        \int_{\R^2}\left(2e^{-2\gamma}|H|^2+4e^{2\gamma}(\Ll\ffi)^2+e^{-4\ffi+2\gamma}(\Ll\omega)^2+4|\nabla\ffi|^2+e^{-4\ffi}|\nabla\omega|^2 \right)\d x &=4\alpha.\label{CL2}
    \end{align}
\end{itemize}
\end{thm}

This theorem has the following corollary, which basically states that if we want to have $T=1$, it suffices to take $C_{high}$ small enough. We will omit the details of its proof because it is actually simplier than the proof of Theorem \ref{theoreme principal}.

\begin{coro}\label{CORO}
Suppose the assumptions of Theorem \ref{theoreme principal} hold. There exists $\e_{small}=\e_{small}(\delta,R)>0$ such that if $C_{high}$ and $\e$ in Theorem \ref{theoreme principal} satisfy
\begin{equation*}
    C_{high},\e\leq\e_{small},
\end{equation*}
then the unique solution exists in $[0,1]\times\R^2$. Moreover, there exists $C_0=C_0(\delta,R)$ such that all the estimates in Theorem \ref{theoreme principal} hold with $C_h$ replaced by $C_0\e$.
\end{coro}

\par\leavevmode\par
We also prove the following theorem, which is a blow-up criterium (See the introduction of Section \ref{section theo 2} for a discussion of this theorem) :

\begin{thm}\label{theo 2}
Let $T>0$ be the maximal time of existence of the solution of \eqref{EVE} obtained in Theorem \ref{theoreme principal}. If $T<+\infty$ and $\e_0$ is small enough (still independent of $C_{high}$) then one of the following holds :
\begin{enumerate}[label=\roman*)]
\item $\sup_{[0,T)} \left( \l \dr \ffi \r_{H^1} + \l \dr\omega \r_{H^1} \right)=+\infty$,
\item $\sup_{[0,T)} \left( \l \dr \ffi \r_{L^4} + \l \dr\omega \r_{L^4} \right)>\e_0$,
\end{enumerate}
\end{thm}

As said in the introduction, one major feature of these two theorems is that the smallness constant $\e_0$ does not depend on $C_{high}$. Note in contrast that the time of existence $T$ does depend on $C_{high}$.

\subsection{The reduced system}\label{section reduced system}

In order to solve the Einstein Equations, we will first solve the following system, which we call the reduced system. It is identical to the one introduced in \cite{hunluk18}.
\begin{align}
     \Delta N&=e^{-2\gamma}N\vert H\vert^2+\frac{\tau^2}{2}e^{2\gamma}N+\frac{2e^{2\gamma}}{N}(e_0\ffi)^2+\frac{e^{2\gamma-4\ffi}}{2N}(e_0\omega)^2,\label{EQ N} \\
     L\beta&=2e^{-2\gamma}NH,\label{EQ beta} \\
     N\tau&=-2e_0\gamma+\dive\beta,\label{EQ tau} \\
    e_0H_{ij}&=-2e^{-2\gamma}NH_i^{\;\,\ell}H_{j\ell}+\dr_{(j}\beta^kH_{i)k}-\frac{1}{2}(\dr_i\tb\dr_j)N+(\delta_i^k\tb\dr_j\gamma)\dr_kN\nonumber\\&\qquad\qquad\qquad\qquad\qquad-(\dr_i\ffi\tb\dr_j\ffi)N-\frac{1}{4}e^{-4\ffi}(\dr_i\omega\tb\dr_j\omega)N,\label{EQ H} \\
    \Ll^2\gamma-e^{-2\gamma}\Delta\gamma&=-\frac{\tau^2}{2}+\frac{1}{2}\Ll\left(\frac{\dive(\beta)}{N}\right)+e^{-2\gamma}\left(\frac{\Delta N}{2N}+\vert\nabla\ffi\vert^2+\frac{1}{4}e^{-4\ffi}\left|\nabla\omega \right|^2\right),\label{EQ gamma} \\
    \Ll^2\ffi-e^{-2\gamma}\Delta \ffi&=\frac{e^{-2\gamma}}{N}\nabla \ffi\cdot\nabla N+\tau\Ll\ffi+\frac{1}{2}e^{-4\ffi}\left( (e_0\omega)^2+|\nabla\omega|^2 \right),\label{EQ ffi}\\
    \Ll^2\omega-e^{-2\gamma}\Delta \omega&=\frac{e^{-2\gamma}}{N}\nabla \omega\cdot\nabla N+\tau\Ll\omega-4e_0\omega e_0\ffi-4\nabla\omega\cdot\nabla\ffi,\label{EQ omega}
\end{align}
where we use the notation $u_i\tb v_j=u_iv_j +u_jv_i-\delta_{ij}u^kv_k$.

Let us explain where equations \eqref{EQ N}-\eqref{EQ omega} come from :
\begin{itemize}
    \item Considering \eqref{appendix R00} and \eqref{JSP}, the equation $R_{00}=T_{00}-g_{00}\tr_gT$ without the term in $e_0\tau$ gives \eqref{EQ N}.
    \item For $\beta$ and $\tau$, the equations \eqref{EQ beta} and \eqref{EQ tau} simply come from \eqref{appendix beta} and \eqref{appendix tau}.
    \item To obtain the equation for $H$, we basically take the traceless part of $R_{ij}$. More precisely, using \eqref{appendix Rij}, \eqref{appendix trace ricci}, \eqref{JSP} and \eqref{JSP 3} the equation 
    \begin{equation*}
       R_{ij}-\frac{1}{2}\delta_{ij}\delta^{k\ell}R_{k\ell}=T_{ij}-g_{ij}\mathrm{tr}_gT-\frac{1}{2}\delta_{ij}\delta^{k\ell}\left(T_{k\ell}-g_{k\ell}\mathrm{tr}_gT \right) 
    \end{equation*}
    gives \eqref{EQ H}.
    \item Considering \eqref{appendix trace ricci} and \eqref{JSP 2}, the equation $\delta^{ij}R_{ij}=\delta^{ij}(T_{ij}-g_{ij}\mathrm{tr}_gT)$ reads :
    \begin{equation*}
        \Delta\gamma=\frac{\tau^2}{2}e^{2\gamma}-\frac{e^{2\gamma}}{2}\Ll\tau-\frac{\Delta N}{2N}-\left|\nabla\ffi\right|^2-\frac{1}{4}e^{-4\ffi}\left|\nabla\omega \right|^2.
    \end{equation*}
    Using \eqref{appendix tau}, we can compute $\Ll\tau$ and inject it in the previous equation to obtain \eqref{EQ gamma}.
    \item For the equation on the matter fields $\ffi$ and on $\omega$, we simply use Proposition \ref{appendix box} to rewrites $\Box_g\ffi$ and $\Box_g\omega$. 
\end{itemize}
After obtaining a solution to the reduced system, our next task will be to prove that this solution is in fact a solution of \eqref{EVE}.

Note that $H$ and $\gamma$ no longer satisfy elliptic equations, whereas in the "full" Einstein equations in the elliptic gauge they do. We follow this strategy to avoid to propagate the two conservation laws \eqref{CL1} and \eqref{CL2}, which would have been essential for solving elliptic equations and obtain a suitable behavior at spacelike infinity for $H$ and $\gamma$. Since we assume these conservation laws to hold initially, we do obtain this behavior while solving the constraints equations.

Therefore, only $N$ and $\beta$ satisfy elliptic equations, and the reduced system is a coupled hyperbolic-elliptic-transport system. 

\subsection{Outline of the proof}

We briefly discuss the structure of this article, which aims at proving Theorems \ref{theoreme principal} and \ref{theo 2}.
\par\leavevmode\par
First of all, in Section \ref{initial data}, we solve the constraints equations. More precisely we prove that an admissible free initial data set gives rise to an actual admissible initial data set, thus satisfying the constraints equations. Then, we split the proof of Theorem \ref{theoreme principal} into two parts :
\begin{itemize}
\item in Section \ref{section solving the reduced system}, we solve the reduced system \eqref{EQ N}-\eqref{EQ omega} using an iteration scheme, with initial data given by Section \ref{initial data}. During this iteration scheme, we first prove that our sequence of approximate solution is uniformaly bounded (see Section \ref{uniforme boundedness}) and then that it is a Cauchy sequence (see Section \ref{Cauchy}).
\item in Section \ref{section end of proof}, we prove that the solution to the reduced system is indeed a solution to $\eqref{EVE}$ and that it satisfies all the estimates stated in Theorem \ref{theoreme principal}.
\end{itemize}

We prove Theorem \ref{theo 2} in Section \ref{section theo 2}, using a continuity argument based on a special energy estimate which suits the wave map structure of the coupled system satisfied by $\ffi$ and $\omega$.

\par\leavevmode\par
Finally, this article contains two appendices :
\begin{itemize}
\item Appendix \ref{appendix A} presents the computations of the connection coefficients and the Ricci tensor in the elliptic gauge, as well as some formulae related to the stress-energy-momentum tensor. 
\item Appendix \ref{appendix B} presents the main tools regarding the spaces $W^{m,p}_{\delta}$ : embeddings results, product laws, and a theorem due to McOwen which allows us to solve elliptic equations on thoses spaces. It ends with some standard inequalities used in the proof.
\end{itemize}

\section{Initial data and the constraints equations}\label{initial data}

In this section, we follow \cite{hun16} and discuss the initial data for the reduced system, and in particular we solve the constraints equations. More precisely, we will show that an admissible free initial data set gives rise to a unique admissible initial data set satisfying the constraint equations.

We will then derive the initial data for $N$ and $\beta$ and prove their regularity properties. Note that, since $\fir$ and $\omegar$ are prescribed, once we have the initial data for $N$, $\beta$ and $\gamma$, we obtain the initial data for $\dr_t\ffi$ and $\dr_t\omega$.

We will only care about highlighting the dependence on $\e$ and $C_{high}$ in the following estimates and will use the notation $\lesssim$ where the implicit constant only depends on $\delta$, $R$ or on any constants coming from embeddings results. 

\par\leavevmode\par
Before we go into solving the constraint equations, let us prove a simple lemma which will allows us to deal with the $e^{-4\ffi}$ and $e^{\pm2\gamma}$ factors, which will occur many times in the equations. 

\begin{lem}\label{useful lem}
Let $\gamma=-\alpha\chi\ln+\Tilde{\gamma}$ be a function on $\R^2$ such that $0\leq\alpha\leq 1$, $\|\Tilde{\gamma}\|_{H^2_{\delta}}\leq 1$ and $\ffi\in H^3$ a compactly supported function on $\R^2$ such that $\l\ffi\r_{W^{1,4}}\leq\e$. Then, for all functions $f$ on $\R^2$ and $\nu\in\R$, the following estimates holds for $k=0,1,2$ :
\begin{align}
    \left\|e^{-2\gamma}f \right\|_{H^k_{\nu}}&\lesssim \|f\|_{H^k_{\nu+2\alpha}}\label{useful gamma 1},\\
    \left\|e^{2\gamma}f \right\|_{H^k_{\nu}}&\lesssim \|f\|_{H^k_{\nu}}\label{useful gamma 1a},\\
    \l e^{-4\ffi}f\r_{H^{k}_\nu}&\lesssim \l f\r_{H^{k}_\nu}+k\l\nabla\ffi\r_{H^2}\l f\r_{H^{k'-1}}+k(k-1)\l\nabla^2\ffi f\r_{L^2}\label{useful ffi 1}.
\end{align}
Moreover, if in addition $\|\nabla\Tilde{\gamma}\|_{H^2_{\delta'+1}}<\infty$, the following estimate holds :
\begin{align}
    \left\|e^{-2\gamma}f \right\|_{H^3_{\nu}}&\lesssim \|f\|_{H^3_{\nu+2\alpha}}+\|\nabla\Tilde{\gamma}\|_{H^2_{\delta'+1}}\|f\|_{H^2_{\nu+2\alpha}},\label{useful gamma 2}\\
    \left\|e^{2\gamma}f \right\|_{H^3_{\nu}}&\lesssim \|f\|_{H^3_{\nu}}+\|\nabla\Tilde{\gamma}\|_{H^2_{\delta'+1}}\|f\|_{H^2_{\nu}}.\label{useful gamma 2a}
\end{align}
\end{lem}
\begin{proof}

We recall the embedding $H^2_{\delta}\xhookrightarrow{}L^{\infty}$, which implies that $\left| e^{-2\Tilde{\gamma}}\right|\lesssim 1$, which allows us to forget about these factors in the following computations. Similarly, we have $\left|e^{2\alpha\chi\ln} \right|\lesssim \langle x\rangle^{2\alpha}$, which will be responsible for the change of decrease order (this remark also implies that proving \eqref{useful gamma 1} and \eqref{useful gamma 2} is enough to get \eqref{useful gamma 1a} and \eqref{useful gamma 2a}). 

Moreover, we only prove \eqref{useful gamma 2}, since it will be clear that its proof will include the proof of \eqref{useful gamma 1}. With these remarks in mind, we compute directly :
\begin{align*}
    \left\| e^{-2\gamma}f \right\|_{H^3_{\nu}} & \lesssim  \left\|f \right\|_{H^3_{\nu+2\alpha}}+ \left\|\nabla\gamma f \right\|_{L^2_{\nu+2\alpha+1}}
     + \left\|\nabla^2\gamma f \right\|_{L^2_{\nu+2\alpha+2}}\\&\qquad+ \left\|\left( \nabla\gamma\right)^2f \right\|_{L^2_{\nu+2\alpha+2}} +\left\| \nabla\gamma\nabla f\right\|_{L^2_{\nu+2\alpha+2}} +\left\|\nabla^2\gamma\nabla f \right\|_{L^2_{\nu+2\alpha+3}}
     \\&\qquad+\left\|\left(\nabla\gamma \right)^2\nabla f \right\|_{L^2_{\nu+2\alpha+3}}+\left\|\nabla\gamma\nabla^2 f \right\|_{L^2_{\nu+2\alpha+3}}
     \\&\qquad+\left\| \nabla^3\gamma f\right\|_{L^2_{\nu+2\alpha+3}}+\left\|\nabla\gamma\nabla^2\gamma f \right\|_{L^2_{\nu+2\alpha+3}}+\left\|\left(\nabla\gamma \right)^3f \right\|_{L^2_{\nu+2\alpha+3}}
\end{align*}

Because of $\left|\nabla^{a}(\chi\ln) \right|\lesssim \langle x\rangle^{-|a|}$ (which is valid for every multi-index $ a\neq 0$), we can forget about the $\chi\ln$ part in $\gamma$ and pretend that $\gamma$ is replaced by $\Tilde{\gamma}$. Using the product estimate (see Proposition \ref{prop prod}), we can deal with all these terms :

\begin{center}
\begin{tabular}{ c c }
 $\left\|\nabla\Tilde{\gamma} f \right\|_{L^2_{\nu+2\alpha+1}}\lesssim\|\nabla\Tilde{\gamma}\|_{H^1_{\delta+1}}\|f\|_{H^1_{\nu+2\alpha+1}},\quad$  & $\left\|\nabla^2\Tilde{\gamma} f \right\|_{L^2_{\nu+2\alpha+2}}\lesssim \left\| \nabla^2\Tilde{\gamma}\right\|_{L^2_{\delta+2}}\| f\|_{H^2_{\nu+2\alpha}}$,  \\ 
 $\left\|\left(\nabla\Tilde{\gamma}\right)^2 f \right\|_{L^2_{\nu+2\alpha+2}}\lesssim \|\nabla\Tilde{\gamma}\|_{H^1_{\delta+1}}^2\|f\|_{H^2_{\nu+2\alpha}},\quad$ & $\left\|\nabla\Tilde{\gamma} \nabla f \right\|_{L^2_{\nu+2\alpha+2}}\lesssim\| \nabla\Tilde{\gamma}\|_{H^1_{\delta+1}}\|\nabla f\|_{H^1_{\nu+2\alpha+1}}$, \\
 $\left\|\nabla^2\Tilde{\gamma} \nabla f \right\|_{L^2_{\nu+2\alpha+3}}\lesssim \left\|\nabla^2\Tilde{\gamma} \right\|_{L^2_{\delta+2}}\|\nabla f\|_{H^2_{\nu+2\alpha+1}},\quad$ & $\left\|\left(\nabla\Tilde{\gamma}\right)^2 \nabla f \right\|_{L^2_{\nu+2\alpha+3}}\lesssim \|\nabla\Tilde{\gamma}\|_{H^1_{\delta+1}}^2\|\nabla f\|_{H^2_{\nu+2\alpha+1}},$ \\
 $\left\|\nabla\Tilde{\gamma} \nabla^2f \right\|_{L^2_{\nu+2\alpha+3}}\lesssim \|\nabla\Tilde{\gamma}\|_{H^1_{\delta+1}}\left\|\nabla^2 f\right\|_{H^1_{\nu+2\alpha+2}},\quad$ & $\left\|\nabla^3\Tilde{\gamma} f \right\|_{L^2_{\nu+2\alpha+3}}\lesssim \left\|\nabla^3\Tilde{\gamma}\right\|_{L^2_{\delta'+3}}\|f\|_{H^2_{\nu+2\alpha}},$\\
 $\left\|\nabla\Tilde{\gamma}\nabla^2\Tilde{\gamma} f \right\|_{L^2_{\nu+2\alpha+3}}\lesssim \|\nabla\Tilde{\gamma}\|_{H^2_{\delta'+1}}\left\|\nabla^2\Tilde{\gamma}f \right\|_{L^2_{\nu+2\alpha+2}},\quad$ & $\left\|\left(\nabla\Tilde{\gamma}\right)^3 f \right\|_{L^2_{\nu+2\alpha+3}}\lesssim \|\nabla\Tilde{\gamma}\|_{H^2_{\delta'+1}}\left\|\left(\nabla\Tilde{\gamma}\right)^2 f \right\|_{L^2_{\nu+2\alpha+3}}.$
\end{tabular}
\end{center}
Note that the last two estimates involve $\left\|\nabla^2\Tilde{\gamma}f \right\|_{L^2_{\nu+2\alpha+2}}$ and $ \left\|\left(\nabla\Tilde{\gamma}\right)^2 f \right\|_{L^2_{\nu+2\alpha+3}}$, which have already been estimated. Looking at these estimates, we see that the only ones which uses $\| \nabla\Tilde{\gamma}\|_{H^2_{\delta'+1}}$ are the three last ones. Those terms don't appear if we only differentiate twice or less, it is therefore clear why \eqref{useful gamma 1} is also proved. The proof of \eqref{useful ffi 1} is identical, using the embeddings $W^{1,4}\xhookrightarrow{}L^{\infty}$ and $H^2\xhookrightarrow{}L^{\infty}$.
\end{proof}

\subsection{The constraints equations}\label{section constraints equations}

We are now ready to solve the constraints equations, which we rewrite in terms of $\fir$ and $\omegar$ :
\begin{align}
    \dr^iH_{ij}&=-2\fir\dr_j\ffi-\frac{1}{2}e^{-4\ffi}\omegar\dr_j\omega, \label{C1}\\
    \Delta\gamma&=-e^{-2\gamma}\left(\fir^2+\frac{1}{4}e^{-4\ffi}\omegar^2+\frac{1}{2}\vert H\vert^2\right)-\vert\nabla\ffi\vert^2-\frac{1}{4}e^{-4\ffi}|\nabla\omega|^2.\label{C2}
\end{align}

\begin{lem}\label{CI sur H}
The equation \eqref{C1} admits a unique solution $H\in H^3_{\delta+1}$, a symmetric traceless covariant 2-tensor with $\Vert H\Vert_{H^1_{\delta+1}}\lesssim \e^2$.
\end{lem}

\begin{proof}
We look for a solution under the form $H=LY$ where $Y$ is a 1-form. We have $\dr^iH_{ij}=\Delta Y_j$ and $Y$ solves 
\begin{equation*}
    \Delta Y_j=-2\fir\dr_j\ffi-\frac{1}{2}e^{-4\ffi}\omegar\dr_j\omega.
\end{equation*}
Using the definition of $L$, it's easy to check that $LY$ is a traceless symmetric 2-tensor. We use the Theorem \ref{mcowens 1} in the case $p=2$ and $m=0$, the range of the Laplacian is then the functions $f\in H^0_{\delta+2}$ such that $\int f=0$. By assumption, $\int_{\R^2}\left(-2\fir\dr_j\ffi-\frac{e^{-4\ffi}}{2}\omegar\dr_j\omega\right)\d x=0$ and thanks to the support property, the Hölder inequality and \eqref{useful ffi 1} we have :
\begin{equation*}
    \left\| \Delta Y_j\right\|_{ H^0_{\delta+2}}\lesssim  \left\| \fir\dr_j\ffi\right\|_{ L^2}+\left\| \omegar\dr_j\omega\right\|_{ L^2}\lesssim \l\fir\r_{L^4}\l\nabla\ffi\r_{L^4}+\l\omegar\r_{L^4}\l\nabla\omega\r_{L^4}\lesssim \e^2.
\end{equation*}
Thus, there exists a unique solution $Y_j\in H^2_{\delta}$. Moreover we have $\Vert Y_j\Vert_{H^2_{\delta}}\leq \e^2$, which implies $\Vert H\Vert_{H^1_{\delta+1}}\leq \e^2$. 

We can improve the regularity of $H$, by noting that
\begin{equation*}
    \Vert H\Vert_{H^3_{\delta+1}}\leq \Vert Y\Vert_{H^4_{\delta}}\lesssim\Vert \fir\nabla\ffi\Vert_{H^2}+\Vert \omegar\nabla\omega\Vert_{H^2}\lesssim C_{high}^2.
\end{equation*}
In the last inequality we use the fact that in dimension 2, $H^2$ is an algebra.

Our solution $H\in H^3_{\delta+1}$ is unique, because of the following fact : if $H\in H^3_{\delta+1}$ is a traceless symmetric divergence free 2-tensor, we have componentwise $\Delta H_{ij}=0$, which implies $H=0$, again thanks to Theorem \ref{mcowens 1}.
\end{proof}

\begin{lem}\label{lem CI gamma}
For $\e$ sufficiently small, the equation \eqref{C2} admits a unique solution $\gamma=-\alpha \chi\ln+\Tilde{\gamma}$ with $\Tilde{\gamma}\in H^4_{\delta}$, $\Vert \Tilde{\gamma}\Vert_{H^2_{\delta}}\lesssim \e^2$ and $0\leq\alpha\lesssim \e^2$.
\end{lem}

\begin{proof}
We are going to use a fixed point argument in $[0,\e]\times B_{H^2_{\delta}}(0,\e)$. We define on this space the application $\phi:(\alpha^{(1)},\Tilde{\gamma}^{(1)})\longmapsto(\alpha^{(2)},\Tilde{\gamma}^{(2)})$, where $\gamma^{(2)}$ is the unique solution of
\begin{equation}
    \Delta\gamma^{(2)}=-\vert\nabla\varphi\vert^2-\frac{1}{4}e^{-4\ffi}|\nabla\omega|^2-e^{-2\gamma^{(1)}}\left( \frac{1}{2}\vert H\vert^2+\fir^2+\frac{1}{4}e^{-4\ffi}\omegar^2\right),\label{contraction 1}
\end{equation}
with the notation $\gamma^{(i)}=-\alpha^{(i)}\chi(r)\ln(r)+\Tilde{\gamma}^{(i)}$. We want to prove that if $\e$ is small enough, $\phi$ is indeed a contraction.
Let us show that the RHS of \eqref{contraction 1} is in $H^0_{\delta+2}$. By assumption on $\ffi$ and $\omega$ we can write, using Hölder's inequality and \eqref{useful ffi 1} :
\begin{equation*}
    \left\| \vert \nabla \ffi \vert^2+\frac{1}{4}e^{-4\ffi}|\nabla\omega|^2\right\|_{ H^0_{\delta+2}}\lesssim \left\| \vert \nabla \ffi \vert^2\right\|_{ L^2}+\left\| \vert \nabla \omega \vert^2\right\|_{ L^2}\lesssim \e^2.
\end{equation*}

For the term $e^{-2\gamma^{(1)}}|H|^2$, we use \eqref{useful gamma 1}, the product estimate and choose $\e$ small enough :
\begin{equation*}
    \left\| e^{-2\gamma^{(1)}}|H|^2 \right\|_{ H^0_{\delta+2}}\lesssim \left\| |H|^2\right\|_{H^0_{\delta+2(1+\e)}} \lesssim \Vert H\Vert_{H^1_{\delta+1}}^2\lesssim \e^4.
\end{equation*}
The last terms is handled with the same arguments :
\begin{equation*}
    \left\| e^{-2\gamma^{(1)}}\left(\fir^2+\frac{1}{4}e^{-4\ffi}\omegar^2 \right)\right\|_{H^0_{\delta+2}} \lesssim \left\| \fir^2\right\|_{L^2} +\left\| \omegar^2\right\|_{L^2}\lesssim \e^2.
\end{equation*}
We next prove the bound on $\alpha^{(2)} =-\frac{1}{2\pi}\int_{\R^2}(\text{RHS of \eqref{contraction 1}})$, its positivity being clear. We have 
\begin{align*}
    \left|\alpha^{(2)}\right| & \lesssim \left\| \nabla\ffi \right\|_{L^2}^2+ \left\| \nabla\omega \right\|_{L^2}^2 + \left\|e^{-\gamma^{(1)}} H  \right\|_{L^2}^2 + \left\| e^{-\gamma^{(1)}}\fir \right\|_{L^2}^2+ \left\| e^{-\gamma^{(1)}-2\ffi}\omega \right\|_{L^2}^2
    \\& \lesssim\left\| \nabla\ffi \right\|_{L^4}^2+\left\| \nabla\omega \right\|_{L^4}^2+  \left\|H\right\|_{H^0_{\e}}^2+\left\| \fir \right\|_{L^4}^2+\left\| \omegar \right\|_{L^4}^2
    \\& \lesssim \e^2,
\end{align*}
where we used Hölder's inequality, \eqref{useful gamma 1} (for the three last terms) and the support property of $\ffi$, $\fir$, $\omega$ and $\omegar$. In conclusion, thanks to Corollary \ref{mcowens 2}, if $\e$ is small enough, $\phi$ is indeed an application from $[0,\e]\times B_{H^2_{\delta}}(0,\e)$ to itself and we can prove in the same way that this is a contraction.
\par\leavevmode\par
We can improve the regularity of $\Tilde{\gamma}$, using \eqref{useful gamma 1} and \eqref{useful ffi 1} :
\begin{align*}
    \left\| \Tilde{\gamma}\right\|_{H^4_{\delta}}&\lesssim \left\| e^{-2\gamma}H^2 \right\|_{ H^2_{\delta+2}} + \left\| e^{-2\gamma}\fir^2\right\|_{H^2}+ + \left\| \vert \nabla \ffi \vert^2\right\|_{ H^2}
    \\& \lesssim \left\| H^2 \right\|_{H^2_{\delta+2\e+2}}+\left\| \fir^2\right\|_{H^2}+\left\| \vert \nabla \ffi \vert^2\right\|_{ H^2}
    \\& \lesssim \|H\|_{H^3_{\delta+1}}^2+C_{high}^2,
\end{align*}
where in the last inequality, we used the product estimate (with $\e$ small enough) for the first term and the algebraic structure of $H^2$ for the remaining terms. Thanks to Lemma \ref{CI sur H}, the final quantity is finite, which concludes the proof.

\end{proof}

\subsection{Initial data to the reduced system}

The equations satisfied by $N$ and $\beta$ are :
\begin{align}
     \Delta N & =e^{-2\gamma}N\left(\vert H\vert^2+\fir^2+\frac{1}{4}e^{-4\ffi}\omegar^2\right) , \label{CI sur N}\\
   L\beta & =2e^{-2\gamma}NH.\label{CI sur beta}
\end{align}
The equation \eqref{CI sur N} comes from \eqref{EQ N} in the case $\tau=0$, and the equation \eqref{CI sur beta} comes from \eqref{appendix beta}.
\begin{lem}\label{lem CI N}
For $\e$ sufficiently small, the equation \eqref{CI sur N} admits a unique solution $N=1+N_a\chi\ln+\Tilde{N}$ with $\Tilde{N}\in H^4_{\delta}$, $\Vert \Tilde{N}\Vert_{H^2_{\delta}}\lesssim \e^2$ and $0\leq N_a\lesssim \e^2$.
\end{lem}

\begin{proof}
We look for a solution of the form $N=1+N_a\chi(r)\ln(r)+\Tilde{N}$, with $N_a\geq 0$. On the space $[0,\e]\times B_{H^2_{\delta}}(0,\e)$, we define the application $\phi(N_a^{(1)},\Tilde{N}^{(1)})=(N_a^{(2)},\Tilde{N}^{(2)})$ where (with the notation $N^{(i)}=1+N_a^{(i)}\chi(r)\ln(r)+\Tilde{N}^{(i)}$), $N^{(2)}$ is the solution of
\begin{equation}
    \Delta N^{(2)}=e^{-2\gamma}N^{(1)}(\vert H\vert^2+\fir^2+\frac{1}{4}e^{-4\ffi}\omegar^2).
\end{equation}
Let's show that the RHS is in $H^0_{\delta+2}$.
Thanks to the support property of $\fir$ and $\omegar$, the first term is handled quite easily using \eqref{useful gamma 1}, \eqref{useful ffi 1} and the fact that $\left\|N^{(1)} \right\|_{L^{\infty}(B_R)}\lesssim 1$ (note the embedding $H^2_\delta\xhookrightarrow{}L^{\infty}$) :
\begin{equation*}
    \left\| e^{-2\gamma}N^{(1)}\left(\fir^2+\frac{1}{4}e^{-4\ffi}\omegar^2\right)\right\|_{H^0_{\delta+2}}\lesssim \left\|N^{(1)} \right\|_{L^{\infty}(B_R)}\left(\left\|\fir^2\right\|_{L^2}+\l\omegar^2\r_{L^2}\right)\lesssim\e^2.
\end{equation*}
Using again \eqref{useful gamma 1}, the fact that $\vert \chi\ln\vert\lesssim\langle x\rangle^{\frac{\delta+1}{2}} $, the embedding $H^2_{\delta}\xhookrightarrow{}L^{\infty}$ (used for $\Tilde{N}^{(1)}$) and the product estimate, we handle the second term :
\begin{align*}
    \left\| e^{-2\gamma}N^{(1)}|H|^2\right\|_{H^0_{\delta+2}} & \lesssim\left\| N^{(1)}|H|^2\right\|_{H^0_{\delta+2+2\e}}\nonumber\\ & \lesssim \left\| |H|^2\right\|_{H^0_{\delta+2+2\e}}\left(1+\left\|\Tilde{N}^{(1)}\right\|_{H^2_{\delta}}\right)+\e\left\|  |H|^2\right\|_{H^0_{\delta+2+2\e+\frac{\delta+1}{2}}}
    \\&\lesssim(1+\e)\left\| H\right\|_{H^1_{\delta+1}}^2
    \\&\lesssim \e^4.
\end{align*}
We showed that, for $\e$ small enough, we have $\Vert\Delta N^{(2)}\Vert_{H^0_{\delta+2}}\lesssim \e^2$.
\\
We have :
\begin{equation}
    2\pi N_a^{(2)}=\int_{\R^2}e^{-2\gamma}N^{(1)}|H|^2+\int_{\R^2}e^{-2\gamma}N^{(1)}\fir^2+\frac{1}{4}\int_{\R^2}e^{-2\gamma-4\ffi}N^{(1)}\omegar^2.
\end{equation}
If $\e$ is small enough, we have $N^{(1)}\geq 0$ (using the embedding $H^2_{\delta}\xhookrightarrow{}L^{\infty}$) so that $N_a^{(2)}\geq 0$. With the same kind of arguments than previously, we can easily show that $N_a^{(2)}\lesssim\e^2$. 
\\
This concludes the fact that $\phi$ is well defined (providing $\e$ is small and thanks to Corollary \ref{mcowens 2}), and that this is a contraction (the calculations are likewise, since the equation is linear). 
\par\leavevmode\par
We can improve the regularity of $\Tilde{N}$, using \eqref{useful gamma 1} and \eqref{useful ffi 1} :
\begin{align}
    \left\|\Tilde{N} \right\|_{H^4_{\delta}}&\lesssim \left\|e^{-2\gamma}N|H|^2 \right\|_{H^2_{\delta+2}}+ \left\|e^{-2\gamma}N\fir^2 \right\|_{H^2}+ \left\|e^{-2\gamma-4\ffi}N\omegar^2 \right\|_{H^2} \label{idem1}
    \\& \lesssim \left\| |H|^2\right\|_{H^2_{\delta+2+2\e}}\left(1+\left\|\Tilde{N}^{(1)}\right\|_{H^2_{\delta}}\right)+\e\left\| \chi\ln |H|^2\right\|_{H^2_{\delta+2+2\e}}
    + \left\|N\right\|_{L^{\infty}(B_R)}\left(\left\|    \fir^2 \right\|_{H^2}+\l\omegar^2\r_{H^2}\right).\nonumber 
\end{align}
Using $\vert \chi\ln\vert\lesssim\langle x\rangle^{\frac{\delta+1}{2}} $ and $\left| \nabla^a(\chi\ln)\right|\lesssim\langle x\rangle^{-|a|}$ (for $a\neq0$), we easily show that $\left\| \chi\ln |H|^2\right\|_{H^2_{\delta+2+2\e}}\lesssim \left\|  |H|^2\right\|_{H^2_{\delta+2+2\e+\frac{\delta+1}{2}}}$ to obtain :
\begin{equation*}
    \left\| \Tilde{N}\right\|_{H^4_{\delta}}\lesssim (1+\e)\left\| H\right\|_{H^3_{\delta+1}}^2+\|\fir\|_{H^2}^2+\|\omegar\|_{H^2}^2,
\end{equation*}
which is finite, thanks to Lemma \ref{CI sur H}.
\end{proof}

The following simple lemma will be useful in order to use Theorem \ref{mcowens 1} for $\beta$ :

\begin{lem}\label{divergence nulle}
Let $m\in\N$, $\nu\in\R$ and $u=(u_1,u_2)$ be a fonction from $\R^2$ to $\R^2$ such that $u_i\in H^m_{\nu}$. If $m\geq 2$ and $\nu>0$, then 
\begin{equation*}
    \int_{\R^2}\dive(u)=0
\end{equation*}
\end{lem}

\begin{proof}
We fix $R>0$ and use the Stokes formula :
\begin{align*}
    \bigg| \int_{B_R}\dive(u)\bigg| = \bigg| \int_{\dr B_R}u.n\,\d\sigma\bigg| \leq  \int_{\dr B_R}\langle x\rangle^{-\nu-1}\langle x\rangle^{\nu+1}|u.n|\,\d\sigma \lesssim \Vert u\Vert_{C^0_{\nu+1}}R^{-\nu}
\end{align*}
If $m\geq 2$ and $\nu>0$ we have the Sobolev embeddings $H^m_{\nu}\subset C^0_{\nu+1}$, which concludes the proof since the last inequality implies 
\begin{equation*}
    \lim_{R\to +\infty}\int_{B_R}\dive(u)=0.
\end{equation*}
\end{proof}

\begin{lem}\label{lem CI beta}
For $\e$ sufficiently small, the equation \eqref{CI sur beta} admits a unique solution $\beta\in H^4_{\delta'}$ with $\Vert \beta\Vert_{H^2_{\delta'}}\lesssim\e^2$.
\end{lem}

\begin{proof}
We take the divergence of \eqref{CI sur beta} to obtain the following elliptic equation :
\begin{equation}
    \Delta\beta_j=\dr^i(2Ne^{-2\gamma}H_{ij})\label{laplacien beta CI}
\end{equation}
Thanks to Lemma \ref{divergence nulle}, $\int_{\R^2}\dr^i(2Ne^{-2\gamma}H_{ij})=0$ (the fact that $e^{-2\gamma}NH\in H^2_{\delta'+1}$ will be proved in the sequel of this proof). Thus, in order to apply Theorem \ref{mcowens 1}, it remains to show that $\Vert \dr^i(2Ne^{-2\gamma}H_{ij})\Vert_{H^0_{\delta'+2}}\lesssim \e^2$. For that, we use \eqref{useful gamma 1}, $\e|\chi\ln|\lesssim\langle x\rangle^{\frac{\e}{2}}$, Lemmas \ref{CI sur H} and \ref{lem CI N} :
\begin{align*}
    \left\| \dr^i(2Ne^{-2\gamma}H_{ij})\right\|_{H^0_{\delta'+2}} &\lesssim \left\| e^{-2\gamma}NH\right\|_{H^1_{\delta'+1}}\\&\lesssim \left\| H\right\|_{H^1_{\delta+1}}\left(1+\left\|\Tilde{N} \right\|_{H^2_{\delta}} \right)+\left\|  H\right\|_{H^1_{\delta'+1+C\e^2+\frac{\e}{2}}}
    \\& \lesssim \e^2,
\end{align*}
where in the last inequality, we take $\e$ such that $C\e^2\leq\frac{\e}{2}$. Thus, we can apply Theorem \ref{mcowens 1} to obtain the existence of a solution to \eqref{laplacien beta CI}. We can improve the regularity of this solution using \eqref{useful gamma 2} :
\begin{align*}
    \left\|\beta \right\|_{H^4_{\delta'}}&\lesssim \left\| e^{-2\gamma}NH\right\|_{H^3_{\delta'+1}}
    \\& \lesssim\left\|NH \right\|_{H^3_{\delta'+C\e^2+1}}+\left\|\nabla\Tilde{\gamma} \right\|_{H^2_{\delta'+1}}\left\| NH\right\|_{H^2_{\delta'+C\e^2+1}}
    \\& \lesssim\left(1+\left\|\nabla\Tilde{\gamma} \right\|_{H^2_{\delta'+1}} \right)\left(\left\| H\right\|_{H^3_{\delta+1}}\left(1+\left\|\Tilde{N} \right\|_{H^4_{\delta}} \right)+\left\|  H\right\|_{H^3_{\delta'+1+C\e^2+\frac{\e}{2}}}\right).
\end{align*}
Taking $\e$ such that $C\e^2\leq\frac{\e}{2}$, we conclude using  Lemmas \ref{CI sur H}, \ref{lem CI gamma} and \ref{lem CI N} that $\left\|\beta \right\|_{H^4_{\delta'}}<\infty$.
\par\leavevmode\par
It remains to show that our solution $\beta$ satisfies $L\beta=2Ne^{-2\gamma}H$. We have shown that $L\beta-2Ne^{-2\gamma}H$ is a covariant symmetric traceless divergence free 2-tensor, it implies that its components are harmonic, and thus vanishes (because they belong to $H^4_{\delta'}$). We use the same argument to show that the solution is unique.
\end{proof}

In order to have $\tau_{|\Sigma_0}=0$, we must have the following :

\begin{lem}
We set $e_0\gamma=\frac{1}{2}\dive(\beta)$. Then, we have $e_0\gamma\in H^3_{\delta'+1}$ and $\Vert e_0\gamma\Vert_{H^1_{\delta'+1}}\leq\e^2$.
\end{lem}

\begin{proof}
It follows directly from the estimates on $\beta$ proved in Lemma \ref{lem CI beta} and from Lemma \ref{B1}.
\end{proof}

We summarise in the next corollary our results about the constraints equations and the initial data :

\begin{coro}\label{coro premiere section}
For $\e$ sufficiently small depending only on $\delta$, given a free initial data set, there exists an initial data set to the reduced system such that the constraints equations are satisfied and $\tau_{|\Sigma_0}=0$ . Moreover, we have the following estimates :
\begin{itemize}[label=\textbullet]
    \item there exists $C>0$ depending only on $\delta$ and $R$ such that :
    \begin{equation}
        \Vert H\Vert_{H^1_{\delta+1}}+\vert\alpha\vert+\l\Tilde{\gamma}\r_{H^2_{\delta}}+\Vert e_0\gamma\Vert_{H^1_{\delta'+1}}+\vert N_a\vert+\l\Tilde{N}\r_{H^2_{\delta}}+\Vert \beta\Vert_{H^2_{\delta}}\leq C\e^2\label{CI petit}
    \end{equation}
    \item there exists $C_i>0$ depending on $\delta$, $R$ and $C_{high}$ such that :
    \begin{equation}
        \Vert H\Vert_{ H^3_{\delta+1}}+\l\Tilde{\gamma}\r_{ H^4_{\delta}}+\Vert e_0\gamma\Vert_{H^3_{\delta'+1}}+\l\Tilde{N}\r_{ H^4_{\delta}}+\Vert\beta\Vert_{ H^4_{\delta'}}\leq C_i\label{CI gros}
    \end{equation}
\end{itemize}
\end{coro}

\section{Solving the reduced system}\label{section solving the reduced system}

In this section, we solve the reduced system of equations introduced in Section \ref{section reduced system} by an iteration methode. We first prove that we can construct a sequence, defined in Section \ref{section iteration scheme} and bounded in a small space (this is done in Section \ref{uniforme boundedness}). Then we prove in Section \ref{Cauchy} that the sequence is Cauchy in a larger space, which will imply the existence and uniqueness of solutions to the reduced system of equations.

\subsection{Iteration scheme}\label{section iteration scheme}

In order to solve the reduced system \eqref{EQ N}-\eqref{EQ ffi}, we construct the sequence 
\begin{equation*}
    (N^{(n)}=1+N_a^{(n)}\chi\ln+\Tilde{N}^{(n)},\tau^{(n)},H^{(n)},\beta^{(n)}, \gamma^{(n)}=-\alpha\chi\ln+\Tilde{\gamma}^{(n)},\ffi^{(n)},\omega^{(n)})
\end{equation*}
iteratively as follows : for $n=1,2$, let $N^{(n)},\tau^{(n)},H^{(n)},\beta^{(n)}, \gamma^{(n)},\ffi^{(n)}$ be time-independent, with initial data as in Section \ref{initial data}. For $n\geq 2$, given the $n$-th iterate, the $(n+1)$-st iterate is then defined by solving the following system with initial data as in Section \ref{initial data} :
\begin{align}
    \Delta N^{(n+1)}&=e^{-2\gamma^{(n)}}N^{(n)}\left| H^{(n)}\right|^2+\frac{\left(\tau^{(n)}\right)^2}{2}e^{2\gamma^{(n)}}N^{(n)}\nonumber\\&\qquad\quad\quad+\frac{2e^{2\gamma^{(n)}}}{N^{(n)}}\left(e_0^{(n-1)}\ffi^{(n)}\right)^2+\frac{e^{2\gamma^{(n)}-4\ffi^{(n)}}}{2N^{(n)}}\left(e_0^{(n-1)}\omega^{(n)}\right)^2 \label{reduced system N}\\
    L\beta^{(n+1)}&=2e^{-2\gamma^{(n)}}N^{(n)}H^{(n)} \label{reduced system beta}\\
    \tau^{(n+1)}&=-2\Ll^{(n-1)}\gamma^{(n)}+\frac{\dive\left(\beta^{(n)}\right)}{N^{(n-1)}} \label{reduced system tau}\\
    e_0^{(n+1)}\left(H^{(n+1)}\right)_{ij}&=-2e^{-2\gamma^{(n)}}N^{(n)}\left(H^{(n)}\right)_i^{\;\,\ell}\left(H^{(n)}\right)_{j\ell}+\dr_{(j}\left(\beta^{(n)}\right)^k\left(H^{(n)}\right)_{i)k}\nonumber\\&\qquad-\frac{1}{2}\left(\dr_i\tb\dr_j\right)N^{(n)}+\left(\delta_i^k\tb\dr_j\gamma^{(n)}\right)\dr_kN^{(n)}\label{reduced system H}\\&\qquad-\left(\dr_i\ffi^{(n)}\tb\dr_j\ffi^{(n)}\right)N^{(n)}-\frac{1}{4}e^{-4\ffi^{(n)}}\left(\dr_i\omega^{(n)}\tb\dr_j\omega^{(n)}\right)N^{(n)} \nonumber\\
    \left(\Ll^{(n)}\right)^2\gamma^{(n+1)}-e^{-2\gamma^{(n)}}\Delta \gamma^{(n+1)}&=-\frac{\left(\tau^{(n)}\right)^2}{2}+\frac{1}{2N^{(n)}}e_0^{(n-1)}\left(\frac{\dive\left(\beta^{(n)}\right)}{N^{(n-1)}}\right)\nonumber\\&\qquad\quad\quad+e^{-2\gamma^{(n)}}\left(\frac{\Delta N^{(n)}}{2N^{(n)}}+\left|\nabla\ffi^{(n)}\right|^2+\frac{1}{4}e^{-4\ffi^{(n)}}\left|\nabla\omega^{(n)} \right|^2\right) \label{reduced system gamma}\\
    \left(\Ll^{(n)}\right)^2\ffi^{(n+1)}-e^{-2\gamma^{(n)}}\Delta \ffi^{(n+1)}&=\frac{e^{-2\gamma^{(n)}}}{N^{(n)}}\nabla \ffi^{(n)}\cdot\nabla N^{(n)}+\frac{\tau^{(n)} e_0^{(n-1)}\ffi^{(n)}}{N^{(n)}}\nonumber\\&\qquad\qquad\qquad +\frac{1}{2}e^{-4\ffi^{(n)}}\left( \left(e_0^{(n-1)}\omega^{(n)}\right)^2+\left|\nabla\omega^{(n)}\right|^2 \right)\label{reduced system fi}\\
    \left(\Ll^{(n)}\right)^2\omega^{(n+1)}-e^{-2\gamma^{(n)}}\Delta \omega^{(n+1)}&=\frac{e^{-2\gamma^{(n)}}}{N^{(n)}}\nabla \omega^{(n)}\cdot\nabla N^{(n)}+\frac{\tau^{(n)} e_0^{(n-1)}\omega^{(n)}}{N^{(n)}}\nonumber\\&\qquad\qquad\qquad-4e_0^{(n-1)}\omega^{(n)} e_0^{(n-1)}\ffi^{(n)}-4\nabla\omega^{(n)}\cdot\nabla\ffi^{(n)},\label{reduced system omega}
\end{align}

This system is not a linear system in the $(n+1)$-th iterate, because of the term $e_0^{(n+1)}H^{(n+1)}$ in \eqref{reduced system H} (which contains $\beta^{(n+1)}\cdot\nabla H^{(n+1)}$). The local well-posedness of this system follows from the estimates we are about to prove. Note that we use the following notation :
\begin{equation*}
e_0^{(k)} = \dr_t - \nabla \beta^{(k)}\cdot \nabla \quad \text{and}\quad \Ll^{(k)}=\frac{e_0^{(k)}}{N^{(k)}}.
\end{equation*}

\subsection{Boundedness of the sequence}\label{uniforme boundedness}

The first step is to show that the sequence is uniformly bounded in appropriate function spaces. We proceed by strong induction and suppose that the following estimates hold for all $k$ up to some $n\geq 2$ and for all $t\in [0,T]$. Here, $A_0\ll A_1\ll A_2\ll A_3\ll A_4$ are all sufficiently large constants independent of $\e$ or $n$ to be choosen later. We also set $\delta'=\delta-\e$ and take $\e$ small enough so that $-1<\delta'$. We also choose $\lambda>0$ a small constant such that $\lambda<\delta+1$.

\begin{itemize}[label=\textbullet]
    \item $N^{(k)}$ is of the form $N^{(k)}=1+N_a^{(k)}\chi\ln+\Tilde{N}^{(k)}$ with $N_a^{(k)}\geq 0$ and
    \begin{align}
        \left| N_a^{(k)}\right|+\left\|\Tilde{N}^{(k)}\right\|_{H^2_{\delta}}&\leq \e\label{HR N 1},\\
        \left|\dr_tN_a^{(k)}\right|+\left\|\Tilde{N}^{(k)}\right\|_{H^3_{\delta}}+\left\|\dr_t\Tilde{N}^{(k)}\right\|_{H^2_{\delta}}&\leq 2C_i\label{HR N 2},
        \\ \left\| \Tilde{N}^{(k)}\right\|_{H^4_{\delta}} & \leq A_2 C_i^2\label{HR N 3}.
    \end{align}
    \item $\beta^{(k)}$ satisfies 
    \begin{align}
        \left\|\beta^{(k)}\right\|_{H^2_{\mathbf{\delta'}}}&\leq \e\label{HR beta 1},\\
        \left\|\beta^{(k)}\right\|_{H^3_{\mathbf{\delta'}}}&\leq A_0 C_i\label{HR beta 2},\\
        \left\| \nabla e_0^{(k-1)}\beta^{(k)}\right\|_{L^2_{\delta'+1}}&\leq C_i\label{HR beta 2.5},\\
        \left\| e_0^{(k-1)}\beta^{(k)}\right\|_{H^2_{\delta'}}&\leq A_1C_i\label{HR beta 3},\\
        \left\| e_0^{(k-1)}\beta^{(k)}\right\|_{H^3_{\delta'}}&\leq A_4C_i^2.\label{HR beta 4}
    \end{align}
    \item $H^{(k)}$ satisfies
    \begin{align}
        \left\| H^{(k)}\right\|_{H^2_{\delta+1}}&\leq 2C_i\label{HR H 1},\\
        \left\| e_0^{(k)}H^{(k)}\right\|_{L^2_{1+\lambda}}&\leq \e,\label{HR H 1.5}\\
        \left\| e_0^{(k)}H^{(k)}\right\|_{H^1_{\delta+1}}&\leq A_0C_i,\label{HR H 2}\\
        \left\| e_0^{(k)}H^{(k)}\right\|_{H^2_{\delta+1}}&\leq A_3C_i^2.\label{HR H 3}
    \end{align}
    \item $\tau^{(k)}$ satisfies
    \begin{align}
        \left\| \tau^{(k)}\right\|_{H^2_{\delta'+1}}&\leq A_1C_i\label{HR tau 1},\\
        \left\| \dr_t\tau^{(k)}\right\|_{L^2_{\delta'+1}}&\leq A_2C_i\label{HR tau 2},\\
        \left\| \dr_t\tau^{(k)}\right\|_{H^1_{\delta'+1}}&\leq A_3C_i.\label{HR tau 3}
    \end{align}
    
    \item $\gamma^{(k)}$ is of the form $\gamma^{(k)}=-\alpha\chi\ln+\Tilde{\gamma}^{(k)}$ with $\alpha$ as previously and $\Tilde{\gamma}^{(k)}$ satisfies
    \begin{align}
        \sum_{|\alpha|\leq 2}  \left\Vert\Ll^{(k-1)}\nabla^{\alpha}\Tilde{\gamma}^{(k)}\right\Vert_{L^2_{\delta'+1+|\alpha|}}+\left\| \nabla \Tilde{\gamma}^{(k)}\right\|_{H^2_{\delta'+1}}&\leq 8 C_i,\label{HR gamma 1}
        \\ \left\Vert\dr_t\left(\Ll^{(k-1)}\Tilde{\gamma}^{(k)}\right)\right\Vert_{L^2_{\delta'+1}}&\leq A_0C_i,\label{HR gamma 2}
        \\ \left\Vert\dr_t\left(\Ll^{(k-1)}\Tilde{\gamma}^{(k)}\right)\right\Vert_{H^1_{\delta'+1}}&\leq A_2C_i.\label{HR gamma 3}
    \end{align}
    \item $\ffi^{(k)}$ and $\omega^{(k)}$ are compactly supported in 
    \begin{equation*}
        \enstq{(t,x)\in[0,T]\times\R^2}{\vert x\vert\leq R+C_s(1+R^{\e})t},
    \end{equation*}
    where $C_s>0$ is to be choosen in Lemma \ref{support fi n+1}. Choosing $T$ smaller if necessary, we assume that the above set is a subset of $[0,T]\times B_{2R}$. Moreover, the following estimates hold :
    \begin{align}
        \left\|\dr_t\ffi^{(k)}\right\|_{H^2}+\left\|\nabla\ffi^{(k)}\right\|_{H^2}+\left\Vert\dr_t\left(\Ll^{(k-1)}\ffi^{(k)}\right)\right\Vert_{H^1}&\leq A_0C_i,\label{HR fi 1}\\
        \left\|\dr_t\omega^{(k)}\right\|_{H^2}+\left\|\nabla\omega^{(k)}\right\|_{H^2}+\left\Vert\dr_t\left(\Ll^{(k-1)}\omega^{(k)}\right)\right\Vert_{H^1}&\leq A_0C_i,\label{HR omega 1}
    \end{align}
\end{itemize}

Recalling the statement of Theorem \ref{theoreme principal}, $C_{high}$ is a potentially large constant on which $T$ can depend, but $\e_{0}$ has to be independent of $C_{high}$ and $C_i$ (which, as explained in Corollary \ref{coro premiere section}, depends on $C_{high}$). Therefore, in the following estimates, we will keep trace of $C_i$, and $\e C_i$ is not a small constant.

We will use the symbol $\lesssim$ where the implicit constants are independent of $A_0$, $A_1$, $A_2$, $A_3$, $A_4$ and $C_i$ and use $C$ as the notation for such a constant. Moreover, $C(A_i)$ will denote a constant depending on $A_i$, but not on $C_i$. At the end of the proof, we will choose $A_0$, $A_1$, $A_2$, $A_3$ and $A_4$ such that $C(A_i)\ll A_{i+1}$ for all $i=0,\dots,3$.
\par\leavevmode\par
Our goal now is to prove that all this estimates are still true for the next iterate. For most of these, we will in fact show that they hold with better constants on the RHS.

\subsubsection{Preliminary estimates}

The next result will be very useful in the sequel :

\begin{prop}\label{commutation estimate}
The following estimate holds :
\begin{equation*}
    \left\|\Ll^{(n-1)}\Tilde{\gamma}^{(n)}\right\|_{H^2_{\delta'+1}}\leq 9C_i.\label{commutation estimate eq}
\end{equation*}
\end{prop}

\begin{proof}
In view of \eqref{HR gamma 1}, we have to commute $\Ll^{(n-1)}$ with $\nabla^{\alpha}$ (for $|\alpha|\leq 2$). Indeed :
\begin{equation*}
    \left\|\Ll^{(n-1)}\Tilde{\gamma}^{(n)}\right\|_{H^2_{\delta'+1}}\leq \sum_{|\alpha|\leq 2}\left( \left\Vert\Ll^{(n-1)}\nabla^{(\alpha)}\Tilde{\gamma}^{(n)}\right\Vert_{L^2_{\delta'+1+|\alpha|}}+\left\|\left[ \Ll^{(n-1)},\nabla^{\alpha}\right]\Tilde{\gamma}^{(n)}\right\|_{L^2_{\delta'+1+|\alpha|}}\right)
\end{equation*}
Using the commutation formula $[e_0^{(n-1)},\nabla]=\nabla\beta^{(n-1)}\nabla$, we compute
\begin{equation*}
    \left[\Ll^{(n-1)},\nabla\right] \Tilde{\gamma}^{(n)} =\frac{\nabla\beta^{(n-1)}}{N^{(n-1)}}\nabla \Tilde{\gamma}^{(n)}-\frac{\nabla N^{(n-1)}}{N^{(n-1)}}\Ll^{(n-1)}\Tilde{\gamma}^{(n)}
\end{equation*}
We need smallness for the metric component so we use on one hand \eqref{HR beta 1} the product estimate, the fact that $|\frac{1}{N^{(n-1)}}|\lesssim 1$ and on the other hand the fact that $|\nabla(\chi\ln)|\lesssim \langle x\rangle^{-1}$ and \eqref{HR N 1} to write
\begin{align*}
    \left\|\left[ \Ll^{(n-1)},\nabla\right]\Tilde{\gamma}^{(n)}\right\|_{L^2_{\delta'+2}}&\lesssim \left\|\nabla\beta^{(n-1)}\right\|_{H^1_{\delta'+1}}\left\|\nabla\Tilde{\gamma}^{(n)}\right\|_{H^1_{\delta'+1}}+\left|N_a^{(n-1)}\right|\left\|\Ll^{(n-1)}\Tilde{\gamma}^{(n)}\right\|_{L^2_{\delta'+1}} \\ & \quad+ \left\|\nabla\Tilde{N}^{(n-1)}\right\|_{H^1_{\delta+1}}\left\|\Ll^{(n-1)}\Tilde{\gamma}^{(n)}\right\|_{H^1_{\delta'+1}}\\
    &\lesssim \e\left(\left\|\nabla\Tilde{\gamma}^{(n)}\right\|_{H^2_{\delta'+1}}+  \left\|\Ll^{(n-1)}\Tilde{\gamma}^{(n)}\right\|_{L^2_{\delta'+1}}\right)+\e\left\|\Ll^{(n-1)}\Tilde{\gamma}^{(n)}\right\|_{H^2_{\delta'+1}}.
\end{align*}
Now we compute $\left[\Ll^{(n-1)},\nabla^2\right] \Tilde{\gamma}^{(n)}$ :
\begin{align*}
    \left[\Ll^{(n-1)},\nabla^2\right] \Tilde{\gamma}^{(n)}&=2\frac{\nabla\beta^{(n-1)}}{N^{(n-1)}}\nabla^2 \Tilde{\gamma}^{(n)}-2\frac{\nabla N^{(n-1)}}{N^{(n-1)}}\Ll^{(n-1)}\nabla\Tilde{\gamma}^{(n)} +\left( \frac{\nabla^2\beta^{(n-1)}}{N^{(n-1)}}+\frac{\nabla N^{(n-1)}\nabla\beta^{(n-1)}}{\left(N^{(n-1)}\right)^2}\right)\nabla \Tilde{\gamma}^{(n)}\\
    &-\left(\frac{\nabla^2 N^{(n-1)}}{N^{(n-1)}}+\left(\frac{\nabla N^{(n-1)}}{N^{(n-1)}}  \right)^2 \right)\Ll^{(n-1)}\Tilde{\gamma}^{(n)}.
\end{align*}
Using the product estimate and $|\frac{1}{N^{(n-1)}}|\lesssim 1$ we do the following :
\begin{align*}
    &\left\|\left[ \Ll^{(n-1)},\nabla^2\right]  \Tilde{\gamma}^{(n)}  \right\|_{L^2_{\delta'+3}}&  \\&\lesssim \left\| \nabla^2\Tilde{\gamma}^{(n)} \right\|_{H^1_{\delta'+2}} \left\| \nabla\beta^{(n-1)} \right\|_{H^1_{\delta'+1}} +\left\|N_a^{(n-1)}\nabla(\chi\ln)\Ll^{(n-1)}\nabla\Tilde{\gamma}^{(n)}  \right\|_{L^2_{\delta'+3}} +\left\| \nabla\Tilde{N}^{(n-1)} \right\|_{H^1_{\delta+1}}\left\|\Ll^{(n-1)}\nabla\Tilde{\gamma}^{(n)} \right\|_{H^1_{\delta'+2}} \\
    & \qquad+\left\| \nabla\Tilde{\gamma}^{(n)} \right\|_{H^2_{\delta'+1}}\left(\left\| \nabla^2\beta^{(n-1)} \right\|_{L^2_{\delta'+2}}+\left\| \nabla N^{(n-1)}\nabla\beta^{(n-1)} \right\|_{L^2_{\delta'+2}} \right)\\
    & \qquad+ \left\| \Ll^{(n-1)}\Tilde{\gamma}^{(n)} \right\|_{H^2_{\delta'+1}}\left( \left\| \nabla^2N^{(n-1)} \right\|_{L^2_{\delta+2}} +\left\| \left(\nabla\Tilde{N}^{(n-1)}\right)^2+N_a^{(n-1)}\nabla(\chi\ln)\nabla\Tilde{N}^{(n-1)} \right\|_{L^2_{\delta+2}} \right)\\
    &\qquad+\left\| \left( N_a^{(n-1)}\nabla(\chi\ln)\right)^2\Ll^{(n-1)}\Tilde{\gamma}^{(n)} \right\|_{L^2_{\delta'+3}}.
\end{align*}
Now using \eqref{HR N 1}, \eqref{HR beta 1}, $|\nabla(\chi\ln)|\lesssim \langle x\rangle^{-1}$ we have :
\begin{align*}
    \left\|\left[ \Ll^{(n-1)},\nabla^2\right]  \Tilde{\gamma}^{(n)}  \right\|_{L^2_{\delta'+3}} \lesssim & \;\e\left( \sum_{|\alpha|\leq 2}  \left\Vert\Ll^{(n-1)}\nabla^{\alpha}\Tilde{\gamma}^{(n)}\right\Vert_{L^2_{\delta'+1+|\alpha|}}+\left\|\nabla\Tilde{\gamma}^{(n)}\right\|_{H^2_{\delta'+1}}\right) \\
     & +\e\left\| \Ll^{(n-1)}\Tilde{\gamma}^{(n)} \right\|_{H^2_{\delta'+1}}+\e\left\| \Ll^{(n-1)}\nabla\Tilde{\gamma}^{(n)} \right\|_{H^1_{\delta'+2}}
\end{align*}
It remains to deal with the last term in this last inequality. Using the same type of arguments as above we can show that :
\begin{equation*}
    \left\|\Ll^{(n-1)}\nabla\Tilde{\gamma}^{(n)} \right\|_{H^1_{\delta'+2}}\lesssim   \left\|\Ll^{(n-1)}\nabla\Tilde{\gamma}^{(n)} \right\|_{L^2_{\delta'+2}} + \left\|\Ll^{(n-1)}\nabla^2\Tilde{\gamma}^{(n)} \right\|_{L^2_{\delta'+3}} +\e\left\| \nabla^2\Tilde{\gamma}^{(n)} \right\|_{H^1_{\delta'+2}}
\end{equation*}
Summarising, we get :   
\begin{align*}
    \left\|\Ll^{(n-1)}\Tilde{\gamma}^{(n)}\right\|_{H^2_{\delta'+1}} \lesssim & (1+\e)\left(\sum_{|\alpha|\leq 2}  \left\Vert\Ll^{(n-1)}\nabla^{\alpha}\Tilde{\gamma}^{(n)}\right\Vert_{L^2_{\delta'+1+|\alpha|}}+\left\|\nabla\Tilde{\gamma}^{(n)}\right\|_{H^2_{\delta'+1}}\right)\\&\qquad+\e \left\|\Ll^{(n-1)}\Tilde{\gamma}^{(n)}\right\|_{H^2_{\delta'+1}}
\end{align*}
By choosing $\e$ small enough, we can absorb the last term of the RHS into the LHS and using \eqref{HR gamma 1} we finally prove the desired result.
\end{proof}

We continue with a propagation of smallness result.

\begin{prop}\label{prop propsmall}
The following estimates hold for $T$ sufficiently small and $C_p>0$ a constant depending on $\delta$ and $R$ only :
\begin{align}
    \left\|\dr_t\ffi^{(n)}\right\|_{L^4}+\left\|\nabla\ffi^{(n)}\right\|_{L^4}+\left\|\dr_t\omega^{(n)}\right\|_{L^4}+\left\|\nabla\omega^{(n)}\right\|_{L^4}&\leq C_p\e\label{propsmall fi},\\
    \left\| H^{(n)}\right\|_{H^1_{\delta+1}}&\leq C_p\e^2\label{propsmall H},\\
    \left\| \Tilde{\gamma}^{(n)}\right\|_{H^2_{\delta'}}&\leq C_p \e^2\label{propsmall gamma},\\
    \left\|\Ll^{(n-1)}\Tilde{\gamma}^{(n)} \right\|_{H^1_{\delta'+1}}&\leq C_p\e^2\label{propsmall eo gamma},\\
    \left\|\tau^{(n)}\right\|_{H^1_{\delta'+1}}&\leq C_p\e^2.\label{propsmall tau}
\end{align}
\end{prop}

\begin{proof}
By Corollary \ref{coro premiere section}, all these quantities satisfy the desired smallness estimates at $t=0$.
The fact that these estimates are true for all $t\in[0,T]$ will then follow from calculus inequalities of the type 
\begin{equation*}
    \sup_{s\in[0,T]}\|u\|_{W^{m,p}_{\eta}}(s)\leq C'\left( \|u\|_{W^{m,p}_{\eta}}(0)+\int_0^T\|\dr_tu\|_{W^{m,p}_{\eta}}(s)\d s\right).
\end{equation*}
Therefore, it remains to show that the $\dr_t$ derivatives (we recall that $\dr_t=e_0^{(n)}+\beta^{(n)}\cdot\nabla=e_0^{(n-1)}+\beta^{(n-1)}\cdot\nabla$) of all these terms in the relevant norms are bounded by a constant depending on $A_0$, $A_1$, $A_2$, $A_3$, $A_4$ or  $C_i$, and then to choose $T$ small enough. We proceed as follows :
\begin{itemize}
    \item for $\nabla\ffi^{(n)}$ and $\nabla\omega^{(n)}$, we use the embedding $H^1\xhookrightarrow{}L^4$ and \eqref{HR fi 1} :
    \begin{equation}
        \left\|\dr_t\nabla\ffi^{(n)}\right\|_{L^4}\lesssim \left\|\nabla\dr_t\ffi^{(n)}\right\|_{H^1}\lesssim \left\|\dr_t\ffi^{(n)}\right\|_{H^2}\lesssim A_0C_i,\label{propsmall a}
    \end{equation}
    and we do the same for $\nabla\omega^{(n)}$, using \eqref{HR omega 1}.    
    \item for $\dr_t\ffi^{(n)}$ and $\dr_t\omega^{(n)}$, we use the support property of $\ffi^{(n)}$, the embedding $H^1\xhookrightarrow{}L^4$, \eqref{HR fi 1}, \eqref{HR beta 1}, \eqref{HR beta 3}, \eqref{HR N 2} and \eqref{propsmall a} :
    \begin{align*}
        \left\|\dr_t^2\ffi^{(n)}\right\|_{L^4}&\leq \left\|N^{(n-1)}\dr_t\left(\Ll^{(n-1)}\ffi^{(n)}\right)\right\|_{L^4}+\left\|\Ll^{(n-1)}\ffi^{(n)}\dr_tN^{(n-1)}\right\|_{L^4}+
        \left\|\dr_t\left( \beta^{(n)}\cdot\nabla\ffi^{(n)}\right)\right\|_{L^4}\\
        & \lesssim \left\|\dr_t\left(\Ll^{(n-1)}\ffi^{(n)}\right)\right\|_{H^1}+\left\|\dr_t\ffi^{(n)}\right\|_{H^2}+ \l \dr_t\beta^{(n)} \r_{H^2} \left\|\nabla\ffi^{(n)}\right\|_{H^2}+ \left\|\dr_t\nabla\ffi^{(n)}\right\|_{L^4}\\& \lesssim A_0C_i+ \e A_1C_i,
    \end{align*}
    and we do the same for $\dr_t\omega^{(n)}$, using \eqref{HR omega 1}.
    \item for $H^{(n)}$, we use \eqref{HR H 1}, \eqref{HR H 2}, \eqref{HR beta 1} and the product estimate :
    \begin{equation*}
        \left\|\dr_t H^{(n)} \right\|_{H^1_{\delta+1}}\leq \left\| e_0^{(n)}H^{(n)} \right\|_{H^1_{\delta+1}}+\left\|\beta^{(n)}\nabla H^{(n)} \right\|_{H^1_{\delta+1}}\lesssim C_i.
    \end{equation*}
    \item for $\Tilde{\gamma}^{(n)}$, we use \eqref{commutation estimate eq}, \eqref{HR beta 1} and \eqref{HR gamma 1} :
    \begin{align*}
        \left\| \dr_t\Tilde{\gamma}^{(n)}\right\|_{H^2_{\delta'}}&\leq \left\| N^{(n-1)}\Ll^{(n-1)}\Tilde{\gamma}^{(n)} \right\|_{H^2_{\delta'}}+\left\|\beta^{(n-1)}\cdot\nabla\Tilde{\gamma}^{(n)} \right\|_{H^2_{\delta'}}\\
        & \lesssim\left\| \Ll^{(n-1)}\Tilde{\gamma}^{(n)} \right\|_{H^2_{\delta'+1}}+ \left\|\nabla\Tilde{\gamma}^{(n)} \right\|_{H^2_{\delta'+1}}\left\|\beta^{(n-1)} \right\|_{H^2_{\delta'+1}}\\
        &\lesssim C_i+A_0C_i^2.
    \end{align*}
    \item for $\Ll^{(n-1)}\Tilde{\gamma}^{(n)}$ and $\tau^{(n)}$, we simply use \eqref{HR gamma 3} and \eqref{HR tau 3}, which give directly the result.
\end{itemize}
\end{proof}

\subsubsection{Elliptic estimates}

We begin with the two elliptic equations (the ones for $N$ and $\beta$). These are the most difficult to handle, because we can't rely on the smallness of a time parameter and therefore have to keep properly trace of the $\e$, $C_i$ and $A_i$.

\begin{prop}\label{hr+1 N prop}
For $n\geq 2$, $N^{(n+1)}$ admits a decomposition 
\begin{equation*}
    N^{(n+1)}=1+N_a^{(n+1)}\chi\ln+\Tilde{N}^{(n+1)},
\end{equation*}
with $N_a^{(n+1)}\geq 0$ and such that 
\begin{align}
        \left| N_a^{(n+1)}\right|+\left\|\Tilde{N}^{(n+1)}\right\|_{H^2_{\delta}}&\lesssim \e^2\label{HR+1 N 1},\\
        \left|\dr_tN_a^{(n+1)}\right|+\left\|\Tilde{N}^{(n+1)}\right\|_{H^3_{\delta}}+\left\|\dr_t\Tilde{N}^{(n+1)}\right\|_{H^2_{\delta}}&\lesssim \e C(A_3)C_i\label{HR+1 N 2},
        \\ \left\| \Tilde{N}^{(n+1)}\right\|_{H^4_{\delta}}&\lesssim \e^2C(A_2)C_i^2+C(A_0)C_i^2.\label{HR+1 N 3}
    \end{align}
\end{prop}

\begin{proof}
We claim that :
\begin{equation*}
    \left\| \text{RHS of \eqref{reduced system N}} \right\|_{L^2_{\delta+2}}\leq C\e^2.
\end{equation*}
Except for the term $e^{2\gamma^{(n)}}N^{(n)}\left( \tau^{(n)}\right)^2$, all the terms in \eqref{reduced system N} can be estimated in an identical manner as in Lemma \ref{lem CI N}, except that we estimate the norms using Proposition \ref{prop propsmall} instead of using the assumtions on the reduced data and the estimates in Lemmas \ref{CI sur H} and \ref{lem CI gamma}. It therefore remains to control $e^{2\gamma^{(n)}}N^{(n)}\left( \tau^{(n)}\right)^2$. Using \eqref{useful gamma 1a} and \eqref{HR N 1}, we see that $\left\|e^{2\gamma^{(n)}} N^{(n)}\right\|_{C^0_\e}\lesssim 1$. We finally use the product estimate and \eqref{propsmall tau} to handle $\left( \tau^{(n)}\right)^2$ :
\begin{equation*}
    \left\|e^{2\gamma^{(n)}}N^{(n)}\left( \tau^{(n)}\right)^2 \right\|_{L^2_{\delta+2}} \lesssim \left\| e^{2\gamma^{(n)}}  N^{(n)}\right\|_{C^0_\e} \left\| \tau^{(n)} \right\|_{H^1_{\delta'+1}}^2\lesssim \e^4.
\end{equation*}
This proves the claim. Applying Corollary \ref{mcowens 2} to $N^{(n+1)}-1$ yields the existence of the decomposition of $N^{(n+1)}$, as well as the estimate \eqref{HR+1 N 1}.
\par\leavevmode\par
We now turn to the proof of \eqref{HR+1 N 2}. To obtain the $H^3_{\delta}$ bound for $\Tilde{N}^{(n+1)}$, we need to control the RHS of \eqref{reduced system N} in $H^1_{\delta+2}$ :
\begin{itemize}
    \item for the term $e^{-2\gamma^{(n)}}N^{(n)}|H^{(n)}|^2$, we do exactly the same calculations as in \eqref{idem1}, but in $H^1_{\delta+2}$ instead of $H^2_{\delta+2}$. In contrast to \eqref{idem1}, here we have less liberty to bound the term $|H^{(n)}|^2$ (because we need $C_i$ and not $C_i^2$ bounds), therefore we use \eqref{HR H 1} and \eqref{propsmall H} to write 
    \begin{equation*}
        \l e^{-2\gamma^{(n)}}N^{(n)}\left|H^{(n)}\right|^2\r_{H^1_{\delta+2}} \lesssim\left\| \left|H^{(n)}\right|^2\right\|_{H^1_{\delta+2+2\e+\frac{\delta+1}{2}}}\lesssim \left\|H^{(n)} \right\|_{H^1_{\delta+1}}\left\|H^{(n)} \right\|_{H^2_{\delta+1}}\lesssim \e^2C_i.
    \end{equation*}
    \item for the term $e^{2\gamma^{(n)}}N^{(n)}( \tau^{(n)})^2$, we note that $\tau^{(n)}$ and $H^{(n)}$ satisfy the exact same estimate (according to \eqref{HR H 1}, \eqref{HR tau 1}, \eqref{propsmall H} and \eqref{propsmall tau}), except for a slight difference of weights ($\delta'$ instead of $\delta$) and constants ($A_1$ compared to 2). Therefore we treat this term exactly as the previous one and omit the details.
    \item we now discuss the term $\frac{e^{2\gamma^{(n)}}}{N^{(n)}}\left(e_0^{(n-1)}\ffi^{(n)}\right)^2$. Since the smallness for $e_0^{(n-1)}\ffi^{(n)}$ is at the $L^4$-level (thanks to \eqref{propsmall fi}), any spatial derivative of $e_0^{(n-1)}\ffi^{(n)}$ destroys the $\e$-smallness, and therefore we have to be precise. Thanks to \eqref{useful gamma 1a}, we can forget the $e^{2\gamma^{(n)}}$ factor, thanks to \eqref{HR N 1} we have $\left| \frac{1}{N^{(n)}}\right|\lesssim 1$ (we also forget about $\nabla(\chi\ln)$) and thanks to \eqref{HR N 2} and the embedding $H^2_{\delta+1}\xhookrightarrow{}L^{\infty}$ we have $\left\|\nabla \Tilde{N}^{(n)}\right\|_{L^{\infty}}\lesssim C_i$ :
    \begin{align*}
        \left\|\frac{e^{2\gamma^{(n)}}}{N^{(n)}}\left(e_0^{(n-1)}\ffi^{(n)}\right)^2 \right\|_{H^1_{\delta+2}}&\lesssim \left\|\left( e_0^{(n-1)}\ffi^{(n)}\right)^2 \right\|_{L^2}\left( 1+ \left\|\nabla \Tilde{N}^{(n)}\right\|_{L^{\infty}}\right)
        +\left\|e_0^{(n-1)}\ffi^{(n)}\nabla\left( e_0^{(n-1)}\ffi^{(n)}\right) \right\|_{L^2}\\
        &\lesssim \left\| e_0^{(n-1)}\ffi^{(n)} \right\|_{L^4}^2\left( 1+ \left\|\nabla \Tilde{N}^{(n)}\right\|_{L^{\infty}}\right)+\l e_0^{(n-1)}\ffi^{(n)} \r_{L^4}\l e_0^{(n-1)}\ffi^{(n)} \r_{H^2},
    \end{align*}
    where in the last inequality we used Hölder's inequality and the Sobolev injection $H^1\xhookrightarrow{}L^4$. We now use \eqref{propsmall fi} and \eqref{HR fi 1} to obtain :
    \begin{align*}
        \left\|\frac{e^{2\gamma^{(n)}}}{N^{(n)}}\left(e_0^{(n-1)}\ffi^{(n)}\right)^2 \right\|_{H^1_{\delta+2}}&\lesssim \e^2(1+C_i)+\e A_0C_i\lesssim \e C(A_0)C_i.
    \end{align*}
    \item the term $\frac{e^{2\gamma^{(n)}-4\ffi^{(n)}}}{N^{(n)}}\left(e_0^{(n-1)}\omega^{(n)}\right)^2$ is handled in a similar way, using first \eqref{useful ffi 1} to get ride of the $e^{-4\ffi^{(n)}}$ factor, and then using \eqref{HR omega 1} instead of \eqref{HR fi 1}.
\end{itemize}
This concludes the proof of the estimate $\left\| \Tilde{N}^{(n+1)}\right\|_{H^3_{\delta}}\lesssim \e C(A_1)C_i$. 
\par\leavevmode\par
We now turn to the estimate for $\dr_tN^{(n+1)}$, including both $\dr_tN_a^{(n+1)}$ and $\dr_t\Tilde{N}^{(n+1)}$. Since the RHS of \eqref{reduced system N} is differentiable in $t$, it is easy to see that $\dr_tN^{(n+1)}=\dr_tN_a^{(n+1)}\chi\ln+\dr_t\Tilde{N}^{(n+1)}$ is the solution given by Corollary \ref{mcowens 2} to the equation 
\begin{equation*}
    \Delta f=\dr_t(\text{RHS of \eqref{reduced system N}}).
\end{equation*}
Therefore, to finish the proof of \eqref{HR+1 N 2}, it suffices to bound the integral of $\dr_t(\text{RHS of \eqref{reduced system N}})$ with respect to $\d x$ and to bound $\dr_t(\text{RHS of \eqref{reduced system N}})$ in $L^2_{\delta+2}$.

Since the estimate for $\dr_t\tau^{(n)}$ are worse than those for $\dr_tH^{(n)}$, and those for $\tau^{(n)}$ and $H^{(n)}$ are similar, we will treat the term $\dr_t \left( e^{2\gamma^{(n)}}N^{(n)}( \tau^{(n)})^2\right) $ and leave the easier term  $\dr_t \left( e^{2\gamma^{(n)}}N^{(n)}| H^{(n)}|^2\right) $ to the reader. We use \eqref{useful gamma 1a} for the $e^{2\gamma^{(n)}}$ factor and the fact that $|\chi\ln|\lesssim \langle x\rangle^{\e}$ :
\begin{align*}
    \left\|\dr_t \left( e^{2\gamma^{(n)}}N^{(n)}( \tau^{(n)})^2\right) \right\|_{L^2_{\delta+2}} &\lesssim\left\|N^{(n)} \right\|_{C^0_{\e}} \left(
    \left\|\tau^{(n)}\dr_t\tau^{(n)} \right\|_{L^2_{\delta+2+\e}}+\left\|\dr_t\Tilde{\gamma}^{(n)}( \tau^{(n)})^2 \right\|_{L^2_{\delta+2+\e}}\right)\\&\qquad +\left|\dr_tN_a^{(n)} \right|\left\| ( \tau^{(n)})^2\right\|_{L^2_{\delta+2+3\e}} +\left\| \dr_t\Tilde{N}^{(n)}( \tau^{(n)})^2\right\|_{L^2_{\delta+2+2\e}}
    \\& \lesssim\e^2C(A_3)C_i.
\end{align*}
where in the last inequality we use $\left\|N^{(n)} \right\|_{C^0_{\e}}\lesssim 1$ (which comes from \eqref{HR N 1}), $\left\| \dr_t\Tilde{\gamma}^{(n)}\right\|_{H^2_{\delta'+1}}\lesssim C_i$ (wich comes from \eqref{commutation estimate eq}), \eqref{HR N 2} and the product estimate together with \eqref{propsmall tau} (for $(\tau^{(n)})^2$). For $\tau^{(n)}\dr_t\tau^{(n)}$, we use the Hölder's inequality ($L^4_{\delta'+2}\times L^4_{\delta'+1}\xhookrightarrow{}L^2_{\delta+2+\e}$), the embedding $H^1_{\delta'+1}\xhookrightarrow{}L^4_{\delta'+1}$, \eqref{propsmall tau} and \eqref{HR tau 3}. 

We now turn to the compactly supported term $\dr_t\left(\frac{e^{2\gamma^{(n)}}}{N^{(n)}}\left(e_0^{(n-1)}\ffi^{(n)}\right)^2\right)$. We use \eqref{useful gamma 1a} for the $e^{2\gamma^{(n)}}$ factor and $\left|\frac{1}{N^{(n)}}\right|+\left\| \chi\ln\right\|_{L^{\infty}(B_{2R})}\lesssim 1$ :
\begin{align*}
    \left\| \dr_t\left(\frac{e^{2\gamma^{(n)}}}{N^{(n)}}\left(e_0^{(n-1)}\ffi^{(n)}\right)^2\right)\right\|_{L^2} &\lesssim \left\|e_0^{(n-1)}\ffi^{(n)}\dr_t\left(e_0^{(n-1)}\ffi^{(n)}\right) \right\|_{L^2}+\left\|\dr_t\Tilde{\gamma}^{(n)}\left(e_0^{(n-1)}\ffi^{(n)}\right)^2 \right\|_{L^2}\\&\qquad+\left\|\dr_t\Tilde{N}^{(n)}\left(e_0^{(n-1)}\ffi^{(n)}\right)^2 \right\|_{L^2}+\left|\dr_tN_a^{(n)}\right|\left\|\left(e_0^{(n-1)}\ffi^{(n)}\right)^2 \right\|_{L^2}
    \\&\lesssim \left\|e_0^{(n-1)}\ffi^{(n)} \right\|_{L^4}\left(C_i\left\|e_0^{(n-1)}\ffi^{(n)} \right\|_{L^4}+ \left\|\dr_t\left(e_0^{(n-1)}\ffi^{(n)}\right) \right\|_{H^1}\right)
    \lesssim \e C_i,
\end{align*}
where in the last inequality we use $\left\| \dr_t\Tilde{\gamma}^{(n)}\right\|_{H^2_{\delta'+1}}\lesssim C_i$, \eqref{HR N 2} and the embedding $L^4\times H^1\xhookrightarrow{}L^2$. The term $\dr_t\left(\frac{e^{2\gamma^{(n)}-4\ffi^{(n)}}}{N^{(n)}}\left(e_0^{(n-1)}\omega^{(n)}\right)^2\right)$ is handled in the same way, using \eqref{useful ffi 1} and \eqref{propsmall fi} to get rid of the $e^{-4\ffi^{(n)}}$. Combining all these estimates concludes the proof of \eqref{HR+1 N 2}.
\par\leavevmode\par
We now turn to the proof of \eqref{HR+1 N 3}. To obtain the $H^4_{\delta}$ bound for $\Tilde{N}^{(n+1)}$, we need to control the RHS of \eqref{reduced system N} in $H^2_{\delta+2}$. Since we already know that the RHS of \eqref{reduced system N} is in $H^1_{\delta+2}$ with bound $\e C(A_1)C_i$, it remains to bound the $L^2_{\delta+4}$ norm of the second derivative the RHS of \eqref{reduced system N} :
\begin{itemize}
    \item for the term $e^{-2\gamma^{(n)}}N^{(n)}|H^{(n)}|^2$, we first use \eqref{useful gamma 1}, and then the embedding $H^1_{\delta+1}\xhookrightarrow{}L^4_{\delta+1}$, \eqref{HR H 1}, \eqref{propsmall H}, \eqref{HR N 1}, \eqref{HR N 2} and \eqref{HR N 3} :
    \begin{align*}
        \l  \nabla^2\left(e^{-2\gamma^{(n)}}N^{(n)}|H^{(n)}|^2\right) \r_{L^2_{\delta+4}} &\lesssim \l\nabla^2N^{(n)}(H^{(n)})^2 \r_{L^2_{\delta+4}} + \l \nabla N^{(n)}H^{(n)}\nabla H^{(n)} \r_{L^2_{\delta+4}}  \\&\qquad+\l N^{(n)} H^{(n)} \nabla^2H^{(n)} \r_{L^2_{\delta+4}} +\l N^{(n)} (\nabla H^{(n)})^2 \r_{L^2_{\delta+4}} 
        \\& \lesssim\e^2 C(A_2) C_i^2
    \end{align*}
    where we used $L^\infty$ bounds for $N^{(n)}$, $\nabla N^{(n)}$ and $\nabla^2 N^{(n)}$ (see \eqref{HR N 1}, \eqref{HR N 2} and \eqref{HR N 3} respectively). Note that for $N$ the logarithmic growth is handled by adding a small weight. We also used \eqref{HR H 1} and \eqref{propsmall H} and the product law to handle the $H^{(n)}$ terms.
    \item for the term $e^{2\gamma^{(n)}}N^{(n)}( \tau^{(n)})^2$, we note that $\tau^{(n)}$ and $H^{(n)}$ satisfy the exact same estimate (according to \eqref{HR H 1}, \eqref{HR tau 1}), except for a slight difference of weights ($\delta'$ instead of $\delta$) and constants ($A_1$ compared to 2). Therefore we treat this term exactly as the previous one and omit the details.
    \item we next discuss the compactly supported term $\frac{e^{2\gamma^{(n)}}}{N^{(n)}}\left(e_0^{(n-1)}\ffi^{(n)}\right)^2$. We first use \eqref{useful gamma 1} :
    \begin{align*}
        \left\|\nabla^2\left( \frac{e^{2\gamma^{(n)}}}{N^{(n)}}\left(e_0^{(n-1)}\ffi^{(n)}\right)^2\right) \right\|_{L^2} &\lesssim \left\| e_0^{(n-1)}\ffi^{(n)}\nabla^2\left(e_0^{(n-1)}\ffi^{(n)}\right) \right\|_{L^2}+\left\|\left(\nabla\left( e_0^{(n-1)}\ffi^{(n)}\right)\right)^2 \right\|_{L^2}\\&+\left\|\nabla N^{(n)} e_0^{(n-1)}\ffi^{(n)}\nabla\left(e_0^{(n-1)}\ffi^{(n)}\right)\right\|_{L^2}+\left\|\nabla^2N^{(n)}\left( e_0^{(n-1)}\ffi^{(n)}\right)^2 \right\|_{L^2}\\&+\left\|\left( \nabla N^{(n)}\right)^2\left( e_0^{(n-1)}\ffi^{(n)}\right)^2 \right\|_{L^2} \\& \lesssim\e^2C(A_2)C_i^2+C(A_0)C_i^2,
    \end{align*}
    where we used \eqref{HR N 2}, \eqref{HR fi 1} and \eqref{propsmall fi}. The idea is to use $L^{\infty}$-bounds for $\nabla N^{(n)}$ and $\nabla^2 N^{(n)}$ and the Hölder's inequality to deal with the product of terms depending on $\ffi^{(n)}$.
    \item the term $\frac{e^{2\gamma^{(n)}-4\ffi^{(n)}}}{N^{(n)}}\left(e_0^{(n-1)}\omega^{(n)}\right)^2$ is handled in a similar way, but we have to be careful about the case where two derivatives hit $e^{-4\ffi^{(n)}}$. Using \eqref{useful ffi 1}, \eqref{useful gamma 1a} and $1\lesssim N^{(n)}$, this leads to estimating the following term :
    \begin{align*}
    \l\nabla^2\ffi^{(n)}\left(e_0^{(n-1)}\omega^{(n)}\right)^2 \r_{L^2}\lesssim \l \nabla^2\ffi^{(n)}\r_{L^4}\l e_0^{(n-1)}\omega^{(n)}\r_{L^4}\l e_0^{(n-1)}\omega^{(n)}\r_{L^\infty}\lesssim \e C(A_0)C_i^2,
    \end{align*}
    where we used \eqref{HR fi 1}, \eqref{HR omega 1} and \eqref{propsmall fi}.
  
\end{itemize}
This concludes the proof of \eqref{HR+1 N 3}.
\end{proof}

The following lemma will allow us to estimate the $H^1$ norm of solutions of elliptic equations.

\begin{lem}
Let $\xi=(\xi^1,\xi^2)$ a vector field on $\R^2$, for all $\sigma<1$ the following holds :
\begin{align}
    \l\nabla\xi \r_{L^2_{\sigma}} \lesssim \l L\xi\r_{L^2_{1}}\label{killing operator H1}.
\end{align}
\end{lem}

\begin{proof}
We set $A_{ij}\vcentcolon=(L\xi)_{ij}$ and take the divergence to obtain $\Delta\xi^i=\delta^{ij}\dr^kA_{kj}$. Let $w(x)=\langle x\rangle^{2\sigma}$. We multiply this equation by $w\xi^{\ell}$, contract it with $\delta_{i\ell}$ and integrate over $\R^2$ to get (after integrating by parts) :
\begin{equation*}
    \delta_{i\ell}\int_{\R^2}\nabla(w\xi^{\ell})\cdot\nabla\xi^i\d x=\int_{\R^2}\dr^k(w\xi^{\ell})A_{k\ell}\d x,
\end{equation*}
which becomes 
\begin{equation*}
    \l \nabla\xi\r_{L^2_{\sigma}}^2=\frac{1}{2} \int_{\R^2}\Delta w|\xi|^2\d x + \int_{\R^2}w\dr^k\xi^{\ell}A_{k\ell}\d x +\int_{\R^2}\dr^k w\xi^{\ell}A_{k\ell}\d x .
\end{equation*}
Using the Cauchy-Schwarz inequality and the trick $ab\leq \eta a^2+\frac{1}{\eta} b^2$, we have
\begin{equation*}
    \int_{\R^2}w\dr^k\xi^{\ell}A_{k\ell}\d x\lesssim \l\nabla\xi \r_{L^2_{\sigma}} \l A \r_{L^2_{\sigma}}\lesssim \eta\l\nabla\xi \r_{L^2_{\sigma}}^2+\frac{1}{\eta}\l A \r_{L^2_{\sigma}}^2.
\end{equation*}
We note that $|\nabla w|\lesssim \langle x \rangle^{2\sigma-1}$ and $|\Delta w|\lesssim \langle x \rangle^{2\sigma-2}$, which imply that
\begin{equation*}
    \int_{\R^2}\dr^k w\xi^{\ell}A_{k\ell}\d x \lesssim \l\xi \r_{L^2_{\sigma-1}} \l A \r_{L^2_{\sigma}}\lesssim \l\xi \r_{L^2_{\sigma-1}}^2+\l A \r_{L^2_{\sigma}}^2\quad\text{and}\quad
    \frac{1}{2} \int_{\R^2}\Delta w|\xi|^2\d x \lesssim \l\xi \r_{L^2_{\sigma-1}}^2.
\end{equation*}
Thus,
\begin{equation*}
    \l \nabla\xi\r_{L^2_{\sigma}}^2 \lesssim \l\xi \r_{L^2_{\sigma-1}}^2+\left( 1+\frac{1}{\eta}\right)\l A \r_{L^2_{\sigma}}^2+\eta\l\nabla\xi \r_{L^2_{\sigma}}^2.
\end{equation*}
We take $\eta$ small enough in order to absorb $\eta\l\nabla\xi \r_{L^2_{\sigma}}^2$ into the LHS. Taking the square root of the inequality we obtained, we get :
\begin{equation*}
    \l \nabla\xi\r_{L^2_{\sigma}} \lesssim \l\xi \r_{L^2_{\sigma-1}}+\l A \r_{L^2_{\sigma}}.
\end{equation*}
It remains to show that $\l\xi \r_{L^2_{\sigma-1}}\lesssim \l A\r_{L^2_1}$. For that, we start by using Lemma \ref{prop holder 2} : $\sigma<1$ so there exists $r>2$ such that $\sigma<\frac{2}{r}<1$. According to Lemma \ref{prop holder 2}, we have $\l\xi \r_{L^2_{\sigma-1}}\lesssim \l\xi\r_{L^r}$. Recalling that $\Delta\xi^i=\delta^{ij}\dr^kA_{kj}$ we have :
\begin{equation*}
    \xi^i(x)=\frac{\delta^{ij}}{2\pi}\int_{\R^2}\ln|x-y|\dr^kA_{kj}\d y=-\frac{\delta^{ij}}{2\pi}\int_{\R^2}\frac{y^k-x^k}{|x-y|^2}A_{kj}\d y.\
\end{equation*}
Therefore we can use the Hardy-Littlewood-Sobolev inequality (Proposition \ref{prop HLS}), that
\begin{equation}
    \l\xi \r_{L^r}\lesssim \l A*\frac{1}{|\cdot|} \r_{L^r} \lesssim \l A\r_{L^{\frac{2r}{2+r}}} .
\end{equation}
We again use Lemma \ref{prop holder 2} to get the embedding $L^2_1\xhookrightarrow{}L^{\frac{2r}{2+r}}$ (recall that $r>2$), which conclude the proof of \eqref{killing operator H1}.

\end{proof}

\begin{prop}\label{hr+1 beta prop}
For $n\geq 2$, the following estimates hold :
\begin{align}
        \left\|\beta^{(n+1)}\right\|_{H^2_{\delta'}}&\lesssim \e^2\label{HR+1 beta 1},\\
        \left\|\beta^{(n+1)}\right\|_{H^3_{\delta'}}&\lesssim C_i\label{HR+1 beta 2},\\
        \left\|\nabla e_0^{(n)}\beta^{(n+1)}\right\|_{L^2_{\delta'+1}}&\lesssim \e C_i\label{HR+1 beta 2.5},\\
        \left\| e_0^{(n)}\beta^{(n+1)}\right\|_{H^2_{\delta'}}&\lesssim A_0C_i\label{HR+1 beta 3},\\
        \left\| e_0^{(n)}\beta^{(n+1)}\right\|_{H^3_{\delta'}}&\lesssim A_3C_i^2.\label{HR+1 beta 4}
    \end{align}
\end{prop}

\begin{proof}
In view of Proposition \ref{prop propsmall} and \eqref{HR N 1}, the existence and uniqueness of $\beta^{(n+1)}$ and the estimate \eqref{HR+1 beta 1} can be proven in exactly the same manner as in Lemma \ref{CI sur beta} and we omit the details.
\par\leavevmode\par
We begin by the proof of \eqref{HR+1 beta 2}. We take the divergence of \eqref{reduced system beta} to get
\begin{equation}
    \Delta(\beta^{(n+1)})^i=2\delta^{i\ell}\delta^{jk}\dr_k\left( e^{-2\gamma^{(n)}}N^{(n)}(H^{(n)})_{j\ell} \right).\label{laplacien beta}
\end{equation}
Note that the RHS has 0 mean (by Lemma \ref{divergence nulle}) and therefore by Theorem \ref{mcowens 1}, in order to prove \eqref{HR+1 beta 2}, it suffices to bound the RHS of \eqref{laplacien beta} in $H^1_{\delta'+2}$ by $CC_i$. Using \eqref{useful gamma 1}, \eqref{HR H 1}, \eqref{HR N 1} and $\e\left|\chi\ln \right|\lesssim \langle x \rangle^{\frac{\e}{2}}$ and taking $\e$ small enough, we get :
\begin{align}
    \left\|\dr_k\left( e^{-2\gamma^{(n)}}N^{(n)}(H^{(n)})_{j\ell}\right)\right\|_{H^1_{\delta'+2}}&\lesssim \left\|N^{(n)}H^{(n)}\right\|_{H^2_{\delta'+2\e^2+1}} \label{NnHn}\\
    &\lesssim \left\| H^{(n)}\right\|_{H^2_{\delta+1}}\left( 1+ \left\|\Tilde{N}^{(n)} \right\|_{H^2_{\delta}} \right)+
    \left\|  H^{(n)}\right\|_{H^2_{\delta'+2\e^2+\frac{\e}{2}+1}}\nonumber\\& \lesssim C_i.\nonumber
\end{align}

We now turn to the proof of \eqref{HR+1 beta 2.5}. We have $e_0^{(n)}\beta^{(n+1)}=\dr_t\beta^{(n+1)}-\beta^{(n)}\cdot\nabla\beta^{(n+1)}$. Using \eqref{HR beta 1}, \eqref{HR+1 beta 1} and the product estimate we have
\begin{equation*}
    \l \beta^{(n)}\cdot\nabla\beta^{(n+1)}\r_{H^1_{\delta'}} \lesssim \l\beta^{(n)} \r_{H^2_{\delta'}}\l \beta^{(n+1)}\r_{H^2_{\delta'}} \lesssim \e^2.
\end{equation*}
Applying $\dr_t$ to \eqref{reduced system beta}, $\dr_t\beta^{(n+1)}$ satisfy $(L\dr_t\beta^{(n+1)})_{ij}=2\dr_t(e^{-2\gamma^{(n)}}N^{(n)}(H^{(n)})_{ij})$. We apply \eqref{killing operator H1} with $\sigma =\delta'+1$ and use $|\chi\ln|\lesssim \langle x\rangle^{\eta}$ (where $\eta$ is as small as we want) :
\begin{align*}
    \l\nabla\dr_t\beta^{(n+1)} \r_{L^2_{\delta'+1}} &\lesssim \l\dr_t\left(e^{-2\gamma^{(n)}}N^{(n)}(H^{(n)})_{ij}\right) \r_{L^2_{1}}
    \\&\lesssim \l \dr_t\Tilde{\gamma}^{(n)}H^{(n)}\r_{L^2_{1+2\e^2+\eta}}+\l \dr_t\Tilde{\gamma}^{(n)}\Tilde{N}^{(n)}H^{(n)}\r_{L^2_{1+2\e^2}} + \l\dr_t N_a^{(n)}H^{(n)} \r_{L^2_{1+2\e^2+\eta}} \\&\qquad+ \l \dr_t\Tilde{N}^{(n)}H^{(n)} \r_{L^2_{1+2\e^2}} + \l \dr_tH^{(n)} \r_{L^2_{1+2\e^2+\eta}}+ \l \Tilde{N}^{(n)}\dr_tH^{(n)} \r_{L^2_{1+2\e^2}}
    \\& \lesssim  \e(1+ C_i),
\end{align*}
where used \eqref{HR gamma 1}, \eqref{HR N 1}, \eqref{HR N 2}, \eqref{propsmall H} and \eqref{HR H 1.5} (with $\e$ and $\eta$ small enough, depending on $\lambda$).

We now turn to the proof of \eqref{HR+1 beta 3} and \eqref{HR+1 beta 4}. Applying $e_0^{(n)}$ to \eqref{laplacien beta}, we show that the following equation is satisfied :
\begin{equation}
    \Delta(e_0^{(n)}\beta^{(n+1)})^i=2\delta^{i\ell}\delta^{jk}e_0^{(n)}\dr_k\left( e^{-2\gamma^{(n)}}N^{(n)}(H^{(n)})_{j\ell} \right)+\left[ \Delta,e_0^{(n)} \right](\beta^{(n+1)})^i=\vcentcolon I+II.\label{laplacien e0 beta}
\end{equation}
It's easy to check the RHS of \eqref{laplacien e0 beta} has 0 mean, as a consequence we can apply Theorem \ref{mcowens 1}, so that in order to prove the estimate \eqref{HR+1 beta 3}, it suffices to bound the RHS of \eqref{laplacien e0 beta} in $L^2_{\delta'+2}$ by $C_i$ :
\begin{itemize}
    \item For $I$, we first commute $\nabla$ and $e_0^{(n)}$ :
    \begin{equation}
        | I| \lesssim \left|\nabla e_0^{(n)}\left( e^{-2\gamma^{(n)}}N^{(n)}H^{(n)} \right) \right|+\left| \nabla\beta^{(n)}\right| \left|\nabla\left( e^{-2\gamma^{(n)}}N^{(n)}H^{(n)} \right) \right|.\label{commutation estimate''}
    \end{equation}
    It implies, using \eqref{useful gamma 1} :
    \begin{align}
        \| I\|_{L^2_{\delta'+2}} & \lesssim \left\| (e_0^{(n)}\gamma^{(n)})N^{(n)}H^{(n)} \right\|_{H^1_{\delta'+2\e^2+1}}+\left\|(e_0^{(n)}N^{(n)})H^{(n)} \right\|_{H^1_{\delta'+2\e^2+1}}\label{I L^2}\\&\qquad+\left\|N^{(n)}(e_0^{(n)}H^{(n)}) \right\|_{H^1_{\delta'+2\e^2+1}}+\left\|\nabla\beta^{(n)} \right\|_{L^{\infty}}  \left\|N^{(n)}H^{(n)} \right\|_{H^1_{\delta'+2\e^2+1}}.\nonumber
    \end{align}
    Thanks to \eqref{HR beta 2} and \eqref{propsmall H}, the last term is bounded by $\e A_0 C_i$. Thanks to \eqref{HR H 2}, the third term is handled as $N^{(n)}H^{(n)}$ in \eqref{NnHn} and is therefore bounded by $A_0C_i$. Thanks to \eqref{HR N 1}, \eqref{HR N 2}, \eqref{HR beta 1}, \eqref{HR H 1} and \eqref{propsmall H}, the second term is bounded by $C_i$. The first term is similar to the second, and actually easier to bound, so we omit the details. We have shown that $\| I\|_{L^2_{\delta'+2}}\lesssim A_0C_i$.
    \item For $II$, we use the following commutation estimate : 
    \begin{equation}
        \left| \left[ \Delta,e_0^{(n)} \right](\beta^{(n+1)})^i \right| \lesssim \left|\nabla\beta^{(n)} \right|\left|\nabla^2\beta^{(n+1)} \right|+\left|\nabla^2\beta^{(n)} \right|\left|\nabla\beta^{(n+1)} \right|.\label{commutation estimate'}
    \end{equation}
    Now, using in addition \eqref{HR beta 1}, \eqref{HR beta 2}, \eqref{HR+1 beta 1} and \eqref{HR+1 beta 2} and the product estimate :
    \begin{align*}
        \left\| II\right\|_{L^2_{\delta'+2}} &\lesssim \left\|\nabla\beta^{(n)} \right\|_{H^1_{\delta'+1}}\left\|\nabla^2\beta^{(n+1)} \right\|_{H^1_{\delta'+2}}+\left\| \nabla^2\beta^{(n)}\right\|_{H^1_{\delta'+2}}\left\|\nabla\beta^{(n+1)} \right\|_{H^1_{\delta'+1}} \\&\lesssim\e A_0C_i.
    \end{align*}
\end{itemize}
Similarly, in order to prove \eqref{HR+1 beta 4}, we have to bound the RHS of \eqref{laplacien e0 beta} in $H^1_{\delta'+2}$ by $CC_i^2$ :
\begin{itemize}
    \item For $I$, we again use \eqref{commutation estimate''}. Instead of using $L^{\infty}$ bounds for $\nabla\beta^{(n)}$, we use the product estimate and then \eqref{useful gamma 1}. For the terms where $e_0^{(n)}$ appears, we simply use \eqref{useful gamma 1} :
    \begin{align*}
        \| I\|_{H^1_{\delta'+2}} & \lesssim \left\| (e_0^{(n)}\gamma^{(n)})N^{(n)}H^{(n)} \right\|_{H^2_{\delta'+2\e^2+1}}+\left\|(e_0^{(n)}N^{(n)})H^{(n)} \right\|_{H^2_{\delta'+2\e^2+1}}\\&\qquad+\left\|N^{(n)}(e_0^{(n)}H^{(n)}) \right\|_{H^2_{\delta'+2\e^2+1}}+\left\|\nabla\beta^{(n)}\right\|_{H^2_{\delta'+1}}\left\|  N^{(n)}H^{(n)} \right\|_{H^2_{\delta'+2\e^2+1}}.\nonumber
    \end{align*}
    The last term is handled thanks to \eqref{HR beta 2} and \eqref{NnHn} and is indeed bounded by $A_0C_i^2$. The third term is similar to the last one (because $e_0^{(n)}H^{(n)}$ satisfies \eqref{HR H 3}) and is therefore handled as in \eqref{NnHn}, finally it is bounded by $A_3C_i^2$. We handled the two first terms as we did in \eqref{I L^2}, using \eqref{HR H 1} instead of \eqref{propsmall H}, this change explains why we get $C_i^2$ instead of $C_i$.

    \item For $II$, we use again the commutation estimate \eqref{commutation estimate'}, \eqref{HR beta 2} and \eqref{HR+1 beta 2} and the product estimate :
    \begin{align*}
        \left\| \left[ \Delta,e_0^{(n)} \right](\beta^{(n+1)})^i\right\|_{L^2_{\delta'+2}} &\lesssim \left\|\nabla\beta^{(n)} \right\|_{H^2_{\delta'+1}}\left\|\nabla^2\beta^{(n+1)} \right\|_{H^1_{\delta'+2}}+\left\| \nabla^2\beta^{(n)}\right\|_{H^1_{\delta'+2}}\left\|\nabla\beta^{(n+1)} \right\|_{H^2_{\delta'+1}} \\&\lesssim A_0^2C_i^2.
    \end{align*}
\end{itemize}

\end{proof}

We have finished all the elliptic estimates, in the sequel we deal with evolution equations, and we will use the freedom of taking $T$ as small as we want, in order to recover our estimates.

\subsubsection{The transport equation and $\tau^{(n+1)}$}

We begin this section by prooving the estimates on $H^{(n+1)}$. We first prove a technical lemma about the transport equation :

\begin{lem}\label{transport inegalite}
Let $\sigma\in\R$. If $f$ and $h$ satisfy 
\begin{equation*}
    e_0^{(n+1)}f=h,
\end{equation*}
then,
\begin{equation*}
    \sup_{t\in[0,T]}\| f\|_{L^2_{\sigma}}(t)\leq 2 \| f\|_{L^2_{\sigma}}(0)+2\sqrt{T}\sup_{t\in[0,T]}\| h\|_{L^2_{\sigma}}(t)  .
\end{equation*}
\end{lem}

\begin{proof}Let $w(x)=\langle x\rangle^{2\sigma}$. We multiply the equation $e_0^{(n+1)}f=h$ by $w f$ and integrate over $\R^2$. Writing $e_0^{(n+1)}=\dr_t-\beta^{(n+1)}\cdot\nabla$, we get :
\begin{equation*}
    \frac{\d}{\d t}\left(\|f\|_{L^2_{\sigma}}^2\right)=2\int_{\R^2}w fh\,\d x+\int_{\R^2}w \beta^{(n+1)}\cdot\nabla \left(f^2\right)\,\d x.
\end{equation*}
We integrate by part the last term in order to get :
\begin{align*}
    \frac{\d}{\d t}\left(\|f\|_{L^2_{\sigma}}^2\right)&=2\int_{\R^2}w fh\,\d x-\int_{\R^2} f^2\dive\left(w\beta^{(n+1)}\right)\d x .
\end{align*}
For the last term, we use \eqref{HR+1 beta 2} (and the embedding $H^2_{\delta'+1}\xhookrightarrow{}C^0_1$) and $\left| \nabla w\right|\lesssim\frac{w}{\langle x \rangle}$ to obtain :
\begin{align*}
    -\int_{\R^2} f^2\dive\left(w\beta^{(n+1)}\right)\d x  \lesssim C_i \| f\|_{L^2_{\sigma}}
\end{align*}
For the first term, we simply use the Cauchy-Scwharz inequality and $2ab\leq a^2+b^2$ to obtain :
\begin{equation*}
    2\int_{\R^2}w  f h\,\d x\leq 2\left(\int_{\R^2}w f^2 \d x\right)^{\frac{1}{2}}\left(\int_{\R^2}w h^2 \d x\right)^{\frac{1}{2}}\leq \| f\|_{L^2_{\sigma}}^2+ \| h\|_{L^2_{\sigma}}^2.
\end{equation*}
Summarising, we get :
\begin{equation*}
    \frac{\d}{\d t}\left(\|f\|_{L^2_{\sigma}}^2\right)\leq C(C_i)  \| f\|_{L^2_{\sigma}}^2+ \| h\|_{L^2_{\sigma}}^2.
\end{equation*}
We apply Gronwall's Lemma, take $T$ small enough and use $\sqrt{a^2+b^2}\leq a+b$ to get :
\begin{equation*}
    \sup_{t\in[0,T]}\| f\|_{L^2_{\sigma}}(t)\leq 2 \| f\|_{L^2_{\sigma}}(0)+2\sqrt{T}\sup_{t\in[0,T]}\| h\|_{L^2_{\sigma}}(t)  .
\end{equation*}

\end{proof}

\begin{prop}\label{HR+1 H prop}
For $n\geq 2$, the following estimates hold :
\begin{align}
        \left\Vert H^{(n+1)}\right\Vert_{H^2_{\delta+1}}&\leq 2C_i\label{HR+1 H 1},\\
        \left\Vert e_0^{(n+1)}H^{(n+1)}\right\Vert_{L^2_{1+\lambda}}&\lesssim \e^2,\label{HR+1 H 1.5}\\
        \left\Vert e_0^{(n+1)}H^{(n+1)}\right\Vert_{H^1_{\delta+1}}&\lesssim C_i,\label{HR+1 H 2}\\
        \left\Vert e_0^{(n+1)}H^{(n+1)}\right\Vert_{H^2_{\delta+1}}&\lesssim A_2C_i^2.\label{HR+1 H 3}
    \end{align}
\end{prop}

\begin{proof}
To prove \eqref{HR+1 H 1.5} we just bound the RHS of \eqref{reduced system H} in $L^2_{\delta+1}$ using the weighted product estimates $H^1\times H^1\xhookrightarrow{} L^2$ and using weighted $L^{\infty}$ estimates for $N^{(n)}$ in the first and last terms. More concretely, we use \eqref{HR N 1}, \eqref{propsmall H}, \eqref{HR beta 1}, \eqref{propsmall gamma}, \eqref{propsmall fi}, \eqref{useful gamma 1} and \eqref{useful ffi 1}, and we recall that $\lambda<\delta+1$ :
\begin{align*}
    \left\Vert e_0^{(n+1)}H^{(n+1)}\right\Vert_{L^2_{1+\lambda}} &\lesssim \l N^{(n)} (H^{(n)})^2 \r_{L^2_{1+\lambda}} + \l\nabla\beta^{(n)} H^{(n)} \r_{L^2_{1+\lambda}}+\l \nabla^2 N^{(n)}\r_{L^2_{1+\lambda}} \\&\qquad+ \l \nabla\gamma^{(n)}\nabla N^{(n)}\r_{L^2_{1+\lambda}} + \l N^{(n)} (\nabla \ffi^{(n)})^2 \r_{L^2}+ \l N^{(n)} (\nabla \omega^{(n)})^2 \r_{L^2}\\&\lesssim \e^2.
\end{align*}

We continue by the proof of \eqref{HR+1 H 2} and \eqref{HR+1 H 3}, which amounts to bounding the $H^1_{\delta+1}$ and $H^2_{\delta+1}$ norms of the RHS of \eqref{reduced system H}. First notice that the terms $e^{-2\gamma^{(n)}}N^{(n)}(H^{(n)})_i^{\;\,\ell}(H^{(n)})_{j\ell}$, $(\dr_i\ffi^{(n)}\tb\dr_j\ffi^{(n)})N^{(n)}$ and $(\dr_i\omega^{(n)}\tb\dr_j\omega^{(n)})N^{(n)}$ are analogous to terms in \eqref{reduced system N} (because $\nabla\ffi^{(n)}$ and $\dr_t\ffi^{(n)}$ satisfy the same estimates, samewise for $\omega^{(n)}$), and can be treated as in Proposition \ref{hr+1 N prop}. We recall the estimates obtained :
\begin{align*}
    \l e^{-2\gamma^{(n)}}N^{(n)}(H^{(n)})_i^{\;\,\ell}(H^{(n)})_{j\ell} \r_{H^1_{\delta+1}} & \lesssim \e^2 C_i,\\
    \l e^{-2\gamma^{(n)}}N^{(n)}(H^{(n)})_i^{\;\,\ell}(H^{(n)})_{j\ell}\r_{H^2_{\delta+1}} & \lesssim \e^2 C(A_2) C_i^2,\\
    \l(\dr_i\ffi^{(n)}\tb\dr_j\ffi^{(n)})N^{(n)} \r_{H^1_{\delta+1}}+\l e^{-4\ffi}(\dr_i\omega^{(n)}\tb\dr_j\omega^{(n)})N^{(n)} \r_{H^1_{\delta+1}} & \lesssim \e C(A_0)C_i,\\
    \l(\dr_i\ffi^{(n)}\tb\dr_j\ffi^{(n)})N^{(n)} \r_{H^2_{\delta+1}}+\l e^{-4\ffi}(\dr_i\omega^{(n)}\tb\dr_j\omega^{(n)})N^{(n)} \r_{H^2_{\delta+1}} & \lesssim \e^2 C(A_2) C_i^2+C(A_0)C_i^2.
\end{align*}
The remaining terms are treated as follows :
\begin{itemize}
    \item We first use \eqref{HR H 1} and \eqref{HR beta 1} and the product estimate :
    \begin{align*}
        \left\|\dr_{(j}(\beta^{(n)})^k(H^{(n)})_{i)k} \right\|_{H^1_{\delta+1}}& \lesssim \left\|H^{(n)}\right\|_{H^2_{\delta+1}}\left\|\beta^{(n)}\right\|_{H^2_{\delta'+1}}\lesssim \e  C_i.
    \end{align*}
    Then we use \eqref{propsmall H}, \eqref{HR H 1} and \eqref{HR beta 2} and the product estimate :
    \begin{align*}
        \left\|\dr_{(j}(\beta^{(n)})^k(H^{(n)})_{i)k} \right\|_{H^2_{\delta+1}}&\lesssim \left\| \nabla\beta^{(n)}H^{(n)} \right\|_{H^1_{\delta+1}}+\left\| \nabla^3\beta^{(n)}H^{(n)}\right\|_{L^2_{\delta+3}}\\&\qquad+\left\| \nabla^2\beta^{(n)}\nabla H^{(n)}\right\|_{L^2_{\delta+3}}+\left\| \nabla\beta^{(n)}\nabla^2H^{(n)}\right\|_{L^2_{\delta+3}}
        \\&\lesssim\left\|\nabla\beta^{(n)}\right\|_{H^2_{\delta'+1}}\left\|H^{(n)}\right\|_{H^1_{\delta+1}} + \l\nabla^3\beta^{(n)} \r_{L^2_{\delta'+3}}\l H^{(n)} \r_{H^2_{\delta+1}} \\&\qquad+\l\nabla^2\beta^{(n)} \r_{H^1_{\delta'+2}} \l\nabla H^{(n)} \r_{H^1_{\delta+2}} +\left\| \nabla\beta^{(n)}\right\|_{H^2_{\delta'+1}}   \left\|\nabla^2H^{(n)}\right\|_{L^2_{\delta+3}}\\
        &\lesssim \e^2 A_0 C_i+A_0C_i^2.
    \end{align*}
    \item We use \eqref{HR N 2} and the fact that $\langle x\rangle^{\alpha}\in L^2$ if and only if $\alpha<-1$ :
    \begin{equation*}
        \left\| (\dr_i\tb\dr_j)N^{(n)} \right\|_{H^1_{\delta+1}}\leq \left\|\nabla^2\Tilde{N}^{(n)}\right\|_{H^1_{\delta+1}}+\left| N_a^{(n)} \right|\left\|\nabla^2(\chi\ln)\right\|_{H^1_{\delta+1}}\lesssim C_i.
    \end{equation*}
    We then use \eqref{HR N 3} for the $H^2$ estimate :
    \begin{equation*}
        \left\| (\dr_i\tb\dr_j)N^{(n)} \right\|_{H^2_{\delta+1}} \leq \left\|\nabla^2\Tilde{N}^{(n)}\right\|_{H^2_{\delta+1}}+\left| N_a^{(n)} \right|\left\|\nabla^2(\chi\ln)\right\|_{H^2_{\delta+1}}\lesssim A_2 C_i^2.
    \end{equation*}
    \item For the following term, we get both $H^1$ and $H^2$ estimates by using \eqref{HR N 1}, \eqref{HR N 2} and \eqref{HR gamma 1} and the product estimate :
    \begin{align*}
        \left\| (\delta_i^k\tb\dr_j\gamma^{(n)})\dr_kN^{(n)}\right\|_{H^2_{\delta+1}}&\leq \left\|\nabla\Tilde{\gamma}^{(n)}\nabla\Tilde{N}^{(n)}\right\|_{H^2_{\delta+1}}+|\alpha|\left\|\nabla(\chi\ln)\nabla\Tilde{N}^{(n)}\right\|_{H^2_{\delta+1}}\\&\quad+\left|N_a^{(n)}\right|\left\|\nabla(\chi\ln)\nabla\Tilde{\gamma}^{(n)}\right\|_{H^2_{\delta+1}}+|\alpha|\left|N_a^{(n)}\right|\left\|(\nabla(\chi\ln))^2\right\|_{H^2_{\delta+1}}
        \\&\lesssim
        \left\|\nabla\Tilde{\gamma}^{(n)}\right\|_{H^2_{\delta+1}}
        \left\|\nabla\Tilde{N}^{(n)}\right\|_{H^2_{\delta+1}} 
        +\e\left( \left\|\nabla\Tilde{N}^{(n)}\right\|_{H^2_{\delta}}+\left\|\nabla\Tilde{\gamma}^{(n)}\right\|_{H^2_{\delta}} \right)\\&\qquad+
        \e^2\left\|\langle x\rangle^{\delta-1}\right\|_{L^2}
        \\&\lesssim\e C_i^2.
    \end{align*}
\end{itemize}
We now prove \eqref{HR+1 H 1}. We recall the following commutation formula :
\begin{align*}
    \left|\left[ e_0^{(n+1)},\nabla  \right]H^{(n+1)}\right|&\lesssim \left|\nabla\beta^{(n+1)} \right|\left|\nabla H^{(n+1)} \right|,\\
    \left|\left[ e_0^{(n+1)},\nabla^2  \right]H^{(n+1)}\right|&\lesssim \left|\nabla\beta^{(n+1)} \right|\left|\nabla^2 H^{(n+1)} \right|+\left|\nabla^2\beta^{(n+1)} \right|\left|\nabla H^{(n+1)} \right|.
\end{align*}
Hence, using \eqref{HR+1 beta 2} :
\begin{equation}
    \left\|e_0^{(n+1)}\nabla^{\alpha}H^{(n+1)}_{ij} \right\|_{L^2_{\delta+1+|\alpha|}}\lesssim \left\|e_0^{(n+1)}H^{(n+1)}_{ij} \right\|_{H^2_{\delta+1}}+C_i\left\|H^{(n+1)}\right\|_{H^2_{\delta+1}}\label{transport H}
\end{equation}
where $|\alpha|\leq 2$. We apply the Lemma \ref{transport inegalite} with $\sigma=\delta+1+|\alpha|$ and $f=\nabla^{\alpha} H^{(n+1)}$ :
\begin{align*}
    \sup_{t\in[0,T]}\left\| \nabla^{\alpha} H^{(n+1)} \right\|_{L^2_{\delta+1+|\alpha|}}(t) &\leq 2\left\| \nabla^{\alpha} H^{(n+1)} \right\|_{L^2_{\delta+1+|\alpha|}}(0)+2\sqrt{T}\sup_{t\in[0,T]}\left\|e_0^{(n+1)}\nabla^{\alpha}H^{(n+1)}_{ij} \right\|_{L^2_{\delta+1+|\alpha|}} 
    \\& \lesssim 2\left\| \nabla^{\alpha} H^{(n+1)} \right\|_{L^2_{\delta+1+|\alpha|}}(0) +2\sqrt{T}C_i^2+2C_i\sqrt{T}\left\|H^{(n+1)}\right\|_{H^2_{\delta+1}},
\end{align*}
where in the last inequality we use \eqref{transport H} and \eqref{HR+1 H 3}. 

We sum over all $|\alpha|\leq 2$ and absorb the term $\left\|H^{(n+1)}\right\|_{H^2_{\delta+1}}$ of the RHS into the LHS (choosing $T$ small enough). Recalling that $\left\| H^{(n+1)} \right\|_{H^2_{\delta+1}}(0)\leq C_i$ ends the proof of \eqref{HR+1 H 1}.

\end{proof}

Next, we prove the estimates for $\tau^{(n+1)}$, gathered in the following proposition :

\begin{prop}\label{hr+1 tau prop}
For $n\geq 2$, the following estimates hold :
\begin{align}
        \left\| \tau^{(n+1)}\right\|_{H^2_{\delta'+1}}&\lesssim A_0C_i\label{HR+1 tau 1},\\
        \left\| \dr_t\tau^{(n+1)}\right\|_{L^2_{\delta'+1}}&\lesssim A_1C_i\label{HR+1 tau 2},\\
        \left\| \dr_t\tau^{(n+1)}\right\|_{H^1_{\delta'+1}}&\lesssim A_2 C_i\label{HR+1 tau 3}.
\end{align}
\end{prop}

\begin{proof}
In view of $\eqref{reduced system tau}$, the estimates for $\tau^{(n+1)}$ can be obtained by directly controlling 
\begin{equation*}
     -2\Ll^{(n-1)}\gamma^{(n)}+\frac{\dive\left(\beta^{(n)}\right)}{N^{(n-1)}} .
\end{equation*}
We bound the two terms separately, using first \eqref{HR beta 2}, \eqref{HR N 1}, \eqref{HR N 2} and $\left|\frac{1}{N^{(n-1)}} \right|\lesssim 1$ :
\begin{align*}
    \left\| \frac{\dive\left(\beta^{(n)}\right)}{N^{(n-1)}} \right\|_{H^2_{\delta'+1}} & \lesssim 
    \left\| \beta^{(n)}\right\|_{H^3_{\delta'}} + \l \nabla\Tilde{N}^{(n-1)}\nabla \beta^{(n)} \r_{L^2_{\delta'+2}} \\&\qquad+ \l\nabla^2\Tilde{N}^{(n-1)}\nabla \beta^{(n)} \r_{L^2_{\delta'+3}}+\l\nabla\Tilde{N}^{(n-1)}\nabla^2 \beta^{(n)} \r_{L^2_{\delta'+3}}\\
    &\lesssim \left\| \beta^{(n)}\right\|_{H^3_{\delta'}}+ \l \nabla\Tilde{N}^{(n-1)}\r_{H^1_{\delta+1}}\l\nabla\beta^{(n)} \r_{H^1_{\delta'+1}} \\& \qquad +\l\nabla^2\Tilde{N}^{(n-1)}\r_{H^1_{\delta+2}} \l\nabla\beta^{(n)} \r_{H^1_{\delta'+1}}+\l \nabla\Tilde{N}^{(n-1)}\r_{H^1_{\delta+1}}\l \nabla^2\beta^{(n)}\r_{H^1_{\delta'+2}}
    \\& \lesssim (1+\e)A_0C_i.
\end{align*}
Using in addition Proposition \ref{commutation estimate} we get :
\begin{align*}
    \left\|\Ll^{(n-1)}\gamma^{(n)} \right\|_{H^2_{\delta'+1}} \leq \left\| \Ll^{(n-1)}\Tilde{\gamma}^{(n)} \right\|_{H^2_{\delta'+1}} +|\alpha|\left\|\frac{\beta^{(n-1)}\cdot\nabla(\chi\ln)}{N^{(n-1)}} \right\|_{H^2_{\delta'+1}}\lesssim (1+\e)C_i,
\end{align*}
which concludes the proof of \eqref{HR+1 tau 1}. 
\par\leavevmode\par
We now turn to the estimates concerning $\dr_t\tau^{(n+1)}$, which has the following expression :
\begin{equation*}
    \dr_t\tau^{(n+1)}=-2\dr_t\left( \Ll^{(n-1)}\Tilde{\gamma}^{(n)} \right)+2\alpha\nabla(\chi\ln)\cdot\dr_t\left( \frac{\beta^{(n-1)}}{N^{(n-1)}} \right)+\dr_t\left( \frac{\dive(\beta^{(n)})}{N^{(n-1)}} \right)
\end{equation*}
By \eqref{HR gamma 2}, $\left\|\dr_t( \Ll^{(n-1)}\Tilde{\gamma}^{(n)})\right\|_{L^2_{\delta'+1}}\leq A_0 C_i$. Then, we note, that thanks to \eqref{HR beta 1} and \eqref{HR beta 3}, we have $\|\dr_t\beta^{(n)}\|_{H^1_{\delta'+1}}\lesssim  A_1C_i$ (and the same  with $n$ replaced by $n-1$). For the second term, we do the following :
\begin{align}
    \left\| 2\alpha\nabla(\chi\ln)\cdot\dr_t\left( \frac{\beta^{(n-1)}}{N^{(n-1)}} \right)\right\|_{L^2_{\delta'+1}}&\lesssim\;\e\left( \left\|\dr_t\beta^{(n-1)}\right\|_{H^1_{\delta'+1}}+\left| \dr_tN_a^{(n-1)} \right|\left\|\beta^{(n-1)}\right\|_{L^2_{\delta'}}+\left\|\dr_t\Tilde{N}^{(n-1)}\right\|_{L^2_{\delta}}\left\|\beta^{(n-1)}\right\|_{L^{\infty}}\right)\nonumber\\&\lesssim \e A_1C_i,\label{lkj}
\end{align}
where we used \eqref{HR N 2} and \eqref{HR beta 1}. The third term is very similar :
\begin{align*}
    \left\| \dr_t\left( \frac{\dive(\beta^{(n)})}{N^{(n-1)}} \right)\right\|_{L^2_{\delta'+1}}&\lesssim \left\|\dr_t\beta^{(n)}\right\|_{H^1_{\delta'}}+ \left\|\nabla\beta^{(n)}\right\|_{H^1_{\delta'+1}}\left(  \left\|\dr_t\Tilde{N}^{(n-1)}\right\|_{H^1_{\delta}} +\left| \dr_tN_a^{(n-1)} \right|\right)\\& \lesssim A_1C_i.
\end{align*}
This finishes the proof of \eqref{HR+1 tau 2}.
\par\leavevmode\par
We now turn to the proof of \eqref{HR+1 tau 3}. In view of \eqref{HR+1 tau 2}, we just have to bound $\|\nabla \dr_t\tau^{(n+1)}\|_{L^2_{\delta'+2}}$ by $C_i^2$. We have the following expression :
\begin{align*}
    \nabla \dr_t\tau^{(n+1)}&=-2\nabla\dr_t\left( \Ll^{(n-1)}\Tilde{\gamma}^{(n)} \right)+2\alpha\nabla^2(\chi\ln)\cdot \dr_t\left( \frac{\beta^{(n-1)}}{N^{(n-1)}} \right)+2\alpha\nabla(\chi\ln)\dr_t\left( \frac{\nabla\beta^{(n-1)}}{N^{(n-1)}} \right)\\
    &\quad -2\alpha\nabla(\chi\ln)\dr_t\left( \frac{\beta^{(n-1)} \nabla N^{(n-1)}}{\left(N^{(n-1)}\right)^2} \right)+ \dr_t\left( \frac{\dive(\nabla\beta^{(n)})}{N^{(n-1)}} \right)-\dr_t\left( \frac{\nabla N^{(n-1)}\dive(\beta^{(n)})}{\left(N^{(n-1)}\right)^2} \right)
    \\&=\vcentcolon I + II + III+IV+V+VI.
\end{align*}
The term $I$ is easily handle thanks to \eqref{HR gamma 3} : we have $\left\|I\right\|_{L^2_{\delta'+2}}\leq A_2 C_i$. For the other terms, we make the following remarks :
\begin{itemize}
    \item the term $VI$ is worse than the term $IV$,
    \item the term $V$ is worse than the terms $II$ and $III$. 
\end{itemize}
Thus, it only remains to bound the terms $V$ and $VI$, for which we use \eqref{HR N 2}, \eqref{HR beta 1} and \eqref{HR beta 3} :
\begin{align*}
    \l V\r_{L^2_{\delta'+2}} &\lesssim \l\dr_t\beta^{(n)} \r_{H^2_{\delta'}} +\l\nabla^2\beta^{(n)} \r_{L^2_{\delta'+2}}\l\dr_t\Tilde{N}^{(n-1)} \r_{H^2_{\delta}}\lesssim A_1C_i+\e C_i\lesssim A_1 C_i.\\
    \l VI \r_{L^2_{\delta'+2}} & \lesssim \l\nabla\Tilde{N}^{(n-1)} \r_{H^1_{\delta+1}}\l\dr_t\beta^{(n)} \r_{H^1_{\delta'}}+\l\nabla\dr_t\Tilde{N}^{(n-1)} \r_{H^1_{\delta+1}}\l\nabla\beta^{(n-1)} \r_{H^1_{\delta'+1}}\\&\qquad +\l\dr_t\Tilde{N}^{(n-1)} \r_{H^2_{\delta}}\l\nabla\Tilde{N}^{(n-1)} \r_{H^1_{\delta+1}}\l\nabla\beta^{(n-1)}
    \r_{H^1_{\delta'+1}}
    \\&\lesssim \e C_i. 
\end{align*}
This concludes the proof of \eqref{HR+1 tau 3}.
\par\leavevmode\par

\end{proof}

\subsubsection{Energy estimate for $\Box_{g^{(n)}}$} 

In this section, we establish the usual energy estimate for the operator $\Box_{g^{(n)}}$. 

\begin{lem}\label{inegalite d'energie lemme 1}
Let $\sigma\in\R$. If $h$ is a solution of
\begin{equation}\label{partie principale de box}
    \left(\Ll^{(n)} \right)^2h-e^{-2\gamma^{(n)}}\Delta h=f, 
\end{equation}
then, if $T$ is sufficiently small, we have for all $t\in[0,T]$ 
\begin{align}
    \left\Vert\Ll^{(n)}h\right\Vert_{L^2_{\sigma}}(t)&+\left\Vert e^{-\gamma^{(n)}}\nabla h\right\Vert_{L^2_{\sigma}}(t)\nonumber\\&  \leq 2\left( \left\Vert\Ll^{(n)}h\right\Vert_{L^2_{\sigma}}(0)+\left\Vert e^{-\gamma^{(n)}}\nabla h\right\Vert_{L^2_{\sigma}}(0)+\sqrt{2T}\sup_{s\in[0,T]}\left\| fN^{(n)}\right\|_{L^2_{\sigma}}\right).\label{inégalité d'énergie}
\end{align}
\end{lem}
\begin{proof}
Let $w(x)=\langle x \rangle^{2\sigma}$. We multiply the equation by $w e_0^{(n)}h$ and we integrate over $\R^2$ with respect to $\d x$. After integration by parts we obtain :
\begin{equation}
    \int_{\R^2}\frac{w}{2}e_0^{(n)}\left(\Ll^{(n)}h\right)^2\d x+\int_{\R^2}\nabla h\cdot\nabla\left(e^{-2\gamma^{(n)}}w e_0^{(n)}h\right)\d x=\int_{\R^2}w f e_0^{(n)}h\,\d x.\label{IPP}
\end{equation}
We define the energy $E(t)\vcentcolon=\int_{\R^2}w\left( \left(\Ll^{(n)}h\right)^2+e^{-2\gamma^{(n)}}|\nabla h|^2\right)(t,x)\d x$ and compute its time derivative, writing $\dr_t=e_0^{(n)}+\beta^{(n)}\cdot\nabla$ and integrating by parts the terms coming from $\beta^{(n)}\cdot\nabla$ :
\begin{align*}
\frac{\d E}{\d t}(t)  = \int_{\R^2} we_0^{(n)}\left(\Ll^{(n)}h\right)^2\d x + \int_{\R^2} w e_0^{(n)}& \left(e^{-2\gamma^{(n)}} |\nabla h|^2 \right) \d x \\&-\int_{\R^2}\dive  (w\beta^{(n)})\left( \left(\Ll^{(n)}h\right)^2+e^{-2\gamma^{(n)}}|\nabla h|^2\right)\d x 
\end{align*}
We now use \eqref{IPP} to express the first integral in $\frac{\d E}{\d t}$ :
\begin{align*}
\frac{\d E}{\d t}(t)& =2\int_{\R^2}w f e_0^{(n)}h\,\d x - 2\int_{\R^2}\nabla h\cdot\nabla\left(e^{-2\gamma^{(n)}}w e_0^{(n)}h\right)\d x
\\& \quad  + \int_{\R^2} w e_0^{(n)} \left(e^{-2\gamma^{(n)}} |\nabla h|^2 \right) \d x -\int_{\R^2}\dive  (w\beta^{(n)})\left( \left(\Ll^{(n)}h\right)^2+e^{-2\gamma^{(n)}}|\nabla h|^2\right)\d x
\end{align*}
We now expand the second integral and commute $\nabla$ and $e_0^{(n)}$ :
\begin{align*}
- 2\int_{\R^2}\nabla h\cdot\nabla\left(e^{-2\gamma^{(n)}}w e_0^{(n)}h\right)\d x & =- \int_{\R^2}w e^{-2\gamma^{(n)}} e_0^{(n)}\left( |\nabla h|^2 \right)  \d x +2 \int_{\R^2} w e^{-2\gamma^{(n)}} \dr_i h \nabla h \cdot \nabla\beta^{(n)i} \d x
\\& \quad -2 \int_{\R^2} e^{-2\gamma^{(n)}}e_0^{(n)}h \nabla h\cdot\nabla w \d x  -2 \int_{\R^2} w e_0^{(n)} h\nabla h\cdot\nabla\left( e^{-2\gamma^{(n)}}\right)  \d x
\end{align*}
With this, we see that the $\dr h \dr^2h$ terms in $\frac{\d E}{\d t}$ cancel each other. Thus, we obtain the following energy equality :
\begin{align}
\frac{\d E}{\d t}(t)& = -\int_{\R^2}\dive  (w\beta^{(n)})\left( \left(\Ll^{(n)}h\right)^2+e^{-2\gamma^{(n)}}|\nabla h|^2\right)\d x  \nonumber
 +2 \int_{\R^2} w e^{-2\gamma^{(n)}} \dr_i h \nabla h \cdot \nabla\beta^{(n)i} \d x\\& \quad -2 \int_{\R^2} e^{-2\gamma^{(n)}}e_0^{(n)}h \nabla h\cdot\nabla w \d x + 2\int_{\R^2}w f e_0^{(n)}h\,\d x + R_{\gamma^{(n)}}(t) \label{energy equality}
\end{align}
with
\begin{equation*}
    R_{\gamma^{(n)}}(t)\vcentcolon=-2\int_{\R^2}e^{-2\gamma^{(n)}}w|\nabla h|^2\dr_t\gamma^{(n)}\d x+ 4\int_{\R^2}e^{-2\gamma^{(n)}}w e_0^{(n)}h\nabla h\cdot\nabla\gamma^{(n)}\d x-2\int_{\R^2} e^{-2\gamma^{(n)}}w|\nabla h|^2\beta^{(n)}\cdot \nabla\gamma^{(n)}\d x.
\end{equation*}
This sort of remainder contains all the term involving derivatives of $\gamma^{(n)}$, therefore it would vanish if $e^{-2\gamma^{(n)}}$ didn't appear in \eqref{partie principale de box}. We then show that the first three integrals in \eqref{energy equality} can be bounded by $E(t)$ :
\begin{itemize}
    \item Let's show that $|\dive(w\beta^{(n)})|\leq C(C_i)w$. We have $\dive(w\beta^{(n)})=w\dive(\beta^{(n)})+\nabla w\cdot\beta^{(n)}$. We have $|\nabla w|\lesssim\frac{ w}{\langle x\rangle}$ and $\beta^{(n)}$ bounded so $\nabla w\cdot\beta^{(n)}$ is indeed bounded by $w$. For $w\dive(\beta^{(n)})$, we use the embedding $H^2_{\delta'+1}\xhookrightarrow{}L^{\infty}$ and the estimate \eqref{HR beta 2}. This shows that 
    \begin{equation*}
        -\int_{\R^2}\dive  (w\beta^{(n)})\left( \left(\Ll^{(n)}h\right)^2+e^{-2\gamma^{(n)}}\nabla h\right)\d x\lesssim A_0C_iE(t).
    \end{equation*}
    \item Let's show that $|e^{-\gamma^{(n)}}N^{(n)}\nabla w|\lesssim w$. We have $|e^{-\gamma^{(n)}}|\lesssim \langle x\rangle^{\e^2}$ and $|N^{(n)}|\lesssim\langle x\rangle^{\frac{1}{2}}$ so $|e^{-\gamma^{(n)}}N^{(n)}\nabla w|\lesssim w \langle x\rangle^{\e^2-\frac{1}{2}}\lesssim w$, providing $\e$ is small. This allows us to do the following (using $2ab\leq a^2+b^2$) : 
    \begin{equation*}
        -2\int_{\R^2}e^{-2\gamma^{(n)}}e_0^{(n)}h\nabla h\cdot\nabla w\,\d x \lesssim \int_{\R^2} w\left|\frac{e_0^{(n)}h}{N^{(n)}}\right|e^{-\gamma^{(n)}}|\nabla h|\,\d x\lesssim E(t).
    \end{equation*}
    \item We already used the fact that $\nabla\beta^{(n)}$ is bounded by $A_0C_i$ so we simply do
    \begin{equation*}
        2\int_{\R^2}e^{-2\gamma^{(n)}}w\dr_ih\nabla h\cdot\nabla\beta^{(n)i}\,\d x\lesssim A_0C_i\int_{\R^2}e^{-2\gamma^{(n)}}w|\nabla h|^2\,\d x \lesssim A_0C_iE(t).
    \end{equation*}
\end{itemize}
We now show that $R_{\gamma^{(n)}}(t)$ can also be bounded by $E(t)$ :
\begin{itemize}
    \item Let's show that $\dr_t\gamma^{(n)}$ is bounded. Since $\alpha$ doesn't depend on time, we have $\dr_t\gamma^{(n)}=N^{(n-1)}\Ll^{(n-1)}\Tilde{\gamma}^{(n)}+\beta^{(n)}\cdot\nabla\Tilde{\gamma}^{(n)}$. For the first term we use the Proposition \ref{commutation estimate} and the embedding $H^2_{\delta'+1}\xhookrightarrow{}C^0_1$ (together with the fact that $|N^{(n-1)}|\lesssim\langle x\rangle$). For the second term we simply use the embedding $H^2_{\delta'}\xhookrightarrow{}L^{\infty}$. We thus get 
    \begin{equation*}
        2\int_{\R^2}e^{-2\gamma^{(n)}} w|\nabla h|^2\dr_t\gamma^{(n)}\d x \leq C(C_i) E(t).
    \end{equation*}
    \item Let's show that $e^{-\gamma^{(n)}}N^{(n)}\nabla\gamma^{(n)}$ is bounded. We only deal with the $\chi\ln$ part of $N^{(n)}$ (since $\Tilde{N}^{(n)}$ is bounded), and only with the $\nabla\Tilde\gamma^{(n)}$ part in $\nabla\gamma^{(n)}$ (because $\nabla(\chi\ln)$ decrease more than $e^{-\gamma^{(n)}}$). Using $|\chi\ln|\lesssim \langle x\rangle^{\e^2}$ and $e^{-\gamma^{(n)}}\lesssim \langle x\rangle^{\e^2}$, we write
    \begin{equation*}
        \left|e^{-\gamma^{(n)}}\chi\ln\nabla\Tilde{\gamma}^{(n)}\right|\lesssim \|\nabla\Tilde{\gamma}^{(n)}\|_{C^0_{2\e^2}}.
    \end{equation*}
    If $\e$ is small enough we have the embedding $H^2_{\delta'+1}\xhookrightarrow{}C^0_{2\e^2}$ which together with \eqref{HR gamma 1} allows us to say
    \begin{equation*}
        4\int_{\R^2}e^{-2\gamma^{(n)}}w e_0^{(n)}h\nabla h\cdot\nabla\gamma^{(n)}\d x \lesssim C(C_i) \int_{\R^2}w\left|\frac{e_0^{(n)}h}{N^{(n)}}\right|e^{-\gamma^{(n)}}|\nabla h|\,\d x \leq C(C_i) E(t).
    \end{equation*}
    \item We already used multiple times that $\beta^{(n)}$ and $\nabla\gamma^{(n)}$ are bounded (by $\e$ and $C(C_i)$ respectively), and thus
    \begin{equation*}
        2\int_{\R^2} e^{-2\gamma^{(n)}}w|\nabla h|^2\beta^{(n)}\cdot \nabla\gamma^{(n)}\d x\leq C(C_i) E(t).
    \end{equation*}
\end{itemize}
For the last integral in \eqref{energy equality} we apply Cauchy-Schwarz inequality :
\begin{equation*}
    2\int_{\R^2}w  f e_0^{(n)}h\,\d x\leq 2\left(\int_{\R^2}w f^2 N^{(n)2}\d x\right)^{\frac{1}{2}}\left(\int_{\R^2}w \left(\Ll^{(n)}h\right)^2 \d x\right)^{\frac{1}{2}}\leq E(t)+\int_{\R^2}w f^2 N^{(n)2}\d x .
\end{equation*}
Summarising all the estimates, we get :
\begin{equation*}
    \frac{\d E}{\d t}(t)\leq C(C_i)E(t)+\int_{\R^2}w f^2 N^{(n)2}(t,x)\,\d x .
\end{equation*}
We apply Gronwall's inequality with $T$ sufficiently small to obtain 
\begin{equation*}
    E(t)\leq 2\left( E(0)+T\sup_{s\in[0,T]}\int_{\R^2}w f^2 N^{(n)2}(s,x)\,\d x\right) .
\end{equation*}
We recognize in $E(t)$ a weighted Sobolev norm, and using inequality such as $\frac{1}{\sqrt{2}}(a+b)\leq \sqrt{a^2+b^2}\leq a+b$, we obtain the inequality of the lemma.

\end{proof}

\begin{lem}\label{inegalite d'energie lemme}
If $h$ is a solution of \eqref{partie principale de box} then, if $T$ is sufficiently small, we have for all $t\in[0,T]$ 
\begin{align}
    &\sum_{|\alpha|\leq 2} \left( \left\Vert\Ll^{(n)}\nabla^{\alpha}h\right\Vert_{L^2_{\delta'+1+|\alpha|}}(t)+\left\Vert e^{-\gamma^{(n)}}\nabla (\nabla^{\alpha}h)\right\Vert_{L^2_{\delta'+1+|\alpha|}}(t)\right)\nonumber \\
    &    \leq 3\sum_{|\alpha|\leq 2} \left( \left\Vert\Ll^{(n)}\nabla^{\alpha}h\right\Vert_{L^2_{\delta'+1+|\alpha|}}(0)+\left\Vert e^{-\gamma^{(n)}}\nabla (\nabla^{\alpha}h)\right\Vert_{L^2_{\delta'+1+|\alpha|}}(0)\right)+C(C_i)\sqrt{T}\sup_{s\in[0,T]}\left\| fN^{(n)}\right\|_{H^2_{\delta'+1}}(s). \label{inegalite d'energie equation}
\end{align}
\end{lem}

\begin{proof}
For the sake of clarity, we set
\begin{equation*}
    \mathcal{E}^{(n)}[h](t)=\sum_{|\alpha|\leq 2} \left( \left\Vert\Ll^{(n)}\nabla^{\alpha}h\right\Vert_{L^2_{\delta'+1+|\alpha|}}(t)+\left\Vert e^{-\gamma^{(n)}}\nabla (\nabla^{\alpha}h)\right\Vert_{L^2_{\delta'+1+|\alpha|}}(t)\right)
\end{equation*}
If $h$ satisfies \eqref{partie principale de box}, then, applying $\nabla^{\alpha}$ to the equation, we show that $\nabla^{\alpha}h$ satisfies
\begin{equation}
   \left( \Ll^{(n)} \right)^2\nabla^{\alpha}h-e^{-2\gamma^{(n)}}\Delta (\nabla^{\alpha}h)=\nabla^{\alpha}f+\left[\nabla^{\alpha},e^{-2\gamma^{(n)}}\right]\Delta h+\left[\left( \Ll^{(n)} \right)^2,\nabla^{\alpha}\right]h. \label{eq derivee}
\end{equation}
Thanks to the previous lemma, in order to prove \eqref{inegalite d'energie equation}, we have to bound $\sum_{|\alpha|\leq 2}\Vert (\text{RHS of \eqref{eq derivee}})\times N^{(n)}\Vert_{L^2_{\delta'+1+|\alpha|}}$.
\begin{itemize}
    \item First step is bounding $\Vert N^{(n)}\nabla^{\alpha}f \Vert_{L^2_{\delta'+1+|\alpha|}}$ (using the fact that $\frac{1}{N^{(n)}}\in L^{\infty}$, $|\nabla^{\alpha}(\chi\ln)|\leq \langle x\rangle^{-|\alpha|}$, the product estimates and \eqref{HR N 2}). If $|\alpha|=1$ :
    \begin{align*}
        \Vert N^{(n)}\nabla^{\alpha}f \Vert_{L^2_{\delta'+2}} &\lesssim \Vert\nabla^{\alpha}(fN^{(n)})\Vert_{L^2_{\delta'+2}}+  \Vert f\nabla^{\alpha} N^{(n)}\Vert_{L^2_{\delta'+2}}\\
        &\lesssim \Vert fN^{(n)}\Vert_{H^2_{\delta'+1}}+\Vert N^{(n)} f\nabla^{\alpha}(\chi\ln)\Vert_{L^2_{\delta'+2}}+\Vert N^{(n)} f\nabla^{\alpha}\Tilde{N}^{(n)}\Vert_{L^2_{\delta'+2}}\\
        &\lesssim  \Vert fN^{(n)}\Vert_{H^2_{\delta'+1}} \left( 2+ \Vert \nabla^{\alpha}\Tilde{N}^{(n)}\Vert_{H^{2}_{\delta+1}}\right)\\
        &\leq C(C_i)\Vert fN^{(n)}\Vert_{H^2_{\delta'+1}}.
    \end{align*}
    If $|\alpha|=2$, then there exist $\alpha_1,\alpha_2$ with $|\alpha_1|=|\alpha_2|=1$ such that :
    \begin{equation*}
        \Vert N^{(n)}\nabla^{\alpha}f\Vert_{L^2_{\delta'+3}}\lesssim \Vert \nabla^{\alpha}(fN^{(n)}) \Vert_{L^2_{\delta'+3}}+\Vert f\nabla^{\alpha}N^{(n)} \Vert_{L^2_{\delta'+3}}+\Vert \nabla^{\alpha_1}N^{(n)}\nabla^{\alpha_2}f \Vert_{L^2_{\delta'+3}}.
    \end{equation*}
    The two first terms can be handled as in the case $|\alpha|=1$. For the last term we do the following :
    \begin{align*}
        \Vert \nabla^{\alpha_1}N^{(n)}\nabla^{\alpha_2}f \Vert_{L^2_{\delta'+3}} & \lesssim \Vert N^{(n)} \nabla^{\alpha_1}N^{(n)}\nabla
        ^{\alpha_2}f \Vert_{L^2_{\delta'+3}}\\
        &\lesssim \Vert N^{(n)}\nabla^{\alpha_2}f \Vert_{L^2_{\delta'+2}}\Vert \nabla^{\alpha_1}N^{(n)} \Vert_{H^2_{\delta+1}}\\
        &\leq C(C_i)\Vert fN^{(n)}\Vert_{H^2_{\delta'+1}},
    \end{align*}
    where in the last inequality we use the calculation of the $|\alpha|=1$ case. Summarising, we get :
    \begin{equation}
        \sum_{|\alpha|\leq 2}\|N^{(n)}\nabla^{\alpha}f \|_{L^2_{\delta'+1+|\alpha|}}\leq C(C_i)\Vert fN^{(n)}\Vert_{H^2_{\delta'+1}}\label{first step}.
    \end{equation}
    \item Second step is bounding $\left\| N^{(n)}\left[\nabla^{\alpha},e^{-2\gamma^{(n)}}\right]\Delta h \right\|_{L^2_{\delta'+1+|\alpha|}}$. If $|\alpha|=1$ we have $\left[\nabla^{\alpha},e^{-2\gamma^{(n)}}\right]\Delta h=-2e^{-2\gamma^{(n)}}\nabla^{\alpha}\gamma^{(n)}\Delta h$.
    Using the fact that $\left|\Tilde{N}^{(n)}\right|\leq\e$, \eqref{useful gamma 1}, $|\chi\ln|\lesssim\langle x\rangle^{\e^2}$ and the expression of $\gamma^{(n)}$ we have
    \begin{align*}
        \left\| e^{-2\gamma^{(n)}}N^{(n)}\nabla^{\alpha}\gamma^{(n)}\Delta h\right\|_{L^2_{\delta'+2}} & \lesssim \left\| \nabla^{\alpha}\gamma^{(n)}\Delta h\right\|_{L^2_{\delta'+2+3\e^2}} \\ & \lesssim\left\| \nabla^{\alpha}(\chi\ln)\Delta h\right\|_{L^2_{\delta'+2+3\e^2}}+\left\| \nabla^{\alpha}\Tilde{\gamma}^{(n)}\Delta h\right\|_{L^2_{\delta'+2+3\e^2}}.
    \end{align*}
    Since $|\alpha|=1$, we have $|\nabla^{\alpha}(\chi\ln)|\lesssim \langle x \rangle^{-1}$ and $\nabla^{\alpha}\Tilde{\gamma}^{(n)}\in H^2_{\delta'+1}$ which embeds in $C^0_1$. This implies :
    \begin{equation*}
        \left\| e^{-2\gamma^{(n)}}N^{(n)}\nabla^{\alpha}\gamma^{(n)}\Delta h\right\|_{L^2_{\delta'+2}} \leq C(C_i) \|\Delta h\|_{L^2_{\delta'+1+3\e^2}}\leq C(C_i)\sum_{|\alpha'|=1}\left\Vert e^{-\gamma^{(n)}}\nabla (\nabla^{\alpha'}h)\right\Vert_{L^2_{\delta'+2}},
    \end{equation*}
    where in the last inequality, we used that $1\lesssim |e^{-\gamma^{(n)}}|$ and took $\e$ small enough. If $|\alpha|=2$, then there exist $\alpha_1,\alpha_2$ with $|\alpha_1|=|\alpha_2|=1$ such that :
    \begin{align*}
        \left\| N^{(n)}\left[\nabla^{\alpha},e^{-2\gamma^{(n)}}\right]\Delta h \right\|_{L^2_{\delta'+3}} &\lesssim \left\|e^{-2\gamma^{(n)}}N^{(n)}\nabla^{\alpha}\gamma^{(n)} \Delta h \right\|_{L^2_{\delta'+3}}+\left\|e^{-2\gamma^{(n)}}N^{(n)}\nabla^{\alpha_1}\gamma^{(n)} \nabla^{\alpha_2}\Delta h \right\|_{L^2_{\delta'+3}}\\
        &\quad + \l e^{-2\gamma^{(n)}} N^{(n)} \nabla^{\alpha_1}\gamma^{(n)}\nabla^{\alpha_2}\gamma^{(n)}\Delta h \r_{L^2_{\delta'+3}} 
        \\& \lesssim \left\|\nabla^{\alpha}\gamma^{(n)} \Delta h\right\|_{L^2_{\delta'+3+3\e^2}}+\left\|\nabla^{\alpha_1}\gamma^{(n)} \nabla^{\alpha_2}\Delta h \right\|_{L^2_{\delta'+3+3\e^2}}
        \\&\quad + \l  \nabla^{\alpha_1}\gamma^{(n)}\nabla^{\alpha_2}\gamma^{(n)}\Delta h \r_{L^2_{\delta'+3+3\e^2}}
    \end{align*}
    For the first term, we use the fact that $|\nabla^{\alpha}(\chi\ln)|\lesssim \langle x \rangle^{-2}$, $\nabla^{\alpha}\Tilde{\gamma}^{(n)}\in H^1_{\delta'+2}$ (since $|\alpha|=2$) and the product estimate to get :
    \begin{equation*}
        \left\|\nabla^{\alpha}\gamma^{(n)} \Delta h\right\|_{L^2_{\delta'+3+3\e^2}}\leq C(C_i) \|\Delta h\|_{H^1_{\delta'+2}}\leq C(C_i)\sum_{|\alpha'|=1,2}\left\|e^{-\gamma^{(n)}}\nabla(\nabla^{\alpha'}h)\right\|_{L^2_{\delta'+1+|\alpha'|}}.
    \end{equation*}
    For the second term, we again use that $\nabla^{\alpha_1}\Tilde{\gamma}^{(n)},\nabla^{\alpha_1}(\chi\ln)\in C^0_1$ (since $|\alpha_1|=1$) to get 
    \begin{equation*}
        \left\|\nabla^{\alpha_1}\gamma^{(n)} \nabla^{\alpha_2}\Delta h \right\|_{L^2_{\delta'+3+3\e^2}}\leq C(C_i) \|\nabla^{\alpha_2}\Delta h\|_{L^2_{\delta'+2+
        3\e^2}}\leq C(C_i)\sum_{|\alpha'|=2}\left\|e^{-\gamma^{(n)}}\nabla(\nabla^{\alpha'}h)\right\|_{L^2_{\delta'+3}}.
    \end{equation*}
    The third term is easier to handle than the first one. Summarising, we get :
    \begin{equation}
        \sum_{|\alpha|\leq 2}\left\| N^{(n)}\left[\nabla^{\alpha},e^{-2\gamma^{(n)}}\right]\Delta h \right\|_{L^2_{\delta'+1+|\alpha|}}\leq C(C_i) \mathcal{E}^{(n)}[h].\label{second step}
    \end{equation}
    \item Third step is bounding $\left\| N^{(n)}\left[\left( \Ll^{(n)}\right)^2,\nabla^{\alpha}\right]h \right\|_{L^2_{\delta'+1+|\alpha|}}$. 
    Given the expression of $\mathcal{E}^{(n)}[h]$, we are allowed to bound this term by norms involving $\Ll^{(n)}\nabla^{\mu}$, $\nabla^{\nu}$ (for $|\mu|\leq 2$ and $|\nu|\leq 3$) and $\left( \Ll^{(n)}\right)^2$. The strategy is then to express $N^{(n)}\left[\left( \Ll^{(n)}\right)^2,\nabla \right]h$ and $N^{(n)}\left[\left( \Ll^{(n)}\right)^2,\nabla^2 \right]h$ in terms of those operators acting on $h$, using the commutation formula
    \begin{equation*}
        \left[\Ll^{(n)},\nabla\right] h =\frac{\nabla\beta^{(n)}}{N^{(n)}}\nabla h-\frac{\nabla N^{(n)}}{N^{(n)}}\Ll^{(n)}h.
    \end{equation*}
    Doing so, we find the following formula (we don't write the irrelevant numerical constants) :
    \begin{align}
        N^{(n)}\left[\left( \Ll^{(n)}\right)^2,\nabla \right]h & =\left(e_0^{(n)}\left(\frac{\nabla\beta^{(n)}}{N^{(n)}}\right) +\frac{\left(\nabla\beta^{(n)}\right)^2}{N^{(n)}} \right)\nabla h  +\left(\frac{\nabla\beta^{(n)}\nabla N^{(n)}}{N^{(n)}} +e_0^{(n)}\left( \frac{\nabla N^{(n)}}{N^{(n)}} \right) \right)\Ll^{(n)}h\label{commutateur 1}\\& \qquad\qquad+2\nabla\beta^{(n)}\Ll^{(n)}\nabla h -2\nabla N^{(n)}\left( \Ll^{(n)}\right)^2h.\nonumber
    \end{align}
    We recall that $\left( \Ll^{(n)}\right)^2h=e^{-2\gamma^{(n)}}\Delta h+f$, so that the $\left( \Ll^{(n)}\right)^2$ term in \eqref{commutateur 1} has already been estimate during the two first steps. The coefficients in front of $\nabla h$ and $\Ll^{(n)} h$ are all in $\mathcal{C}^1_0$ except the two involving $e_0\nabla N^{(n)}$, for wich we use the product law $H^1\times H^1$ and \eqref{HR N 2}  :
    \begin{align*}
    \l e_0^{(n)}\left( \frac{\nabla N^{(n)}}{N^{(n)}} \right)\Ll^{(n)}h \r_{L^2_{\delta'+2}} & \lesssim \l\nabla \dr_t N^{(n)} \r_{H^1_{\delta+1}} \l \Ll^{(n)} h \r_{H^1_{\delta'+1}} \lesssim C(C_i)\mathcal{E}^{(n)}[h].
    \end{align*}
    We only need to bound the coefficient in front of $\Ll^{(n)}\nabla$ in $L^{\infty}$, which is easily thanks to \eqref{HR N 3}. This allows us to handle the case $|\alpha|=1$ : \begin{align}
        &\left\| N^{(n)}\left[\left( \Ll^{(n)}\right)^2,\nabla^{\alpha}\right]h \right\|_{L^2_{\delta'+2}}\nonumber\\&\leq C(C_i)\left( \sum_{|\alpha'|\leq 1} \left( \left\Vert\frac{e_0^{(n)}\nabla^{\alpha'}h}{N^{(n)}}\right\Vert_{L^2_{\delta'+1+|\alpha'|}}+\left\| e^{-\gamma^{(n)}}\nabla (\nabla^{\alpha'}h)\right\|_{L^2_{\delta'+1+|\alpha'|}} \right)+\Vert fN^{(n)}\Vert_{H^2_{\delta'+1}}    \right).\label{alpha 1}
    \end{align}
    Before turning to the case $|\alpha|=2$, let's remark that, in view of \eqref{eq derivee}, so far we have proved that 
    \begin{equation}
        \left\| \left( \Ll^{(n)}\right)^2\nabla h\right\|_{L^2_{\delta'+2}}\leq C(C_i) \left( \mathcal{E}^{(n)}[h]+\Vert fN^{(n)}\Vert_{H^2_{\delta'+1}}\right).\label{L carré}
    \end{equation}
    This means that, even if  $\left( \Ll^{(n)}\right)^2\nabla h$ doesn't appear in the expression of $\mathcal{E}^{(n)}[h]$, we are allowed to use it in the sequel of the third step.
    
    We now turn to the case $|\alpha|=2$ and push our calculations further to get (we still don't write the irrelevant numerical constants) :
    \begin{align}
        & N^{(n)}\left[\left( \Ll^{(n)}\right)^2,\nabla^2 \right]h  =\nabla N^{(n)}\left( \Ll^{(n)}\right)^2\nabla h +\nabla N^{(n)}\left[\left( \Ll^{(n)}\right)^2,\nabla \right]h \nonumber\\
        & \left( N^{(n)}\nabla\Ll^{(n)}\left( \frac{\nabla\beta^{(n)}}{N^{(n)}}\right)+N^{(n)}\nabla\left(\left(\frac{\nabla\beta^{(n)}}{N^{(n)}}\right)^2\right) +\frac{\nabla N^{(n)}\left(\nabla\beta^{(n)}\right)^2}{N^{(n)} }+ \nabla\beta^{(n)}\Ll^{(n)}\left(\frac{\nabla N^{(n)}}{N^{(n)}}\right)\right)\nabla h \nonumber\\
        &+\left(  e_0^{(n)}\left(\frac{\nabla\beta^{(n)}}{N^{(n)}} \right)+\frac{\left(\nabla\beta^{(n)}\right)^2}{N^{(n)}}\right)\nabla^2 h +N^{(n)}\nabla\left(\frac{\nabla N^{(n)}}{N^{(n)}}\right) \left( \Ll^{(n)}\right)^2 h + 
        \nabla\beta^{(n)}\Ll^{(n)}\nabla^2 h\nonumber\\
        & +\left(   N^{(n)}\nabla\left(\frac{\nabla\beta^{(n)}\nabla N^{(n)}}{N^{(n)}}\right)+N^{(n)}\nabla\Ll^{(n)}\left( \frac{\nabla N^{(n)}}{N^{(n)}}\right)  +\frac{\nabla \beta^{(n)}\left(\nabla N^{(n)}\right)^2}{N^{(n)} }+ \nabla N^{(n)}\Ll^{(n)}\left(\frac{\nabla N^{(n)}}{N^{(n)}}\right) \right)\Ll^{(n)} h\nonumber\\
        & +\left( \nabla\beta^{(n)}\nabla N^{(n)} +e_0^{(n)}\left( \frac{\nabla N^{(n)}}{N^{(n)}}\right)+N^{(n)}\nabla\left( \frac{\nabla\beta^{(n)}}{N^{(n)}} \right)+\frac{\nabla\beta^{(n)}\nabla N^{(n)}}{N^{(n)}} \right)\Ll^{(n)}\nabla h\label{commutateur 2}.
    \end{align}

We need to estimate the $L^2_{\delta'+3}$ norm of \eqref{commutateur 2}. The term $\left( \Ll^{(n)}\right)^2 h$ has already been handled since $h$ satisfies \eqref{partie principale de box}. Since $\nabla N^{(n)}\in C^0_1$ we can use $\eqref{L carré}$ to estimate the term $\left( \Ll^{(n)}\right)^2\nabla h$. With the same argument, using \eqref{alpha 1} we handle the $\left[\left( \Ll^{(n)}\right)^2,\nabla \right]h$ term. Thanks to \eqref{HR beta 2} and \eqref{HR beta 3}, the coefficients in front of $\Ll^{(n)}\nabla^2h$ and $\nabla^2h$ are in the appropriate weighted $L^\infty$-based spaces ($L^\infty$ and $C^0_1$ respectively). The only problematic terms are the ones where two spatial derivatives hit $\beta^{(n)}$ or when at least one spatial derivative and $\Ll^{(n)}$ hit $N^{(n)}$. For them, we use the product estimate (see Proposition \ref{prop prod}). Let us give two examples, the first one using the embedding $H^1\times H^1\xhookrightarrow{}L^2$ (with appropriate weights) : 
\begin{align*}
\l \nabla^2\beta^{(n)} \Ll^{(n)}\nabla h \r_{L^2_{\delta'+3}} & \lesssim \l\nabla^2\beta^{(n)} \r_{H^1_{\delta'+2}} \l \Ll^{(n)}\nabla h \r_{H^1_{\delta'+2}} \lesssim C(C_i) \mathcal{E}^{(n)}[h].
\end{align*}
The second example uses the embedding $L^2\times H^2\xhookrightarrow{}L^2$ (with appropriate weights) :
\begin{align*}
\l \nabla \Ll^{(n)} \nabla N^{(n)}\Ll^{(n)}h\r_{L^2_{\delta'+3}} \lesssim \left( 1+ \l \nabla^2 \dr_t \Tilde{N}^{(n)}  \r_{L^2_{\delta+2}} \right) \l \Ll^{(n)} h \r_{H^2_{\delta'+1}} \lesssim C(C_i) \mathcal{E}^{(n)}[h].
\end{align*}
This allows us to handle entirely the case $|\alpha|=2$. Summarising the third step, we get : 
\begin{equation}
    \sum_{|\alpha|\leq 2}\left\| N^{(n)}\left[\left( \Ll^{(n)}\right)^2,\nabla^{\alpha}\right]h \right\|_{L^2_{\delta'+1+|\alpha|}} \leq C(C_i)\left( \mathcal{E}^{(n)}[h]+\Vert fN^{(n)}\Vert_{H^2_{\delta'+1}}\right).\label{third step}
\end{equation}
\end{itemize}
Combining \eqref{first step}, \eqref{second step} and \eqref{third step}, we get for all $t\in [0,T]$ :
\begin{equation}
    \mathcal{E}^{(n)}[h](t)\leq 2 \mathcal{E}^{(n)}[h](0)+C(C_i)\sqrt{T}\left( \sup_{s\in[0,T]}\Vert fN^{(n)}\Vert_{H^2_{\delta'+1}}(s)+\mathcal{E}^{(n)}[h](t)\right).
\end{equation}

By choosing $T$ sufficiently small, we can absorb the term $\mathcal{E}^{(n)}[h](t)$ of the RHS into the LHS and conclude the proof of the lemma. 
\end{proof}

With this energy estimate, we are ready to prove estimates on $\Tilde{\gamma}^{(n+1)}$, $\ffi^{(n+1)}$ and $\omega^{(n+1)}$. The spatial term in the energy $\mathcal{E}^{(n)}[h]$ is different from what appear in \eqref{HR gamma 1}, but \eqref{propsmall gamma} implies that $1\lesssim e^{-\gamma^{(n)}}$ so we will get back the estimates we want for $\Tilde{\gamma}^{(n+1)}$, $\ffi^{(n+1)}$ and $\omega^{(n+1)}$ if we bound $\mathcal{E}^{(n)}$, using Lemma \ref{inegalite d'energie lemme}.

\subsubsection{Hyperbolic estimates} 

We use our energy estimate to prove the estimates on $\Tilde{\gamma}^{(n+1)}$. Since we are not getting $\gamma^{(n+1)}$ from an elliptic equation, we cannot obtain the decomposition $\gamma^{(n+1)}=-\alpha\chi\ln+\Tilde{\gamma}^{(n+1)}$ directly. Our strategy is to solve for $\gamma^{(n+1)}+\alpha\chi\ln$ to artificially recover our decomposition after having set $\Tilde{\gamma}^{(n+1)}\vcentcolon=\gamma^{(n+1)}+\alpha\chi\ln$. For the sake of clarity, we gather in the following lemma the estimates of the extra terms due to $\alpha\chi\ln$ :

\begin{lem}\label{op chi ln lemme}
We set $\Psi^{(n)}= N^{(n)}\left(\left(\Ll^{(n)}\right)^2(\chi\ln)- e^{-2\gamma^{(n)}}\Delta (\chi\ln) \right)$.
For $n\geq 2$, the following estimate hold :
\begin{align}
    \left\| \Psi^{(n)} \right\|_{H^2_{\delta'+1}}&\leq C(C_i),\label{op chi ln 1}\\
    \left\| \Psi^{(n)} \right\|_{H^1_{\delta'+1}}&\leq A_0C_i.\label{op chi ln 3}
\end{align}
\end{lem}

\begin{proof}
To prove \eqref{op chi ln 1}, we don't need to be very precise about the dependence on $C_i$ of the bound, so we don't give many details. For the first part of $\Psi^{(n)}$ :
\begin{align*}
    \left\| N^{(n)}\left(\Ll^{(n)}\right)^2(\chi\ln)\right\|_{H^2_{\delta'+1}} & \lesssim \left\| \dr_t\left( \frac{\beta^{(n)}}{N^{(n)}}\right) \right\|_{H^2_{\delta'}}+\left\|\beta^{(n)} \right\|_{H^2_{\delta'}}\left( \left\|\frac{\beta^{(n)}}{N^{(n)}} \right\|_{H^2_{\delta'-1}}+\left\|\nabla\left( \frac{\beta^{(n)}}{N^{(n)}} \right) \right\|_{H^2_{\delta'}} \right)\\
    &\leq C(C_i).
\end{align*}
For the second part of $\Psi^{(n)}$,  we notice that $\Delta(\chi\ln)=\Delta(\chi)\ln+\nabla\chi\cdot\nabla\ln$ is a smooth compactly supported function (its support is included in $B_2$) and in particular belongs to all $C^k$ spaces, using \eqref{useful gamma 1} and \eqref{HR N 1} :
\begin{align*}
     \left\|e^{-2\gamma^{(n)}} N^{(n)}\Delta(\chi\ln) \right\|_{H^2_{\delta'+1}} & \lesssim\left\|\Delta(\chi\ln) \right\|_{C^2} \left\| N^{(n)} \right\|_{H^2(B_2)}\lesssim 1.
\end{align*}
 
We now turn to the proof of \eqref{op chi ln 3}. For the first part of $\Psi^{(n)}$, we use \eqref{HR beta 1} (and the product estimate $H^2_{\delta'}\times H^1_{\eta}\xhookrightarrow{}H^1_{\eta}$), \eqref{HR N 1}, \eqref{HR N 2} and \eqref{HR beta 3}, and actually the only term that will bring some $C_i$ are $\dr_t\beta^{(n)}$ and $\dr_t N^{(n)}$ :
\begin{align*}
    \left\| N^{(n)}\left(\Ll^{(n)}\right)^2(\chi\ln)\right\|_{H^1_{\delta'+1}} & \lesssim \left\|\dr_t\left(\frac{\beta^{(n)}}{N^{(n)}} \right)\right\|_{H^1_{\delta'}}+\left\|\beta^{(n)} \right\|_{H^2_{\delta'}}\left( \left\|\frac{\beta^{(n)}}{N^{(n)}} \right\|_{H^1_{\delta'-1}}+\left\|\nabla\left( \frac{\beta^{(n)}}{N^{(n)}} \right) \right\|_{H^1_{\delta'}} \right)
    \\&\lesssim A_0C_i+\e.
\end{align*}
For the second part of $\Psi^{(n)}$, we again use the properties of $\Delta(\chi\ln)$, \eqref{useful gamma 1} and \eqref{HR N 1} :
\begin{align*}
    \left\|e^{-2\gamma^{(n)}} N^{(n)}\Delta(\chi\ln) \right\|_{H^1_{\delta'+1}}
    &\lesssim \left\|\Delta(\chi\ln) \right\|_{C^1} \left\| N^{(n)} \right\|_{H^1(B_2)}\lesssim 1.
\end{align*}

\end{proof}

\begin{prop}\label{hr+1 gamma prop}
For $n\geq 2$ the following estimates hold :
\begin{align}
    \sum_{|\alpha|\leq 2}  \left\Vert\Ll^{(n)}\nabla^{\alpha}\Tilde{\gamma}^{(n+1)}\right\Vert_{L^2_{\delta'+1+|\alpha|}}+\left\|\nabla\Tilde{\gamma}^{(n+1)}\right\|_{H^2_{\delta'+1}}&\leq 8 C_i,\label{HR+1 gamma 1}\\ 
    \left\| \dr_t\left(\Ll^{(n)}\Tilde{\gamma}^{(n+1)}\right)
    \right\|_{L^2_{\delta'+1}}&\lesssim C_i,\label{HR+1 gamma 2}\\
    \left\| \dr_t\left(\Ll^{(n)}\Tilde{\gamma}^{(n+1)}\right)\right\|_{H^1_{\delta'+1}}&\lesssim 
    A_1C_i.\label{HR+1 gamma 3}
  \end{align}
\end{prop}

\begin{proof}
The strategy is to recover the decomposition of $\gamma^{(n+1)}$ by setting $\Tilde{\gamma}^{(n+1)}\vcentcolon=\gamma^{(n+1)}+\alpha\chi\ln$. In view of \eqref{reduced system gamma}, $\Tilde{\gamma}^{(n+1)}$ is solution of 
\begin{equation}
    \left(\Ll^{(n)}\right)^2\Tilde{\gamma}^{(n+1)}-e^{-2\gamma^{(n)}}\Delta \Tilde{\gamma}^{(n+1)}=\left( \text{RHS of \eqref{reduced system gamma}} \right)+ \alpha\frac{\Psi^{(n)}}{N^{(n)}}.\label{eq gamma tilde}
\end{equation}
In order to prove \eqref{HR+1 gamma 1} and in view of \eqref{inegalite d'energie equation}, we have to bound $\left\| \text{RHS of \eqref{eq gamma tilde}}\times N^{(n)} \right\|_{H^2_{\delta'+1}}$, and thanks to the factor $\sqrt{T}$, we don't need to worry about the bounds. Thanks to \eqref{op chi ln 1}, it remains to deal with the RHS of \eqref{reduced system gamma} mutltiplied by $N^{(n)}$, which gives the following expression :
\begin{align*}
    -\frac{N^{(n)}\left(\tau^{(n)}\right)^2}{2}+\frac{1}{2}e_0^{(n-1)}\left(\frac{\dive(\beta^{(n)})}{N^{(n-1)}}\right)&+e^{-2\gamma^{(n)}}\frac{\Delta N^{(n)}}{2}+e^{-2\gamma^{(n)}}N^{(n)}|\nabla\ffi^{(n)}|^2\\&+\frac{1}{4}e^{-2\gamma^{(n)}-4\ffi^{(n)}}N^{(n)} |\nabla\omega^{(n)}|^2=\vcentcolon I+II+III+IV+V.
\end{align*}
\begin{itemize}
    \item For $I$, we mainly use the fact that $H^2_{\delta'+1}$ is an algebra and \eqref{HR tau 1} (and the product estimate to deal with $\chi\ln$) :
    \begin{equation*}
        \|I\|_{H^2_{\delta'+1}}\lesssim \left\| \tau^{(n)}\right\|^2_{H^2_{\delta'+1}}\left( 1+ \left\|\Tilde{N}^{(n)}\right\|_{H^2_{\delta'+1}}\right)\lesssim C(C_i).
    \end{equation*}
    
    \item For $II$, we use the computations already performed about $\left[\Ll^{(n-1)},\nabla \right]$, \eqref{HR N 2} (which implies that $\left|\frac{1}{N^{(n-1)}}\right|$ and  $\left\|\nabla N^{(n-1)}\right\|_{C^1}$ are bounded) to get rid of the $\frac{1}{N^{(n-1)}}$ factors, in order to get :
    \begin{align}
        \| II\|_{H^2_{\delta'+1}} & \lesssim \left\|e_0^{(n-1)}N^{(n-1)} \nabla\beta^{(n)} \right\|_{H^2_{\delta'+1}}+\left\| \Ll^{(n-1)}\beta^{(n)} \right\|_{H^3_{\delta'}}\label{II}\\&\qquad +\left\|\nabla N^{(n-1)}\Ll^{(n-1)}\beta^{(n)} \right\|_{H^2_{\delta'+1}}+\left\|\nabla\beta^{(n-1)}\nabla\beta^{(n)} \right\|_{H^2_{\delta'+1}}\nonumber.
    \end{align}
    Using \eqref{HR beta 4}, it's easy to see that $\left\| \Ll^{(n-1)}\beta^{(n)} \right\|_{H^3_{\delta'}}\leq C(C_i)$, and we recall the embedding $H^3_{\delta'}\xhookrightarrow{}H^2_{\delta'+1}$. Using $|\nabla(\chi\ln)|\lesssim \langle x\rangle^{-1}$ and \eqref{HR N 2}, we see that $\left\|\nabla N^{(n-1)}\right\|_{ H^3_{\delta}}\lesssim C_i$ and thus we use the product estimate to write :
    \begin{equation*}
        \left\| \Ll^{(n-1)}\beta^{(n)} \right\|_{H^3_{\delta'}}+\left\|\nabla N^{(n-1)}\Ll^{(n-1)}\beta^{(n)} \right\|_{H^2_{\delta'+1}} \lesssim \left\| \Ll^{(n-1)}\beta^{(n)} \right\|_{H^3_{\delta'}}\left( 1+\left\|\nabla N^{(n-1)} \right\|_{H^2_{\delta}} \right) \leq C(C_i).
    \end{equation*}
    Using $\nabla\beta^{(n)}\in H^3_{\delta'+1}\xhookrightarrow{}H^2_{\delta'+2}$, $\eqref{HR N 2}$, and $\nabla N^{(n-1)}\in H^2_{\delta}$ :
    \begin{align*}
        \left\|e_0^{(n-1)}N^{(n-1)} \nabla\beta^{(n)} \right\|_{H^2_{\delta'+1}}& \lesssim \left| \dr_t N_a^{(n-1)}\right|\left\| \chi\ln \nabla\beta^{(n)} \right\|_{H^2_{\delta'+1}}+\left\|\dr_t\Tilde{N}^{(n-1)}\nabla\beta^{(n)} \right\|_{H^2_{\delta'+1}}\\&\qquad+\left\|\beta^{(n-1)}\nabla N^{(n-1)}\nabla\beta^{(n)} \right\|_{H^2_{\delta'+1}}
        \\& \lesssim\left(\left| \dr_t N_a^{(n-1)}\right|+\left\|\dr_t\Tilde{N}^{(n-1)}\right\|_{H^2_{\delta}}\right)\left\|
        \nabla\beta^{(n)} \right\|_{H^2_{\delta'+2}}\\&\qquad+\left\|\beta^{(n-1)}\right\|_{H^2_{\delta'+2}}\left\|\nabla N^{(n-1)}\right\|_{H^2_{\delta}}\left\|\nabla\beta^{(n)} \right\|_{H^2_{\delta'+2}} \\& \leq C(C_i).
    \end{align*}
    The last term in \eqref{II} doesn't present any difficulty and we get $\| II\|_{H^2_{\delta'+1}}\leq C(C_i)$.
    
    \item For $III$, we first notice that $\Delta(\chi\ln)=\Delta(\chi)\ln+\nabla\chi\cdot\nabla\ln$ is a smooth compactly supported function and therefore belongs to all $H^k$ spaces.
    We then use \eqref{useful gamma 1} and \eqref{HR N 3} :
    \begin{align*}
        \|III\|_{H^2_{\delta'+1}} &\lesssim \left| N_a^{(n)}\right|\left\|\Delta(\chi\ln) \right\|_{H^2}+\left\|\Delta\Tilde{N}^{(n)} \right\|_{H^2_{\delta+1}}\leq C(C_i).
    \end{align*}
    \item For $IV$, thanks to the support property of $\ffi^{(n)}$, we don't worry about the decrease of our functions. We simply use \eqref{useful gamma 1}, \eqref{HR N 3} (which implies that $\left\| N^{(n)}\right\|_{C^2(B_{2R})}\leq C(C_i)$) and the fact that $H^2$ is an algebra :
    \begin{align*}
        &\| IV\|_{H^2}\lesssim \left\| N^{(n)}\right\|_{C^2(B_{2R})} \left\|\nabla\ffi^{(n)} \right\|_{H^2}^2\leq C(C_i).
    \end{align*}
    \item For $V$, we do as for $IV$, using in addition \eqref{useful ffi 1}, which implies that it remains to deal with the following term :
    \begin{align*}
    \l\nabla^2\ffi^{(n)}|\nabla\omega^{(n)}|^2 \r_{L^2} \lesssim \l \nabla^2\ffi^{(n)}\r_{L^2}\l\nabla\omega^{(n)} \r_{H^2}^2\leq C(C_i).
    \end{align*}
\end{itemize}
Using Lemma \ref{inegalite d'energie lemme}, we get for all $t\in [0,T]$ :
\begin{equation*}
    \mathcal{E}^{(n)}\left[ \Tilde{\gamma}^{(n+1)} \right](t)\leq 3\, \mathcal{E}^{(n)}\left[ \Tilde{\gamma}^{(n+1)} \right](0)+C(C_i)\sqrt{T}.\label{blabla}
\end{equation*}
It remains to show that $\mathcal{E}^{(n)}\left[ \Tilde{\gamma}^{(n+1)} \right](0)$ is bounded by $C_i$. Therefore, the following calculations will be performed on $\Sigma_0$ and we can forget about the indices $(n)$ or $(n+1)$ and use the estimates \eqref{CI petit} and \eqref{CI gros}, which are more comfortable.
Using the calculations performed in Proposition \ref{commutation estimate}, we show that :
\begin{align}
     \sum_{|\alpha|\leq 2}  \left\Vert\Ll\nabla^{\alpha}\Tilde{\gamma}\right\Vert_{L^2_{\delta'+1+|\alpha|}} \leq (1+C\e)\| \Ll\Tilde{\gamma}\|_{H^2_{\delta+1}}+C\e\mathcal{E}\left[ \Tilde{\gamma} \right](0)+C\e\|\Tilde{\gamma}\|_{H^4_{\delta}}.\label{E0 temps} 
\end{align}
Using the same ideas as in the second step of the proof of Lemma \ref{inegalite d'energie lemme}, we show that :
\begin{align}
    \sum_{|\alpha|\leq 2}  \left\Vert e^{-\gamma}\nabla( \nabla^{\alpha}\Tilde{\gamma})\right\Vert_{L^2_{\delta'+1+|\alpha|}}\leq (1+C\e)\|\Tilde{\gamma}\|_{H^4_{\delta}}.\label{E0 espace}
\end{align}
Putting together \eqref{E0 temps} and \eqref{E0 espace} and using \eqref{CI petit} and \eqref{CI gros}, we get 
\begin{equation*}
    \mathcal{E}\left[ \Tilde{\gamma} \right](0) \leq (1+C\e) \left( \| \Ll\Tilde{\gamma}\|_{H^2_{\delta+1}} +\|\Tilde{\gamma}\|_{H^4_{\delta}}\right)+C\e\mathcal{E}\left[ \Tilde{\gamma} \right](0) \leq2 (1+C\e)C_i +C\e\mathcal{E}\left[ \Tilde{\gamma} \right](0).
\end{equation*}
We can absorb the last term of the RHS into the LHS by choosing $\e$ small enough. Taking $T$ small enough and remembering that $1\lesssim e^{-\gamma^{(n)}}$, we finish the proof of \eqref{HR+1 gamma 1}.
\par\leavevmode\par
We now turn to the proof of \eqref{HR+1 gamma 2} and \eqref{HR+1 gamma 3} which amounts to estimating $\dr_t\left(\Ll^{(n)}\Tilde{\gamma}^{(n+1)}\right)$, which, thanks to \eqref{eq gamma tilde}, has the following expression :
\begin{align}
    \dr_t\left(\Ll^{(n)}\Tilde{\gamma}^{(n+1)}\right)&=e^{-2\gamma^{(n)}}N^{(n)}\Delta\Tilde{\gamma}^{(n+1)}-\frac{N^{(n)}\left(\tau^{(n)}\right)^2}{2}+\frac{1}{2}e_0^{(n-1)}\left(\frac{\dive(\beta^{(n)})}{N^{(n-1)}}\right)\\&\qquad+e^{-2\gamma^{(n)}}\frac{\Delta N^{(n)}}{2}+e^{-2\gamma^{(n)}}N^{(n)}\left|\nabla\ffi^{(n)}\right|^2\\&\qquad+\frac{1}{4}e^{-2\gamma^{(n)}-4\ffi^{(n)}}N^{(n)} |\nabla\omega^{(n)}|^2+ \alpha \Psi^{(n)}\\&=\vcentcolon I+II+III+IV+V+VI+VII.
\end{align}
The term $VII$ is handled thanks to \eqref{op chi ln 3}. For the remainings terms, we first bound their $L^2_{\delta'+1}$ norms with $C_i$, and then the $H^1_{\delta'+1}$ norms of their derivatives by $C_i^2$.
\begin{itemize}
    \item For $I$, we first perform the $H^1$ estimate, using \eqref{useful gamma 1}, $\e|\chi\ln|\lesssim \langle x \rangle^{\e}$, \eqref{HR N 1}  and \eqref{HR+1 gamma 1}  :
    \begin{align*}
        \| I\|_{H^1_{\delta'+1}}&\lesssim \left\| N^{(n)}\Delta\Tilde{\gamma}^{(n+1)}\right\|_{H^1_{\delta'+1}}\lesssim \left\|\Delta\Tilde{\gamma}^{(n+1)} \right\|_{H^1_{\delta'+2}}\left(1+\left\|\Tilde{N}^{(n)} \right\|_{H^2_{\delta}} \right)\lesssim C_i.
    \end{align*}
    To get the $L^2$ estimate, we simply use the embeddings $H^1_{\delta'+1}\xhookrightarrow{}L^2_{\delta'+1}$.

    \item For $II$, we first use \eqref{HR N 1} and \eqref{propsmall tau} :
    \begin{equation*}
        \|II\|_{L^2_{\delta'+1}}\lesssim \left\|\tau^{(n)}\right\|_{H^1_{\delta'+1}}^2\left(1+ \l \Tilde{N}^{(n)}\r_{H^2_{\delta}} \right)\lesssim \e^2.
    \end{equation*}
    For the $H^1$ estimate, we use \eqref{HR N 1}, \eqref{propsmall tau} and \eqref{HR tau 1} :
    \begin{equation*}
        \|II\|_{H^1_{\delta'+1}}\lesssim \left\|\tau^{(n)}\right\|_{H^1_{\delta'+1}}\left\|\tau^{(n)}\right\|_{H^2_{\delta'+1}}\left(1+ \l \Tilde{N}^{(n)}\r_{H^2_{\delta}} \right)\lesssim \e A_1C_i.
    \end{equation*}
    \item For $III$, we use \eqref{HR beta 2.5}, \eqref{HR beta 1} and \eqref{HR N 2} :
    \begin{align*}
         \l III \r_{L^2_{\delta'+1}} &\lesssim \l\nabla\dr_t\beta^{(n)} \r_{L^2_{\delta'+1}} + \l \nabla\beta^{(n)}\r_{H^1_{\delta'+1}}\l\dr_t\Tilde{N}^{(n-1)}\r_{H^2_{\delta}}+ \l\beta^{(n-1)}\r_{H^2_{\delta'}}\l\nabla\beta^{(n)} \r_{H^1_{\delta'+1}}\\&\lesssim C_i.
    \end{align*}
    For the $H^1$ estimate, we use :
    \begin{align*}
        \l III \r_{H^1_{\delta'+1}}& \lesssim \l \nabla\dr_t\beta^{(n)}\r_{H^1_{\delta'+1}} + \l\nabla\beta^{(n)}\dr_t\Tilde{N}^{(n-1)} \r_{H^1_{\delta'+1}} + \l \beta^{(n-1)}\nabla\beta^{(n)}\r_{H^1_{\delta'+1}} \\
        &\lesssim \l \nabla\dr_t\beta^{(n)}\r_{H^1_{\delta'+1}} + \l\nabla\beta^{(n)}\r_{H^1_{\delta'+1}} \l \dr_t\Tilde{N}^{(n-1)}\r_{H^2_{\delta}} +\l\beta^{(n-1)} \r_{H^2_{\delta'}} \l\nabla\beta^{(n)} \r_{H^1_{\delta'+1}}
        \\&\lesssim A_1 C_i + \e C_i +\e^2.
    \end{align*}
    \item For $IV$, we recall that $\Delta(\chi\ln)$ is a smooth compactly supported function. We only perform the $H^1$ estimate, because the $L^2$ estimate will be a consequence of the embedding $H^1_{\delta'+1}\xhookrightarrow{}L^2_{\delta'+1}$. We use \eqref{useful gamma 1} and \eqref{HR N 2} :
    \begin{equation*}
         \| IV\|_{H^1_{\delta'+1}} \lesssim \left| N_a^{(n)}\right|\left\|\Delta(\chi\ln) \right\|_{H^1}+\left\|\Delta\Tilde{N}^{(n)}\right\|_{H^1_{\delta+1}}\lesssim C_i.
    \end{equation*}
    \item For $V$, we don't care about the decrease of our functions, thanks to the support property of $\ffi^{(n)}$. We first use \eqref{useful gamma 1} and the fact that $N^{(n)}\in L^{\infty}(B_{2R})$ (which comes from \eqref{HR N 1}), the Hölder's inequality and \eqref{propsmall fi} :
    \begin{equation*}
        \| V\|_{L^2}\leq \left\|N^{(n)} \right\|_{L^{\infty}(B_{2R})} \left\|\nabla\ffi^{(n)} \right\|_{L^4}^2\lesssim \e^2.
    \end{equation*}
    For the $H^1$ estimate, we use \eqref{useful gamma 1}, \eqref{HR N 1}, \eqref{HR N 2} and \eqref{HR fi 1} :
    \begin{align*}
        \|  \nabla V\|_{L^2}&\lesssim \l\nabla\ffi^{(n)}\nabla^2\ffi^{(n)} \r_{L^2}+ \l\nabla\Tilde{N}^{(n)} \left|\nabla\ffi^{(n)}\right|^2\r_{L^2}
        \\& \lesssim \l \nabla^2\ffi^{(n)} \r_{H^1} \l \nabla\ffi^{(n)}\r_{L^4} + \l \nabla\Tilde{N}^{(n)}\r_{H^2_{\delta+1}}\left\|\nabla\ffi^{(n)} \right\|_{L^4}^2
        \\& \lesssim\e A_0C_i.
    \end{align*}
    \item For $VI$, we do as for $V$, using in addition \eqref{useful ffi 1}.
\end{itemize}
\end{proof}

We are now interesting in proving estimates for $\ffi^{(n+1)}$ and $\omega^{(n+1)}$. We first prove their support property :

\begin{lem}\label{support fi n+1}
There exists $C_s>0$ such that for $\e$, $T$ sufficiently small (depending on $R$), $\ffi^{(n+1)}$ is supported in
\begin{equation*}
    \enstq{(t,x)\in[0,T]\times\R^2}{\vert x\vert\leq R+C_s(1+R^{\e})t}.
\end{equation*}
In particular, choosing $T$ small enough, $\mathrm{supp}(\ffi^{(n+1)})\subset [0,T]\times B_{2R}$.
\end{lem}

\begin{proof}
Since the initial data for $\ffi^{(n+1)}$ and $\dr_t\ffi^{(n+1)}$ are compactly supported and $\Box_{g^{(n)}}\ffi^{(n+1)}$ is compactly supported in  
\begin{equation*}
A\vcentcolon=\enstq{(t,x)\in[0,T]\times\R^2}{\vert x\vert \leq R+C_s(1+R^{\e})t}
\end{equation*}
we juste have to show that $\partial A$ is a spacelike hypersurface. We set $f(x,t)=-|x|+C_s(1+R^{\e})t$, in order to have $\partial A=f^{-1}(-R)$. Thus, we have to show that $(g^{(n)})^{-1}(\d f,\d f)$ is non-positive on this hypersurface. We have $\d f=-\frac{x_i}{|x|}\d x^i+C_s(1+R^{\e})\d t$, which implies :
\begin{align*}
    \left(g^{(n)}\right)^{-1}(\d f,\d f)&=\frac{x_ix_j}{|x|^2}\left(g^{(n)}\right)^{ij}+C_s^2(1+R^{\e})^2\left(g^{(n)}\right)^{tt}-\frac{2x_i}{|x|}C_s(1+R^{\e})\left(g^{(n)}\right)^{it}\\
    &=e^{-2\gamma^{(n)}}-\left(\frac{x\cdot\beta^{(n)}}{|x|N^{(n)}}\right)^2-\left( \frac{C_s(1+R^{\e})}{N^{(n)}}\right)^2-\frac{2(x\cdot\beta^{(n)})C_s(1+R^{\e})}{(N^{(n)})^2|x|}
\end{align*}
We have $e^{-2\gamma^{(n)}}\lesssim \langle x\rangle^{2\e^2}$, $\frac{|x\cdot\beta^{(n)}|}{(N^{(n)})^2|x|}+\left(\frac{x\cdot\beta^{(n)}}{|x|N^{(n)}}\right)^2\lesssim \e$, so choosing the parameters appropriately, one easily sees that $(g^{(n)})^{-1}(\d f,\d f)$ is non-positive on the hypersurface.
\end{proof}

\begin{prop}\label{hr+1 fi prop}
For $n\geq 2$, the following estimates holds :
\begin{align}
     \left\|\dr_t\ffi^{(n+1)}\right\|_{H^2}+\left\|\nabla\ffi^{(n+1)}\right\|_{H^2} & \lesssim C_i ,\label{esti ondes phi 1}\\
     \left\Vert \dr_t\left( \Ll^{(n)}\ffi^{(n+1)} \right) \right\Vert_{H^1} & \lesssim C_i.\label{esti ondes phi 2}
\end{align}
\end{prop}

\begin{proof}
First, note that since $\ffi^{(n+1)}$ is compactly supported in $B_{2R}$ for all time by previous lemma, we do not need to worry about the spatial decay in this proof. We recall the wave equation satisfied by $\ffi^{(n+1)}$ :
\begin{align*}
    \left( \Ll^{(n)} \right)^2\ffi^{(n+1)}-e^{-2\gamma^{(n)}}\Delta \ffi^{(n+1)}&=\frac{e^{-2\gamma^{(n)}}}{N^{(n)}}\nabla \ffi^{(n)}\cdot\nabla N^{(n)}+\frac{\tau^{(n)} e_0^{(n-1)}\ffi^{(n)}}{N^{(n)}}\\&\qquad+\frac{1}{2}e^{-4\ffi^{(n)}}\left( \left(e_0^{(n-1)}\omega^{(n)}\right)^2+\left|\nabla\omega^{(n)}\right|^2 \right).
\end{align*}
Using our energy estimate for this wave equation we see that to prove the first part of the proposition we have to bound 
\begin{equation*}
    \left\Vert e^{-2\gamma^{(n)}}\nabla \ffi^{(n)}\cdot\nabla N^{(n)}\right\Vert_{H^2}+\left\Vert \tau^{(n)} e_0^{(n-1)}\ffi^{(n)}\right\Vert_{H^2}+\l e^{-4\ffi^{(n)}}N^{(n)}(e_0^{(n-1)}\omega^{(n)})^2 \r_{H^2} +\l e^{-4\ffi^{(n)}}N^{(n)} |\nabla\omega^{(n)}|^2\r_{H^2} .
\end{equation*}
We mainly use the fact that in dimension 2, $H^2$ is an algebra. Noting that every norm is not taking on the whole space but only on $B_{2R}$, using \eqref{useful gamma 1} and \eqref{useful ffi 1} and thanks to the estimates made on the $n$-th iterate, it's easy to see that this quantity is bounded by some constant $C(A_i,C_i)$. We also recall that :
\begin{equation*}
    \mathcal{E}^{(n)}\left[ \ffi^{(n+1)} \right](0)\lesssim C_i.
\end{equation*}
Thanks to the Lemma \ref{inegalite d'energie lemme}, if $T$ is small enough, we have for all $t\in[0,T]$
\begin{equation*}
    \mathcal{E}^{(n)}\left[ \ffi^{(n+1)} \right](t)\lesssim C_i.
\end{equation*}
Thanks to the support property of $\ffi^{(n+1)}$, the fact that $1\lesssim e^{-\gamma^{(n)}}$ and $|N^{(n)}|\lesssim 1$ (on $B_{2R}$), we have :
\begin{equation}
     \l\dr_t\ffi^{(n+1)}\r_{H^2}+\l\nabla\ffi^{(n+1)}\r_{H^2}\lesssim \mathcal{E}^{(n)}\left[ \ffi^{(n+1)} \right]+C_i.\label{commutation energie}
\end{equation}
which concludes the proof of \eqref{esti ondes phi 1}.
\par\leavevmode\par
We next prove the estimate \eqref{esti ondes phi 2}. We use the equation satisfied by $\ffi^{(n+1)}$ to express the term we want to estimate:
\begin{align*}
    \dr_t\left( \Ll^{(n)}\ffi^{(n+1)} \right) & =e^{-2\gamma^{(n)}}N^{(n)}\Delta\ffi^{(n+1)}+e^{-2\gamma^{(n)}}\nabla \ffi^{(n)}\cdot\nabla N^{(n)}+\tau^{(n)} e_0^{(n-1)}\ffi^{(n)}+\beta^{(n)}\cdot\nabla\left( \Ll^{(n)}\ffi^{(n+1)} \right)\\&\qquad+\frac{1}{2}e^{-4\ffi^{(n)}}N^{(n)} \left(e_0^{(n-1)}\omega^{(n)}\right)^2+\frac{1}{2}e^{-4\ffi^{(n)}}N^{(n)}\left|\nabla\omega^{(n)}\right|^2 \\
    & =\vcentcolon I+II+III+IV+V+VI.
\end{align*}
Thus, it remains to bound those terms by $C_i$ in $H^1$, and the main difficulty is avoiding any $C_i^2$ bound. We mainly use the embedding of $H^1$ in $L^q$ for all $q\geq 2$ and the Hölder inequality, in particular the $L^4\times L^4\xhookrightarrow{}L^2$ and $L^8\times L^8\xhookrightarrow{}L^4$ case (note that in the following we do not write down the factors that are trivially in $L^{\infty}$) :
\begin{itemize}[label=\textbullet]
    \item for $I$, the only issues are the terms where $\Tilde{N}^{(n)}$ or $\Tilde{\gamma}^{(n)}$ get one derivative :
    \begin{align*}
        \|I\|_{H^1} & \lesssim \Vert\nabla\ffi^{(n+1)}\Vert_{H^2}+\|\nabla\Tilde{N}^{(n)}\Delta\ffi^{(n+1)}\|_{L^2}+\|\nabla\Tilde{\gamma}^{(n)}\Delta\ffi^{(n+1)}\|_{L^2}\\
        & \lesssim \Vert\nabla\ffi^{(n+1)}\Vert_{H^2}\left(1+\|\nabla\Tilde{N}^{(n)}\|_{H^1}+\|\nabla\Tilde{\gamma}^{(n)}\|_{H^1}\right)\\
        & \lesssim C_i.
    \end{align*}
    \item for $II$, we forget about the $\chi\ln$ in $N^{(n)}$, which is less problematic than $\Tilde{N}^{(n)}$ :
    \begin{align*}
        \|II\|_{H^1} & \lesssim\|\nabla\ffi^{(n)}\cdot\nabla N^{(n)}\|_{L^2}+\|\nabla^2\ffi^{(n)}\nabla N^{(n)}\|_{L^2}+\|\nabla\ffi^{(n)}\nabla^2 N^{(n)}\|_{L^2}+\|\nabla\Tilde{\gamma}^{(n)}\nabla\ffi^{(n)}\nabla N^{(n)}\|_{L^2}\\
        & \lesssim\|\nabla\ffi^{(n)}\|_{L^4}\|\Tilde{N}^{(n)}\|_{H^4}+\|\nabla\ffi^{(n)}\|_{H^2}\|\Tilde{N}^{(n)}\|_{H^2}\left( 1+\|\Tilde{\gamma}^{(n)}\|_{H^2}\right)\\
        & \lesssim C_i.
    \end{align*}
    \item for $III$, we use \eqref{HR tau 1} when no derivatives hits $e_0^{(n-1)}\ffi^{(n)}$ and \eqref{propsmall tau} when one derivative hits $e_0^{(n-1)}\ffi^{(n)}$ :
    \begin{align*}
        \| III\|_{H^1}&\lesssim \|\tau^{(n)}\|_{H^2}\left(  \Vert\dr_t\ffi^{(n)}\Vert_{L^4}+\Vert\nabla\ffi^{(n)}\Vert_{L^4}\left( 1+\left\| \nabla\beta^{(n-1)}\right\|_{H^1}\right) \right)\\& \quad+\|\tau^{(n)}\|_{H^1}\left( \|\nabla^2\ffi^{(n)}\|_{H^1}+\|\dr_t\nabla\ffi^{(n)}\|_{H^1}\right)\\& \lesssim C_i.
    \end{align*}
    \item for $IV$, we just notice that, applying the same type of arguments as in Proposition \ref{commutation estimate}, it's easy to deduce from the first part of this proof that $\left\|\Ll^{(n)}\ffi^{(n+1)}\right\|_{H^2}\lesssim C_i$ :
    \begin{align*}
        \|IV\|_{H^1} & \lesssim \left\|\Ll^{(n)}\ffi^{(n+1)}\right\|_{H^1}+ \left\|\nabla\beta^{(n)}\nabla\left( \Ll^{(n)}\ffi^{(n+1)} \right)\right\|_{L^2}+\left\|\beta^{(n)}\nabla^2\left( \Ll^{(n)}\ffi^{(n+1)} \right)\right\|_{L^2}\\
        & \lesssim \left\|\Ll^{(n)}\ffi^{(n+1)}\right\|_{H^2}\left( 1+\|\nabla\beta^{(n)}\|_{H^1}\right)\\
        & \lesssim C_i.
    \end{align*}
    \item for $V$, we use first \eqref{useful ffi 1} and the fact that $N^{(n)}$ and $\nabla N^{(n)}$ are bounded, and then \eqref{propsmall fi} and \eqref{HR omega 1} :
    \begin{align*}
    \l V\r_{H^1} & \lesssim \left( 1 + \l \nabla\Tilde{N}^{(n)} \r_{H^2} \right) \l \left(e_0^{(n-1)}\omega^{(n)}\right)^2 \r_{L^2}+ \l e_0^{(n-1)}\omega^{(n)}\nabla e_0^{(n-1)}\omega^{(n)} \r_{L^2}
    \\& \lesssim C_i \l e_0^{(n-1)}\omega^{(n)}\r_{L^4}^2+ \l e_0^{(n-1)}\omega^{(n)}\r_{L^4}\l\nabla e_0^{(n-1)}\omega^{(n)}\r_{H^1}
    \\&\lesssim \e^2 C_i+\e C_i
    \end{align*}
    \item for $VI$, we do as for $V$, since $\nabla\omega^{(n)}$ and $e_0^{(n-1)}\omega^{(n)}$ satisfy the same estimates.
    \end{itemize}
\end{proof}

\begin{prop}\label{hr+1 omega prop}
For $n\geq 2$, the following estimates holds :
\begin{align}
     \left\|\dr_t\omega^{(n+1)}\right\|_{H^2}+\left\|\nabla\omega^{(n+1)}\right\|_{H^2} & \lesssim C_i ,\label{esti ondes omega 1}\\
     \left\Vert \dr_t\left( \Ll^{(n)}\omega^{(n+1)} \right) \right\Vert_{H^1} & \lesssim C_i.\label{esti ondes omega 2}
\end{align}
\end{prop}

\begin{proof}
The proof of Proposition \ref{hr+1 omega prop} uses the same estimates as the one of Proposition \ref{hr+1 fi prop} (since $\ffi^{(n)}$ and $\omega^{(n)}$ satisfy the same estimates), so we omit the details. 
\end{proof}

Looking at the estimates we proved for the $(n+1)$-th iterate in Propositions \ref{hr+1 N prop}, \ref{hr+1 beta prop}, \ref{HR+1 H prop}, \ref{hr+1 tau prop}, \ref{hr+1 gamma prop}, \ref{hr+1 fi prop} and \ref{hr+1 omega prop} we see that in order to recover the estimates \eqref{HR N 1}-\eqref{HR fi 1}, we have to choose the constants $A_0$, $A_1$, $A_2$, $A_3$ and $A_4$ such that $C(A_i)\ll A_{i+1}$ for all $i=0,\dots,3$ and $\e$ small, depending on the $A_i$ constants. We make such a choice.

This concludes the proof of the fact claimed above : the sequence constructed in Section \ref{section iteration scheme} is uniformly bounded. Moreover, the bounds \eqref{HR N 1}-\eqref{HR omega 1} hold for every $k\in\N$. and for every $t\in[0,T]$.

\subsection{Convergence of the sequence}\label{Cauchy}

In this section, we show that the sequence we constructed in fact converges to a limit in larger functional spaces than those used in the previous sequence, where we only proved boundedness. To this end, we will show that the sequence is a Cauchy sequence. We introduce the following distances, as in \cite{hunluk18} :
\begin{align}
    d_1^{(n)} &\vcentcolon= \left\|\Tilde{\gamma}^{(n+1)}-\Tilde{\gamma}^{(n)} \right\|_{H^1_{\delta'}} +\left\|\dr_t\left(\Tilde{\gamma}^{(n+1)}-\Tilde{\gamma}^{(n)} \right) \right\|_{L^2_{\delta'}}+\left\|H^{(n+1)}-H^{(n)} \right\|_{H^1_{\delta+1}}\nonumber\\&\qquad+\left\|\tau^{(n+1)}-\tau^{(n)} \right\|_{L^2_{\delta'+1}}+\left\|\dr_t \left(\ffi^{(n+1)}-\ffi^{(n)} \right) \right\|_{H^1} +\left\|\nabla\left(\ffi^{(n+1)}-\ffi^{(n)}\right) \right\|_{H^1}
    \\&\qquad +\left\|\dr_t \left(\omega^{(n+1)}-\omega^{(n)} \right) \right\|_{H^1} +\left\|\nabla\left(\omega^{(n+1)}-\omega^{(n)}\right) \right\|_{H^1},\nonumber\\
    d_2^{(n)} &\vcentcolon= \left| N_a^{(n+1)}-N_a^{(n)}\right|+\left\|\Tilde{N}^{(n+1)}-\Tilde{N}^{(n)} \right\|_{H^2_{\delta}}+\left\| \beta^{(n+1)}-\beta^{(n)}\right\|_{H^2_{\delta'}},\\
    d_3^{(n)} &\vcentcolon= \left\| \dr_t \left(\Ll^{(n)}\ffi^{(n+1)}-\Ll^{(n-1)}\ffi^{(n)} \right)\right\|_{L^2}+\left\| \dr_t \left(\Ll^{(n)}\omega^{(n+1)}-\Ll^{(n-1)}\omega^{(n)} \right)\right\|_{L^2},\\
    d_4^{(n)} &\vcentcolon= \left\|e_0^{(n+1)}H^{(n+1)}-e_0^{(n)}H^{(n)} \right\|_{H^1_{\delta+1}},\\
    d_5^{(n)} &\vcentcolon= \left\|\dr_t \left(\Ll^{(n)}\Tilde{\gamma}^{(n+1)}-\Ll^{(n-1)}\Tilde{\gamma}^{(n)} \right) \right\|_{L^2_{\delta'}}+\left\|\dr_t\left(\tau^{(n+1)}-\tau^{(n)} \right) \right\|_{L^2_{\delta'+1}},\\
    d_6^{(n)} &\vcentcolon= \left|\dr_t\left(N_a^{(n+1)}-N_a^{(n)} \right) \right|+\left\|\dr_t\left(\Tilde{N}^{(n+1)}-\Tilde{N}^{(n)} \right) \right\|_{H^2_{\delta}}+\left\|e_0^{(n)}\beta^{(n+1)}-e_0^{(n-1)}\beta^{(n)} \right\|_{H^2_{\delta'}}.
\end{align}

The goal is to show that each series $\sum_{n\leq 0}d_i^{(n)}$ is converging. This is a consequence of the following Proposition. At this low-level of regularity, its proof is identical to the corresponding one done in \cite{hunluk18} (see Proposition $8.19$ and Corollary $8.20$ in this article).

\begin{prop}\label{distance prop}
If $T$ and $\e$ is small enough (where $\e$ does not depend on $C_i$), the following bounds hold for every $n\geq 3$ : 
\begin{align*}
d_1^{(n)} + d_2^{(n)} + d_3^{(n)} + d_4^{(n)} +  d_5^{(n)} +  d_6^{(n)} \lesssim 2^{-n}.
\end{align*}
\end{prop}

It shows that in the function spaces involved in the definition of the distances $d^{(n)}_i$, the sequence we constructed is Cauchy and therefore convergent to some
\begin{equation}\label{strong limit}
(N=1+N_a\chi\ln+\Tilde{N},\tau,H,\beta,\gamma=-\alpha\chi\ln+\Tilde{\gamma},\ffi,\omega).
\end{equation}
Since the sequence is bounded in a smaller space, we can find a subsequence weakly converging to some limit, which has to coincide with the strong limit \eqref{strong limit}. Consequently, \eqref{strong limit} satisfies the estimates \eqref{HR N 1}-\eqref{HR omega 1}, from which we can prove that it is a solution to the reduced system \eqref{EQ N}-\eqref{EQ omega}.

If there are two solutions to the reduced system, we can control their difference using the distances $d^{(n)}_i$ and arguing as in Proposition \ref{distance prop} we show that these two solutions coincide. This proves the uniqueness of solution to the reduced system. 

We summarize this discussion in the following
corollary :

\begin{coro}\label{coro reduced sytem}
Given the initial conditions in Section \ref{initial data}, there exists a unique solution
\begin{equation}
    (N,\beta,\tau,H,\gamma,\ffi,\omega)\label{sol}
\end{equation}
to the reduced system \eqref{EQ N}-\eqref{EQ omega} such that :
\begin{itemize}
    \item $\gamma$ and $N$ admit the decompositions
    \begin{equation*}
        \gamma=-\alpha\chi\ln+\Tilde{\gamma},\qquad N=1+N_a\chi\ln+\Tilde{N},
    \end{equation*}
    where $\alpha\geq0$ is a constant, $N_a(t)\geq 0$ is a function of $t$ alone and 
    \begin{equation*}
        \Tilde{\gamma}\in H^3_{\delta'}, \quad \Ll\Tilde{\gamma}\in H^2_{\delta'+1}, \quad \dr_t\Ll\Tilde{\gamma}\in H^1_{\delta'+1}, \quad \Tilde{N}\in H^4_{\delta}, \quad \dr_t\Tilde{N}\in H^2_{\delta},
    \end{equation*}
    with estimates depending on $C_i$, $\delta$ and $R$.
    \item $(\beta,\tau,H)$ are in the following spaces :
    \begin{equation*}
        \beta\in H^3_{\delta'},\quad e_0\beta\in H^3_{\delta'},\quad \tau\in H^2_{\delta'+1},\quad \dr_t\tau\in H^1_{\delta'+1}, \quad H,e_0H\in H^2_{\delta'+1},
    \end{equation*}
    with estimates depending on $C_i$, $\delta$ and $R$.
    \item The smallness conditions in \eqref{HR N 1} and \eqref{HR beta 1} and Proposition \ref{prop propsmall} hold (without the $(n)$).
\end{itemize}
\end{coro}

\section{End of the proof of Theorem \ref{theoreme principal} }\label{section end of proof}

In this section we conclude the proof of Theorem \ref{theoreme principal} in two steps. As a first step, we show that the unique solution of the reduced system obtained in Corollary \eqref{coro reduced sytem} is actually a solution of the full system \eqref{EVE}. As we will see in Proposition \ref{Gij et Goo}, this involves among other things propagating the gauge condition $\tau =0$ (the condition $\Bar{g}=e^{2\gamma}\delta$ is also a gauge condition but we don't need to propagate it). As in the harmonic gauge, this step is done using the Bianchi equation and the constraint equations. While in the harmonic gauge the Bianchi equation implies a second order hyperbolic system for the gauge, here we obtain a transport system (see Proposition \ref{usage de bianchi}). 

In a second step, we prove the remaining estimates stated in Theorem \ref{theoreme principal}, i.e the $H^4$ norm of the metric coefficients with a loss of one regularity order for each time derivative. For this, we use the full Einstein equations in the elliptic gauge, thanks to the first step.

\subsection{Solving the Einstein equations}

In order to solve \eqref{EVE} in the elliptic gauge, it only remains to prove that $G_{\mu\nu}=T_{\mu\nu}$ (the wave equations for $\ffi$ and $\omega$ being already included into the reduced system) and that $\tau=0$. 

To define properly the tensors $G$ and $T$ we need to define a metric. Let $g$ be the metric on $\R^2\times [0,T]$ defined by the geometric quantities $N$, $\gamma$ and $\beta$ (obtained from \eqref{sol}) as in \eqref{metrique elliptique}. To compute the Einstein tensor of $g$, we need the second form fundamental and its traceless part. We define $K$ with $H$, $\gamma$ and $\tau$ (obtained from \eqref{sol}) according to \eqref{def H} and \eqref{g bar}. Thanks to \eqref{EQ beta} and \eqref{EQ tau} we have
\begin{align*}
    K_{ij}=H_{ij}+\frac{1}{2}e^{2\gamma}\tau\delta_{ij}
     = -\frac{1}{2N}e_0\left( e^{2\gamma}\right)\delta_{ij}+\frac{e^{2\gamma}}{2N}\left( \dr_i\beta_j+\dr_j\beta_i \right).
\end{align*}
By \eqref{seconde forme fonda}, this proves that $K_{ij}$ is the second fundamental form of $\Sigma_t$. On the other hand, by \eqref{EQ tau}, we know that $\tau$ is the mean curvature of $\Sigma_t$. This implies that $H_{ij}$ is the traceless part of $K_{ij}$ with respect to $\Bar{g}=e^{2\gamma}\delta$. We also define the tensor $T$ with $g$ and $(\ffi,\omega)$ (obtained from \eqref{sol}) according to \eqref{tenseur energie impulsion }.

We can now use both our computations in the elliptic gauge and the reduced system to compute $G_{00}-T_{00}$ and $G_{ij}-T_{ij}$.

\begin{prop}\label{Gij et Goo}
Given a solution to \eqref{EQ N}-\eqref{EQ omega}, the Einstein tensor in the basis $(e_0,\dr_i)$ is given by :
\begin{align}
    G_{00}&=\frac{N}{2}e_0\tau+T_{00},\label{G 00}\\
    G_{ij}&=\frac{e^{2\gamma}e_0\tau}{2N}\delta_{ij}+T_{ij}.\label{G ij}
\end{align}
Moreover, we have $D^{\mu}T_{\mu\nu}=0$.
\end{prop}

\begin{proof}
In this proof we just have to put together our calculations about $R_{\mu\nu}$ and $T_{\mu\nu}$ perdormed in Appendix \eqref{appendix A} and the reduced system \eqref{EQ N}-\eqref{EQ omega}. Note that putting \eqref{EQ tau} and \eqref{EQ gamma} together gives back an elliptic equation satisfied by $\gamma$ :
\begin{equation}
    \Delta\gamma=\frac{\tau^2}{2}e^{2\gamma}-\frac{e^{2\gamma}}{2N}e_0\tau-\frac{\Delta N}{2N}-\left|\nabla\ffi\right|^2-\frac{1}{4}e^{-4\ffi}\left|\nabla\omega \right|^2.\label{vraie eq sur gamma}
\end{equation}

In order to compute $G_{\mu\nu}$, we need the scalar curvature $R$, which, thanks to \eqref{EQ N}, \eqref{vraie eq sur gamma} and \eqref{appendix R}, has the following expression :
\begin{equation*}
    R= -\Ll\tau+2e^{-2\gamma}|\nabla\ffi|^2-\frac{2}{N^2}(e_0\ffi)^2+\frac{1}{2}e^{-2\gamma-4\ffi}|\nabla\omega|^2-\frac{1}{2N^2}e^{-4\ffi}(e_0\omega)^2.
\end{equation*}
We also recall the expression of $g_{\mu\nu}$ in the $(e_0,\dr_i)$ basis : $g_{00}=-N^2$, $g_{ij}=e^{2\gamma}\delta_{ij}$ and $g_{0i}=0$.
Since $N$ satisfies \eqref{EQ N} and thanks to \eqref{appendix R00} we get :
\begin{align*}
    G_{00}&=R_{00}-\frac{1}{2}g_{00}R\\
    &=\frac{N}{2}e_0\tau+(e_0\ffi)^2+N^2e^{-2\gamma}|\nabla\ffi|^2+\frac{1}{4}e^{-4\ffi}\left(  (e_0\omega)^2+e^{-2\gamma}N^2|\nabla\omega|^2\right),
\end{align*}
which, looking at \eqref{T 00}, gives \eqref{G 00}. Thanks to \eqref{EQ H} and \eqref{appendix Rij} we get :
\begin{equation*}
    R_{ij}=\delta_{ij}\left(-\Delta\gamma+\frac{\tau^2}{2}e^{2\gamma}-\frac{e^{2\gamma}}{2}\Ll\tau-\frac{\Delta N}{2N}\right)+2\dr_i\ffi\dr_j\ffi-\delta_{ij}|\nabla\ffi|^2+\frac{1}{2}e^{-4\ffi}\dr_i\omega\dr_j\omega-\frac{1}{4}e^{-4\ffi}\delta_{ij}|\nabla\omega|^2,
\end{equation*}
which, using \eqref{vraie eq sur gamma} gives $R_{ij}=2\dr_i\ffi\dr_j\ffi+\frac{1}{2}e^{-4\ffi}\dr_i\omega\dr_j\omega$. It gives us 
\begin{equation*}
    G_{ij}=\frac{1}{2}e^{2\gamma}\Ll\tau\delta_{ij}+2\dr_i\ffi\dr_j\ffi+\frac{e^{2\gamma}}{N^2}(e_0\ffi)^2\delta_{ij}-|\nabla\ffi|^2\delta_{ij}+\frac{1}{4}e^{-4\ffi}\left( 2\dr_i\omega\dr_j\omega+\frac{e^{2\gamma}}{N^2}(e_0\omega)^2\delta_{ij}-|\nabla\omega|^2\delta_{ij}\right),
\end{equation*}
which, looking at \eqref{T ij}, gives \eqref{G ij}. The conservation law $D^{\mu}T_{\mu\nu}=0$ is just a consequence of \eqref{EQ ffi}, \eqref{EQ omega} and \eqref{divergence de T}.
\end{proof}

By Proposition \ref{Gij et Goo}, in order to show that a solution to \eqref{EQ N}-\eqref{EQ omega} is indeed a solution to \eqref{EVE} it remains
to show that $\tau=0$ and $G_{0i} - T_{0i}=0$. These will be shown simultaneously and the Bianchi identities
\begin{equation*}
    D^{\mu}G_{\mu\nu}=0
\end{equation*}
are used in the following proposition to obtain a coupled system for this two quantities. For the sake of clarity, we use the following notations :
\begin{equation*}
    A_i\vcentcolon=G_{0i}-T_{0i},\qquad B_i\vcentcolon=G_{0i}-T_{0i}-\frac{N}{2}\dr_i\tau,
\end{equation*}
and $\dive(A)=\delta^{ij}\dr_iA_j$. The important remark about these quantities is that if we manage to show that $e_0\tau=0$, $A_i=0$ and $B_i=0$, we first have $G_{0i}-T_{0i}=0$, which, looking at the expression of $B_i$ implies that $\nabla\tau =0$, which, in addition to $e_0\tau=0$ and $\tau_{|\Sigma_0}=0$ implies that $\tau=0$ in the whole space-time. 

\begin{prop}\label{usage de bianchi}
The quantities $A_i$, $B_i$ and $e_0\tau$ satisfy the following coupled system :
\begin{align}
    e_0A_i &=\frac{N}{2}\dr_ie_0\tau+\frac{\dr_iN}{2}e_0\tau+\left(\Ll N+N\tau\right)A_i+
    \dr_i\beta^j A_j ,\label{coupled system 1}\\
    e_0B_i & = \frac{\dr_iN}{2}e_0\tau+N\tau A_i+\Ll NB_i+\dr_i\beta^jB_j,\label{coupled system 3}\\
    e_0\left( e_0\tau\right)&=2e^{-2\gamma}N\dive(A)+2e^{-2\gamma}\delta^{ij}\dr_iNA_j+(2N\tau+\Ll N) e_0\tau .\label{coupled system 2}
\end{align}
\end{prop}

\begin{proof}

The equation \eqref{coupled system 3} follows from \eqref{coupled system 1} and we omit the proof, which is a direct computation. Thanks to the previous proposition and the Bianchi identity, we have $D^{\mu}(G_{\mu\nu}-T_{\mu\nu})=0$.

We first prove \eqref{coupled system 1}. By \eqref{covariante 1} and \eqref{covariante 2}, 
\begin{align}
    D_0(G_{0i}-T_{0i}) & = e_0(G_{0i}-T_{0i})-\Ll N(G_{0i}-T_{0i})-e^{-2\gamma}\delta^{jk}N\dr_kN(G_{ij}-T_{ij})
    \nonumber\\&\quad-\frac{\dr_iN}{N}(G_{00}-T_{00})-\frac{1}{2}\left(2\delta^{j}_ie_0\gamma+\dr_i\beta^j-\delta_{ik}\delta^{j\ell}\dr_{\ell}\beta^k \right)(G_{0j}-T_{0j})
    \nonumber\\ & =e_0(G_{0i}-T_{0i})-\Ll N(G_{0i}-T_{0i})-\dr_iNe_0\tau\label{DoAoi}\\&\quad-\frac{1}{2}\left(2\delta^{j}_ie_0\gamma+\dr_i\beta^j-\delta_{ik}\delta^{j\ell}\dr_{\ell}\beta^k \right)(G_{0j}-T_{0j}),\nonumber
\end{align}
where in the last equality we have used \eqref{G 00} and \eqref{G ij}. Similarly, by \eqref{covariante 4} and \eqref{G 00}-\eqref{G ij},
\begin{align}
    g^{jk}D_j(G_{ki}-T_{ki}) = \frac{1}{2}\dr_i\Ll\tau-\frac{\tau}{N}(G_{0i}-T_{0i})-\frac{1}{2N^2}\left(2\delta_i^ke_0\gamma-\delta_{\ell}^k\dr_i\beta^{\ell}-\delta_{i\ell}\delta^{jk}\dr_j\beta^{\ell} \right)(G_{0k}-T_{0k})
    \label{DjAji}
\end{align}
Thanks to $D^{\mu}(G_{\mu i}-T_{\mu i})=0$, we have 
\begin{equation*}
    \text{\eqref{DoAoi}}-N^2\times\text{\eqref{DjAji}}=0,
\end{equation*}
which, after some straightforward simplifications, gives exactly \eqref{coupled system 1}.

We now prove \eqref{coupled system 2}. By \eqref{covariante 1} and \eqref{G 00},
\begin{equation}
    D_0(G_{00}-T_{00})=\frac{N}{2}e_0(e_0\tau)-\frac{e_0 N}{2}e_0\tau-2e^{-2\gamma}\delta^{ij}N\dr_iN(G_{0j}-T_{0j}).\label{DoAoo}
\end{equation}
On the other hand, by \eqref{covariante 3}-\eqref{covariante 4} and \eqref{G 00}-\eqref{G ij},
\begin{equation}
    g^{ij}D_i(G_{j0}-T_{j0})=e^{-2\gamma}\delta^{ij}\dr_{i}(G_{j0}-T_{j0})-e^{-2\gamma}\delta^{ij}\frac{\dr_iN}{N}(G_{j0}-T_{j0})+\tau e_0\tau.\label{DiAi0}
\end{equation}
Thanks to $D^{\mu}(G_{\mu 0}-T_{\mu 0})=0$, we have
\begin{equation*}
    -\frac{1}{N}\times\text{\eqref{DoAoo}}+N\times\text{\eqref{DiAi0}}=0,
\end{equation*}
which, after some straightforward simplifications, gives exactly \eqref{coupled system 2}.
\end{proof}

\begin{prop}
Suppose the solution to \eqref{EQ N}-\eqref{EQ omega} as constructed in Section \ref{section solving the reduced system} arises from initial
data with $\tau_{|\Sigma_0}=0$ and that the constraint equations are initially satisfied, then the
solution satisfies
\begin{align*}
    \tau&=0,\\
    G_{0i}&=T_{0i}.
\end{align*}
As a consequence, the solution to \eqref{EQ N}-\eqref{EQ omega} is indeed a solution to \eqref{EVE}.
\end{prop}

\begin{proof}
We set the following energy :
\begin{equation*}
    E(t)\vcentcolon=\left\| e_0\tau\right\|_{L^2}^2+\sum_{i=1,2}\left(\left\|2 e^{-\gamma}A_i\right\|_{L^2}^2+\left\| B_i\right\|_{L^2}^2 \right).
\end{equation*}
We first note that $E(0)=0$ because our solution arises from initial date satisfying the constraint equations (which implies that $(G_{0i}-T_{0i})_{|\Sigma_0}=0$) and because $\tau_{|\Sigma_0}=0$. Our goal is to show that $E(t)=0$ for all $t\in[0,T]$.

We first multiply \eqref{coupled system 3} by $B_i$ and sum over $i=1,2$ the two equations we obtain. We integrate over $\R^2$ and write $e_0=\dr_t-\beta\cdot\nabla$ to obtain (after an integration by part on the last term) :
\begin{align*}
    \frac{1}{2}\frac{\d}{\d t}\sum_{i=1,2}\left\| B_i\right\|_{L^2}^2&=\int_{\R^2}\sum_{i=1,2}\frac{\dr_iN}{2}B_ie_0\tau+\int_{\R^2}\sum_{i=1,2}N\tau B_iA_i+\int_{\R^2}\sum_{i=1,2}\Ll NB_i^2\\&\qquad+\int_{\R^2}\sum_{i=1,2}\dr_i\beta^jB_iB_j-\int_{\R^2}\sum_{i=1,2}\frac{1}{2}\dive(\beta)B_i^2.
\end{align*}
Using Corollary \ref{coro reduced sytem} and Proposition \ref{embedding}, we see that the quantities $\nabla N$, $N\tau$, $\Ll N$ and $\nabla\beta$ are bounded (for $N\tau$ and $N\dr_iN$, we use the decay property of $\tau$ and $\dr_i N$ to deal with the logarithmic growth of $N$). Using the trick $2ab\leq a^2+b^2$ and the fact that $1\lesssim e^{-\gamma}$, we get :
\begin{equation}
    \frac{\d}{\d t}\sum_{i=1,2}\left\| B_i\right\|_{L^2}^2\leq C E(t).\label{coupled energy 1}
\end{equation}

Similarly, multiplying \eqref{coupled system 2} by $e_0\tau$ , we get :
\begin{align}
    \frac{\d}{\d t}\left\|e_0\tau \right\|_{L^2}^2 & = 4\int_{\R^2}e^{-2\gamma}N\,\dive(A)e_0\tau+2\int_{\R^2}\sum_{i=1,2}e^{-2\gamma}e_0\tau\dr_iN A_i+2\int_{\R^2}(2N\tau+\Ll N)\left( e_0\tau\right)^2\nonumber\\&\qquad\qquad-\int_{\R^2}\dive(\beta)\left( e_0\tau\right)^2
    \nonumber\\&=-4\int_{\R^2}e^{-2\gamma}NA_i\dr_ie_0\tau+O(E(t)),\label{coupled energy 2}
\end{align}
where we integrated by part the first term and bound the other terms just as we did for $\|B_i\|_{L^2}$, mainly using Corollary \ref{coro reduced sytem}. Now, writing $\dr_t=e_0+\beta\cdot\nabla$ and integrating by part, we get :
\begin{align*}
    \frac{1}{2}\frac{\d}{\d t}\sum_{i=1,2}\|e^{-\gamma} A_i\|_{L^2}^2&=\int_{\R^2}\sum_{i=1,2}e^{-2\gamma}A_ie_0A_i-\frac{1}{2}\int_{\R^2}\sum_{i=1,2}\dive\left(e^{-2\gamma}\beta\right)A_i^2+\frac{1}{2}\int_{\R^2}\sum_{i=1,2}\dr_t(-2\Tilde{\gamma})e^{-2\gamma}A_i^2
    \\& =\int_{\R^2}\sum_{i=1,2}e^{-2\gamma}A_ie_0A_i+O(E(t)).
\end{align*}
Using \eqref{coupled system 1}, we thus get :
\begin{align}
    \frac{\d}{\d t}\sum_{i=1,2}\|2e^{-\gamma} A_i\|_{L^2}^2= 4\int_{\R^2}\sum_{i=1,2}e^{-2\gamma}NA_i\dr_ie_0\tau+O(E(t)).\label{coupled energy 3}
\end{align}
Looking at \eqref{coupled energy 2} and \eqref{coupled energy 3}, we see that our choice of scaling in the expression of $E(t)$ implies a cancellation and we finally get, recalling \eqref{coupled energy 1} :
\begin{equation*}
    \frac{\d}{\d t}E(t)\leq CE(t)
\end{equation*}
which, using the Gronwall's Lemma and $E(0)=0$, implies that $E(t)=0$ for all $t\in [0,T]$, which implies the desired result.
\end{proof}

\subsection{Improved regularity}

To conclude the proof of the Theorem \ref{theoreme principal}, it only remains to prove the bounds stated in this theorem. Notice that some of the estimates are already obtained in Corollary \ref{coro reduced sytem}. This improvement of regularity is due to the fact that we now know that the solution of the reduced system is also a solution to the system \eqref{EVE}, and therefore all the metric components solves elliptic equations. 

\begin{lem}\label{lemme elliptique}
The metric components $N$, $\gamma$ and $\beta$ satisfy the following elliptic equations :
\begin{align}
    \Delta N & =e^{-2\gamma}N|H|^2+\frac{2e^{2\gamma}}{N}(e_0\ffi)^2+\frac{e^{2\gamma-4\ffi}}{2N}(e_0\omega)^2,\label{elliptic N}\\
    \Delta\gamma & = -|\nabla\ffi|^2-\frac{1}{4}e^{-4\ffi}|\nabla\omega|^2-\frac{e^{2\gamma}}{N^2}(e_0\ffi)^2-\frac{e^{2\gamma-4\ffi}}{4N^2}(e_0\omega)^2-\frac{e^{-2\gamma}}{2}|H|^2,\label{elliptic gamma}\\
    \Delta\beta^j & =\delta^{jk}\delta^{i\ell}(L\beta)_{ik}\left(\frac{\dr_{\ell}N}{2N}-\dr_{\ell}\gamma \right)-2\delta^{kj}e_0\ffi\dr_k\ffi-\frac{1}{2}e^{-4\ffi}\delta^{kj}e_0\omega\dr_k\omega.\label{elliptic beta}
\end{align}
\end{lem}

\begin{proof}
Since we solved \eqref{EVE}, we have $R_{00}=T_{00}-g_{00}\tr_{g}T$, which, according to \eqref{appendix R00} and \eqref{JSP}, easily implies \eqref{elliptic N}.

Using \eqref{appendix R00}, \eqref{appendix R} and the fact that $\tau=0$, we get that
\begin{equation*}
    G_{00}=N^2e^{-2\gamma}\left(- \Delta\gamma-\frac{e^{-2\gamma}}{2}|H|^2\right),
\end{equation*}
Using \eqref{T 00} and the fact that $G_{00}=T_{00}$ we get \eqref{elliptic gamma}. 

The equation $R_{0j}=2e_0\ffi\dr_j\ffi+\frac{1}{2}e^{-4\ffi}e_0\omega\dr_j\omega$ and the fact that $\tau=0$ together with \eqref{appendix R0j} and \eqref{appendix beta} gives \eqref{elliptic beta}.

\end{proof}

In the following proposition, we state and prove the missing estimates :

\begin{prop}
Taking $\e_0$ smaller if necessary, the following estimates hold :
\begin{align}
         \left\|\Tilde{\gamma}\right\|_{H^4_{\delta}}+ \left\| \beta \right\|_{H^4_{\delta'}}  & \leq C_h,\label{improved regularity 1}\\
         \left\|\dr_t\Tilde{\gamma}\right\|_{H^3_{\delta}}+ \left\|\dr_t\Tilde{N}\right\|_{H^3_{\delta}}+\left\| \dr_t\beta \right\|_{H^3_{\delta'}} & \leq C_h,\label{improved regularity 2}\\
         \left|\dr_t^2N_a\right|+\left\|\dr_t^2\Tilde{N}\right\|_{H^2_{\delta}}+\left\| \dr_t^2\beta \right\|_{H^2_{\delta'}}+\left\|\dr_t^2\Tilde{\gamma}\right\|_{H^2_{\delta}} & \leq C_h.\label{improved regularity 3}
    \end{align}
\end{prop}

\begin{proof}
The idea is just to apply Corollary \ref{mcowens 2} to the equations  \eqref{elliptic N}-\eqref{elliptic gamma}-\eqref{elliptic beta}, after having proven, using the regularity obtained in Corollary \ref{coro reduced sytem}, that the RHS of these equations are in the appropriate spaces. For the estimates involving time derivatives, we proceed in the same way after having differenciated once or twice the equations \eqref{elliptic N}-\eqref{elliptic gamma}-\eqref{elliptic beta}. We omit the details, since the computations are straightforward (mainly because now we don't have to worry about the constants in the estimates).

\end{proof}

This concludes the proof of Theorem \ref{theoreme principal}.

\section{Proof of Theorem \ref{theo 2}}\label{section theo 2}

\subsection{Almost $H^2$ well-posedness}

At this stage, thanks to Theorem \ref{theoreme principal}, we proved that the system \eqref{EVE} is well posed locally in time with initial data $(\dr\ffi,\dr\omega)\in H^2$. The next step would be to consider initial data $(\dr\ffi,\dr\omega)$ which are only in $H^1$. In order to obtain well-posedness in this setting, we could regularize the initial data with a sequence $(\dr\ffi_n,\dr\omega_n)\in H^2$ to which we can apply Theorem \ref{theoreme principal}, thus obtaining a sequence of solution to \eqref{EVE} on $[0,T_n)$. A priori, if $(\dr\ffi,\dr\omega)$ only belongs to $H^1$, the $H^2$ norm of $(\dr\ffi_n,\dr\omega_n)$ explodes as $n$ tends to $+\infty$ and therefore the sequence $(T_n)_{n\in\N}$ converges to 0, forbidding us to define a limit on some non-trivial interval.

To prevent this to happen, we need to prove that the $H^2$ and $L^4$ estimates of each $(\dr\ffi_n,\dr\omega_n)$  can be propagated on some fixed interval using only their $H^1$ norm (which are bounded by the $H^1$ norm of the initial data) using the system that $\ffi$ and $\omega$ solve, i.e the system \eqref{WM ffi}-\eqref{WM omega} below. As we will see in the rest of this section, it is possible to improve the $H^2$ norm. But unfortunately, we can't improve the $L^4$ estimates using only the $H^1$ norm and the system \eqref{WM ffi}-\eqref{WM omega}. Note that this difficulty already occured in the proof of Theorem \ref{theoreme principal} but we bypassed it by taking advantage of the smallness of the time of existence (see Proposition \ref{prop propsmall}), something that we cannot do in this approximation procedure.

Therefore, we can't prove local well-posedness at the $H^2$ level. Instead we prove a blow-up criterium, meaning that the $L^4$ estimates that we can't propagate is assumed to hold from the start. It only remains to improve the $H^2$ estimates.

\subsection{The wave map structure}

To prove Theorem \ref{theo 2}, we argue by contradiction and assume throughout this section that the following statements both hold on $[0,T)$ :
\begin{align}
\l \dr \ffi \r_{H^1} + \l \dr\omega \r_{H^1}&\leq C_0,\label{propa H1}\\
\l \dr \ffi \r_{H^1} + \l \dr\omega \r_{L^4}&\leq \e_0,\label{propa L4}
\end{align}
for some $C_0>0$, and $\e_0>0$ defined in Theorem \ref{theoreme principal}, and where $T$ is the maximal time of existence of a solution to \eqref{EVE}. The goal is to show that we can actually bound the $H^2$ norm of $\dr\ffi$ and $\dr\omega$ on $[0,T)$, and hence up to $T$, using \eqref{propa H1}. Then, using in addition \eqref{propa L4} and applying Theorem \ref{theoreme principal}, we construct a solution of \eqref{EVE} beyond the time $T$. This would contradict the maximality of $T$, and thus prove Theorem \ref{theo 2}.
\par\leavevmode\par
In order to estimate the $H^2$ norm of $\dr\ffi$ and $\dr\omega$ using \eqref{propa H1}, we are going to use the wave map structure of the coupled wave equations solved by $\ffi$ and $\omega$, which we recall :
\begin{align}
\Box_g\ffi&=-\frac{1}{2}e^{-4\ffi}\dr^\rho\omega\dr_\rho\omega,\label{WM ffi}\\
\Box_g\omega&=4\dr^\rho\omega\dr_\rho\ffi.\label{WM omega}
\end{align}
We also recall the expression of the operator $\Box_g$ in the case $\tau=0$ :
\begin{equation}
\Box_g f= -\Ll^2f+\frac{e^{-2\gamma}}{N}\dive(N\nabla f),\label{expression de box}
\end{equation}
where $f$ is any function on $\mathcal{M}$. Note the following notation for the rest of this section : $U$ stands for $\ffi$ or $\omega$, $g$ stands for any metric coefficient, meaning $N$, $\gamma$ and $\beta$.

\subsubsection{The naive energy estimate}

We want to control the $H^2$ norm of $\dr U$. As $U$ satisfies a wave equation, we could use Lemma \ref{inegalite d'energie lemme}. With our formal notation, this wave equation writes
\begin{equation*}
\Box_g U= g^{-1}(\dr U)^2. 
\end{equation*}
Thus, Lemma \ref{inegalite d'energie lemme} would basically implies that
\begin{align*}
\l \dr\nabla^2 U \r_{L^2}^2 & \lesssim C_{high}^2 + \int_0^t \l \dr\nabla^2 U \nabla^2\left( g^{-1}(\dr U)^2\right) \r_{L^1}
\\& \lesssim C_{high}^2 +\int_0^t\l(\dr\nabla^2 U)^2\dr U \r_{L^1} + \cdots,
\end{align*}
where the dots reprensent term easily bounded by $\l\dr U \r_{H^2}^2$. The problem is that, using only \eqref{propa H1} and \eqref{propa L4}, the term $\l(\dr\nabla^2 U)^2\dr U \r_{L^1}$ cannot be bounded by $\l\dr U \r_{H^2}^2$, it requires necessarily $\l\dr U \r_{H^2}^{2+\eta}$ with $\eta>0$. Thus, a continuity argument, aiming at proving boundedness in $H^2$, would be impossible to carry out.  

Therefore, we need to use deeper the structure of the coupled equations \eqref{WM ffi} and \eqref{WM omega}. This structure will allows us to define a third order energy, which will have the property of avoiding $\l\dr U \r_{H^2}^{2+\eta}$ terms into the energy estimate.

\subsubsection{The third order energy}

The system \eqref{WM ffi}-\eqref{WM omega} has actually more structure than we could expect : it is a wave map system, as shown in \cite{Malone}. More precisely, if we consider the map $u=(\ffi,\omega)$, then $u$ is an harmonic map from $([0,T)\times \R^3,g)$ to $(\R^2, h)$ with $h$ being the following metric :
\begin{equation*}
2(\d x)^2 + \frac{1}{2}e^{-4x}(\d y)^2.
\end{equation*}
For those wave map systems, Choquet-Bruhat in \cite{CBwavemaps} noted that we can define a third order energy, which in our case is 
\begin{equation*}
\mathscr{E}_3\vcentcolon=\mathscr{E}_3^\ffi+\mathscr{E}_3^\omega,
\end{equation*}
with
\begin{align*}
\mathscr{E}_3^\ffi & \vcentcolon= \int_{\R^2} 2 \left[ \frac{1}{N^2}\left(e_0\dr_j\dr_i\ffi+\frac{1}{2}e^{-4\ffi}\dr_j\dr_i\omega e_0\omega \right)^2 +e^{-2\gamma}\left|\nabla\dr_j\dr_i\ffi+\frac{1}{2}e^{-4\ffi}\dr_j\dr_i\omega \nabla\omega \right|^2\right]\d x, \\
\mathscr{E}_3^\omega &\vcentcolon = \int_{\R^2}\frac{1}{2}e^{-4\ffi}\left[ \frac{1}{N^2}\left( e_0\dr_j\dr_i\omega -2\dr_j\dr_i\omega e_0\ffi-2\dr_j\dr_i\ffi e_0\omega \right)^2 \right.\\&\qquad\qquad\qquad\qquad\qquad \left.+e^{-2\gamma}\left|\nabla\dr_j\dr_i\omega -2\dr_j\dr_i\omega \nabla\ffi-2\dr_j\dr_i\ffi \nabla\omega \right|^2\right] \d x .
\end{align*}
Our goal is to show that we can estimate $\mathscr{E}_3$ by $\l\dr U \r_{H^2}^{2}$. We start by commuting $\Box_g$ with $\dr_i\dr_j$ to obtain :

\begin{align}
\Box_g\dr_j\dr_i\ffi+e^{-4\ffi}g^{\alpha\beta}\dr_\alpha\dr_j\dr_i\omega\dr_\beta\omega  & = F^\ffi_{ij}\label{WM dd ffi},\\
\Box_g\dr_j\dr_i\omega - 4g^{\alpha\beta}\dr_\alpha\dr_j\dr_i\omega\dr_\beta\ffi- 4g^{\alpha\beta}\dr_\alpha\omega\dr_\beta\dr_j\dr_i\ffi&= F^\omega_{ij},\label{WM dd omega}
\end{align}
where we set 
\begin{align}
F^\ffi_{ij} & \vcentcolon= \left[\Box_g ,\dr_j\dr_i \right]\ffi+-\frac{1}{2}\dr_i\dr_j\left(e^{-4\ffi}g^{\alpha\beta} \right)\dr_\alpha\omega\dr_\beta\omega\nonumber
\\&\qquad\qquad\qquad-\dr_{(i}\left(e^{-4\ffi}g^{\alpha\beta} \right)\dr_\alpha\dr_{j)}\omega\dr_\beta\omega - e^{-4\ffi}g^{\alpha\beta}\dr_\alpha\dr_i\omega\dr_\beta\dr_j\omega,\label{F ffi}\\
F^\omega_{ij} & \vcentcolon=\left[\Box_g ,\dr_j\dr_i \right]\omega+ 4\dr_i\dr_jg^{\alpha\beta}\dr_\alpha\omega\dr_\beta\ffi\nonumber\\&\qquad +4\dr_{(i}g^{\alpha\beta}\dr_\alpha\omega\dr_\beta\dr_{j)}\ffi+4\dr_{(i}g^{\alpha\beta}\dr_\alpha\dr_{j)}\omega\dr_\beta\ffi+4g^{\alpha\beta}\dr_\alpha\dr_{(i}\omega\dr_\beta\dr_{j)}\ffi.\label{F omega}
\end{align}
We also define the following quantity :
\begin{align}\label{def Reste}
\mathscr{R}&\vcentcolon = \l\dr_t\gamma \r_{L^\infty}\left(\l\dr \nabla^2 U\r_{L^2}^2 + \l\dr U\nabla^2 U \r_{L^2}^2\right)\nonumber
\\& \quad \; +\left(\l\dr \nabla^2 U\r_{L^2} + \l\dr U\nabla^2 U \r_{L^2}\right) \left( \l\nabla ^2 U (\dr U)^2 \r_{L^2}+\l (\dr U)^3\r_{L^2}+\l F^U\r_{L^2}\right.\\&\qquad\qquad\qquad\qquad\qquad\qquad\qquad\left.+\l (\dr\nabla U)^2 \r_{L^2} +\l\nabla g\nabla^3U \r_{L^2} + \l\nabla g\nabla U\nabla^2 U \r_{L^2}  \right) \nonumber
\end{align}
where by $F^U$ we mean either $F^\ffi$ or $F^\omega$.  For clarity, the computations for the time derivative of the energy $\mathscr{E}_3$ are done in Appendix \ref{appendix C}, where we prove the following proposition. 
\begin{prop}\label{dernier coro}
The energy $\mathscr{E}_3$ satisfies
\begin{equation*}
\frac{\d}{\d t}\mathscr{E}_3=O(\mathscr{R}(t)).
\end{equation*}
\end{prop}
This proposition shows the interest of the energy $\mathscr{E}_3$ : its time derivative do not include terms of the form $\l(\dr\nabla^2U)^2\dr U \r_{L^1}$, unlike the usual energy estimate of Lemma \ref{inegalite d'energie lemme}.

\subsection{Continuity argument}

Before starting the continuity argument, we need to show that $\mathscr{R}$ can be bounded by $\l \dr U \r_{H^2}^2$ (Lemmas \ref{estimation F} and \ref{majorer R}) and to compare $\mathscr{E}_3$ with $\l \dr U \r_{H^2}^2$ (Lemma \ref{comparaison}). To this end, we will use the following key estimates :
\begin{align}
\l u \r_{L^4}&\lesssim \l u\r_{L^2}^{\frac{1}{2}}\l u\r_{H^1}^{\frac{1}{2}},\label{estimate A}\\
\l u\r_{L^\infty}&\lesssim \l u\r_{L^2}^{\frac{1}{2}}\l \nabla^2u\r_{L^2}^{\frac{1}{2}}.\label{estimate B}
\end{align}
Both are consequences of the Gagliardo-Nirenberg interpolation inequality, see Proposition \ref{GN}. We will use without mention the fact that $\l\ffi\r_{L^2}+ \l \omega\r_{L^2}\lesssim \l\ffi\r_{L^4}+ \l \omega\r_{L^4}\lesssim \e_0$ (since $\ffi$ and $\omega$ are compactly supported function and because of \eqref{propa L4}), and also the fact that $\l g\r_{H^2}\lesssim \e_0$. We also need to estimate $\nabla^3g$. For this, we apply the usual elliptic estimate to the equation $\Delta g = (\nabla g)^2 + (\dr U)^2$ (this is the type of equations solved by the metric coefficients in the elliptic gauge, see Lemma \ref{lemme elliptique}). It first gives :
\begin{align*}
\l \nabla g \r_{W^{2,\frac{4}{3}}} & \lesssim \l  \nabla^2 g \nabla g \r_{L^\frac{4}{3}} + \l \dr \nabla U \dr U \r_{L^\frac{4}{3}}
\\& \lesssim \l\nabla g \r_{H^1}^2 + \l\dr U \r_{H^1}^2
\end{align*}
where we used Hölder's inequality $L^2\times L^4 \xhookrightarrow{} L^\frac{4}{3}$ and the embedding $H^1\xhookrightarrow{}L^4$. The embedding $W^{2,\frac{4}{3}}\xhookrightarrow{}L^\infty$ then gives :
\begin{equation}
\l \nabla g \r_{L^\infty} \lesssim  \e_0^2 + C_0^2. \label{L infini dg}
\end{equation}
The $L^2$ elliptic estimate implies :
\begin{align*}
\l \nabla g \r_{H^2} & \lesssim \l  \nabla^2 g \nabla g \r_{L^2} + \l \dr \nabla U \dr U \r_{L^2}
\\ & \lesssim \e_0 \l \nabla g \r_{H^2} + \e_0 C_0 \l \dr U\r_{H^2}^\frac{1}{2}
\end{align*}
where we used $\l g \r_{H^2}\lesssim \e_0$, the Hölder's inequality and \eqref{estimate A}. Taking $\e_0$ small enough this gives :
\begin{equation}
\l \nabla g \r_{H^2}\lesssim C(C_0)\l\dr U \r_{H^2}^\frac{1}{2}. \label{H 2 dg}
\end{equation}
In the sequel, we will commute without mention $\dr$ and $\nabla$ since $[e_0,\nabla]=\nabla\beta \nabla$ and $\nabla\beta$ can be bounded using \eqref{L infini dg}.

\subsubsection{The energy $\mathscr{R}$}

We start by the estimates for $F^U$ :

\begin{lem}\label{estimation F}
There exists $C(C_0)>0$ such that
\begin{align*}
\l F^\ffi_{ij}\r_{L^2}+\l F^\omega_{ij}\r_{L^2}& \lesssim C(C_0)\l\dr U \r_{H^2}.
\end{align*}
\end{lem}

\begin{proof}
The expressions of $F^\ffi_{ij}$ and $F^\omega_{ij}$ are given by \eqref{F ffi} and \eqref{F omega}. We start by estimate the commutator $[\Box_g,\nabla^2]U$. Looking at the expression \eqref{expression de box}, we start by the spatial part of $\Box_g$ :
\begin{align*}
\l \left[ \nabla^2, g\nabla (g\nabla\cdot) \right] U \r_{L^2} & \lesssim \l g\nabla^3 g \nabla U \r_{L^2} + \l \nabla g \nabla^2 g \nabla U \r_{L^2}
+ \l g\nabla^2 g \nabla^2 U \r_{L^2} + \l (\nabla g)^2 \nabla^2 U \r_{L^2}
\\&\quad + \l g\nabla g \nabla^3 U \r_{L^2}.
\end{align*}
For the last two terms, we simply bound $\nabla g$ using \eqref{L infini dg} and put $\dr U$ in $H^2$ :
\begin{align*}
\l (\nabla g)^2 \nabla^2 U \r_{L^2} + \l g\nabla g \nabla^3 U \r_{L^2} & \lesssim C(C_0) \left(  \l \nabla^2 U \r_{L^2} + \l \nabla^3 U \r_{L^2} \right).
\end{align*}
We do the same for the second term, using in addition $\l\nabla^2 g \r_{L^2}\lesssim 1$ and the embedding $H^2\xhookrightarrow{}L^\infty$ :
\begin{align*}
\l \nabla g \nabla^2 g \nabla U \r_{L^2} & \lesssim C(C_0) \l \nabla U \r_{L^\infty}.
\end{align*}
We deal with the third term using first the Hölder's inequality, the embedding $H^1\xhookrightarrow{}L^4$, \eqref{H 2 dg} and \eqref{estimate A} :
\begin{align*}
\l g\nabla^2 g \nabla^2 U \r_{L^2} & \lesssim \e_0 \l \nabla^2g \r_{H^1} \l \nabla^2U \r_{L^4}
\\& \lesssim C(C_0) \l \dr U \r_{H^2}^{\frac{1}{2}} \l \nabla^2 U \r_{H^1}^{\frac{1}{2}}.
\end{align*}
For the first term, we put $\nabla^3g$ in $L^2$ and $\nabla U$ in $L^\infty$, and then  use \eqref{H 2 dg} and \eqref{estimate B} :
\begin{align*}
\l g\nabla^3 g \nabla U \r_{L^2} & \lesssim C(C_0) \l \dr U \r_{H^2}^{\frac{1}{2}} \l \nabla U \r_{H^2}^{\frac{1}{2}}.
\end{align*}
Summarizing everything we obtain :
\begin{equation}
\l \left[ \nabla^2, g\nabla (g\nabla\cdot) \right] U \r_{L^2} \lesssim C(C_0) \l\dr U \r_{H^2}.\label{spatial}
\end{equation}
We now estimate the contribution of $\Ll^2$ to the commutator. We have :
\begin{align*}
\l \left[ \nabla^2, \Ll^2 \right] U \r_{L^2} & \lesssim \l \nabla^2 g \dr_t^2 U \r_{L^2} + \l \nabla g \nabla \dr_t^2 U \r_{L^2} + \l\nabla^2\dr_t g\dr U \r_{L^2}  +\l\nabla^2g \nabla \dr U \r_{L^2} + \l\nabla g \nabla^2\dr U \r_{L^2}
\end{align*}
The last two terms have already been estimated during the proof of \eqref{spatial}. For the first two terms, we use the equation $\Box_g U= (\dr U)^2$ satisfied by $U$ to express $\dr_t^2 U$. It shows that
\begin{align}
|\dr_t^2 U| & \lesssim \left( 1 + |\nabla g| \right) |\dr U |^2 + |g\nabla^2U|
\lesssim C(C_0) \left(  |\dr U |^2 + |\nabla^2U|\right)\label{dt2}
\end{align}
where we also used \eqref{L infini dg}. We can put $\nabla^2g$ in $L^2$ and $\dr U$ in $L^\infty$ using \eqref{estimate B} (note that the second term has already been estimated) :
\begin{align*}
\l \nabla^2 g \dr_t^2 U \r_{L^2} & \lesssim \l\nabla^2g (\dr U)^2\r_{L^2} + \l \nabla^2g \nabla^2 U\r_{L^2} \lesssim C(C_0) \l \dr U \r_{H^2}
\end{align*}
The equation $\Box_g U= (\dr U)^2$ also gives us
\begin{equation*}
|\nabla \dr_t^2 U|\lesssim |\nabla \dr U \dr U | + |\nabla^2 g \nabla U| + |\nabla^3 U| + |\nabla g \nabla^2 U|.
\end{equation*}
Because of \eqref{L infini dg}, $\l \nabla g \nabla \dr_t^2 U \r_{L^2} \lesssim C(C_0)\l  \nabla \dr_t^2 U \r_{L^2} $ and the previous estimate shows therefore that all the terms in $\l \nabla g \nabla \dr_t^2 U \r_{L^2} $ have already been estimated. This gives :
\begin{equation*}
\l \nabla g \nabla \dr_t^2 U \r_{L^2} \lesssim C(C_0) \l \dr U \r_{H^2}.
\end{equation*}
It remains to deal with the term involving $\dr_t g$. This quantity satisfies the following equation :
\begin{equation*}
\Delta \dr_t g = \nabla \dr_t g \nabla g + \dr_t^2 U \nabla U + \dr_t U \nabla \dr_t U.
\end{equation*}
The usual elliptic estimates gives us 
\begin{align*}
\l\dr_t g \r_{H^2} & \lesssim \l \nabla \dr_t g \nabla g \r_{L^2} + \l\dr_t^2 U \nabla U \r_{L^2} + \l\dr_t U \nabla \dr_t U \r_{L^2}
\\& \lesssim \e_0 \l \nabla \dr_t g \r_{H^1} + C(C_0) \l \dr U \r_{H^2}^\frac{1}{2}
\end{align*}
where we used \eqref{estimate A} and \eqref{estimate B}. Taking $\e_0$ small enough, this shows that $\l\dr_t g \r_{H^2}\lesssim C(C_0) \l \dr U \r_{H^2}^\frac{1}{2}$. With this, we estimate the remaining term in the commutator $\left[ \nabla^2,\Ll^2\right]$ using in addition \eqref{estimate B} :
\begin{align*}
\l\nabla^2\dr_t g\dr U \r_{L^2} & \lesssim \l \nabla^2\dr_t g \r_{L^2} \l \dr U \r_{L^\infty} \lesssim C(C_0) \l \dr U \r_{H^2}
\end{align*}
Thus, we obtain :
\begin{equation}
\l \left[ \nabla^2, \Ll^2 \right] U \r_{L^2} \lesssim C(C_0) \l \dr U \r_{H^2}.\label{time}
\end{equation}
Putting \eqref{spatial} and \eqref{time} together we finally obtain :
\begin{equation*}
\l \left[ \Box_g, \nabla^2 \right] U \r_{L^2} \lesssim C(C_0) \l \dr U \r_{H^2}.
\end{equation*}
The lemma is actually proved because all the remaining terms in $F^U$ have already been estimated in the proof of \eqref{spatial} and \eqref{time}.
\end{proof}

We now estimate $\mathscr{R}$ :

\begin{lem}\label{majorer R}
There exists $C'(C_0)>0$ such that
\begin{equation*}
\mathscr{R}\leq C'(C_0) \l \dr U \r_{H^2}^2.
\end{equation*}
\end{lem}

\begin{proof}
First, note that the previous lemma handles the term $\l F^U \r_{L^2}$. Most of the remaining terms in $\mathscr{R}$ can simply be estimated using the Hölder's inequality, \eqref{estimate A} and \eqref{estimate B} :
\begin{align*}
\l\dr U \nabla^2 U \r_{L^2} & \lesssim C_0 \l\dr U \r_{H^2},
\\ \l \nabla^2U (\dr U)^2\r_{L^2} & \lesssim \e_0 C_0 \l\dr U \r_{H^2},
\\ \l (\dr U)^3\r_{L^2} & \lesssim \e_0^2 \l \dr U \r_{H^2},
\\ \l (\dr\nabla U)^2 \r_{L^2} & \lesssim C_0 \l \dr U \r_{H^2},
\\ \l \nabla g \nabla U \nabla^2 U \r_{L^2} & \lesssim \e_0C(C_0)\l\dr U \r_{H^2} .
\end{align*}
Let us give the details only for the last one. We bound $\nabla g$ with \eqref{L infini dg}, and then we the Hölder's inequality $L^4 \times  L^4\xhookrightarrow{} L^2$ and the embedding $H^1\xhookrightarrow{} L^4$ :
\begin{align*}
\l \nabla g \nabla U \nabla^2 U \r_{L^2} & \lesssim \l \nabla g \r_{L^\infty}  \l \nabla U \r_{L^4} \l \nabla^2 U \r_{L^4}\lesssim \e_0C(C_0)\l\dr U \r_{H^2} .
\end{align*}
Because of \eqref{appendix tau} and the gauge condition $\tau=0$ we have $|\dr_t\gamma|\lesssim |\nabla g|$ so we estimate $\l \dr_t\gamma \r_{L^\infty}$ with \eqref{L infini dg}. Samewise with \eqref{L infini dg} we estimate the very last term appearing in $\mathscr{R}$ :
\begin{align*}
\l \nabla g \nabla^3 U  \r_{L^2} \lesssim (\e_0^2+C_0^2) \l \dr U \r_{H^2}.
\end{align*}
\end{proof}

In the next lemma, we compare $\mathscr{E}_3$ with the $H^2$ norm of $\dr U$. We omit the proof since all the terms involved have been already estimated in the two previous lemmas.
\begin{lem}\label{comparaison}
There exists $K(C_0)>0$  such that 
\begin{align*}
\mathscr{E}_3& \leq K(C_0) \l \dr U \r_{H^2}^2,\\
\l \nabla^2\dr U \r_{L^2}^2& \leq K(C_0)\mathscr{E}_3 +\e_0^2K(C_0) \l\dr U \r_{H^2}^2.
\end{align*}

\end{lem}

\subsubsection{Conclusion}
Putting everything together, we can now complete the continuity argument by propagating the $H^2$ regularity. We consider the following bootstrap assumption :
\begin{equation}
\l \dr U \r_{H^2}(t) \leq C_1\exp(C_1t),\label{bootstrap H2}
\end{equation}
with $C_1>0$ to be chosen later. Let $T_0<T$ be the maximal time such that \eqref{bootstrap H2} holds for all $0\leq t\leq T_0$. Note that if $C_1$ is large enough we have $T_0>0$, since $\dr\ffi$ and $\dr\omega$ are initially in $H^2$.

\begin{prop}
If $\e_0$ is small enough (still independent of $C_{high}$) and $C_1$ is large enough, the following holds on $[0,T_0]$ :
\begin{equation}
\l \dr U \r_{H^2}(t) \leq \frac{1}{2}C_1\exp(C_1t).
\end{equation}
\end{prop}

\begin{proof}
The $H^1$ norm of $\dr U$ is already controled, so it suffices to prove the bound stated in the proposition for $\l \nabla^2\dr U \r_{L^2}$. For this,  we use the Proposition \ref{dernier coro}, which implies that for $t\in[0,T_0]$ (we also use Lemma \ref{comparaison} and \eqref{bootstrap H2}) :
\begin{align*}
\l \nabla^2\dr U \r_{L^2}^2(t) & \leq K^2\l\dr U \r_{H^2}^2(0)+\e_0^2K \l\dr U \r_{H^2}^2(t)+CK\int_0^t\mathscr{R}(s)\d s
\\&\leq K^2C_{high}^2+\e_0^2K C_1^2\exp(2C_1t)+CK\int_0^t\mathscr{R}(s)\d s,
\end{align*}
for some $C>0$ given by Proposition \ref{dernier coro}. We now use Lemma \ref{majorer R} :
\begin{align*}
\l \nabla^2\dr U \r_{L^2}^2(t) & \leq K^2C_{high}^2+\e_0^2K C_1^2\exp(2C_1t)+CKC'(C_0)\int_0^t\l\dr U \r_{H^2}^2(s)\d s
\\& \leq K^2C_{high}^2+\e_0^2K C_1^2\exp(2C_1t) + \frac{1}{2}CKC'(C_0)C_1\exp(2C_1t).
\end{align*}
We now choose $C_1\geq \max\left(3CKC'(C_0),\sqrt{6}KC_{high}\right)$ and $\e_0\leq \frac{1}{\sqrt{6K}}$, so that each term of the previous inequality is bounded by $\frac{1}{6}C_1^2\exp(2C_1t)$. This concludes the proof.

\end{proof}

By continuity of the quantities involved, the previous proposition contradicts the maximality of $T_0$, and thus proves that $T_0=T$. As explained at the beginning of Section \ref{section theo 2}, this concludes the proof of Theorem \ref{theo 2}.

\appendix

\section{Computations in the elliptic gauge}\label{appendix A}

In this section, we collect some computations for the spacetime metric in the elliptic gauge defined in Section \ref{section geometrie}. See also \cite{hunluk18}.

\subsection{Connection coefficients}\label{connection coefficients}
The 2+1 metric $g$ has the form
\begin{equation*}
    g=-N^2\d t+\Bar{g}_{ij}\left(\d x^i+\beta^i\d t\right)\left(\d x^j+\beta^j\d t\right),
\end{equation*}
with $\Bar{g}=e^{2\gamma}\delta$. 
In the basis $(e_0,\dr_i)$, we have $g_{00}=-N^2$, $g_{0i}=0$ and $g_{ij}=e^{2\gamma}\delta_{ij}$, which gives $\det(g)=-e^{4\gamma}N^2$. 

In the basis $(\dr_t,\dr_i)$ we have :
\begin{equation}\label{inverse de g}
    g^{-1}=\frac{1}{N^2}
    \begin{pmatrix}
    -1 & \beta^1 & \beta^2 \\
    \beta^1 & N^2e^{-2\gamma}-\left(\beta^1\right)^2 & -\beta^1\beta^2 \\
    \beta^2 & -\beta^1\beta^2 & N^2e^{-2\gamma}-\left(\beta^2\right)^2
    \end{pmatrix}
\end{equation}
This allows us to compute $\Box_gh$ for $h$ a function on $\mathcal{M}$ :
\begin{prop}\label{appendix box}
If $h$ is a function on $\M$, we have
\begin{align*}
    \Box_gh & =-\Ll^2h+\frac{e^{-2\gamma}}{N}\dive(N\nabla h)+\tau\Ll h
    \\&=-\Ll^2h+e^{-2\gamma}\Delta h+\frac{e^{-2\gamma}}{N}\nabla h\cdot\nabla N+\tau\Ll h.
\end{align*}
\end{prop}

\begin{proof}
By definition of $\Box_g$, we have :
\begin{align*}
\Box_g h & = \frac{1}{\sqrt{|\det(g)|}}\dr_\beta\left(g^{\beta\alpha}\sqrt{|\det(g)|}\dr_\alpha h \right)
\\& = \frac{e^{-2\gamma}}{N}\dr_t \left(\frac{e^{2\gamma}}{N} \left( -\dr_t h +\beta^1\dr_1h+\beta^2\dr_2h \right) \right)
\\& \quad + \frac{e^{-2\gamma}}{N} \dr_1 \left( \frac{e^{2\gamma}}{N} \left( \beta^1\dr_t h+\left(N^2e^{-2\gamma}-\left(\beta^1\right)^2\right) \dr_1 h -\beta^1\beta^2\dr_2h \right) \right)
\\& \quad + \frac{e^{-2\gamma}}{N} \dr_2 \left( \frac{e^{2\gamma}}{N} \left( \beta^2\dr_t h -\beta^1\beta^2\dr_1h+\left(N^2e^{-2\gamma}-\left(\beta^2\right)^2\right) \dr_2 h \right) \right),
\end{align*}
where we used the expression of $g^{-1}$ in the basis $(\dr_t,\dr_i)$ (see expression \eqref{inverse de g}). By rearranging the terms, we get :
\begin{align*}
\Box_g h & =- \frac{1}{N}\dr_t\Ll h -\frac{2\dr_t\gamma}{N}\Ll h + \frac{e^{-2\gamma}}{N}\dive\left( e^{2\gamma}\Ll h\beta + N \nabla h \right) 
\\& = -\Ll^2h+\frac{e^{-2\gamma}}{N}\dive(N\nabla h)+\left(-2\Ll\gamma+\frac{\dive(\beta)}{N} \right)\Ll h.
\end{align*}
This proves the proposition, by looking at \eqref{appendix tau}.
\end{proof}

We now compute the connection coefficients for the metric \eqref{metrique elliptique} in the basis $(e_0,\dr_i)$. Notice that $[e_0,\dr_i]=\dr_i\beta^j\dr_j$. Using this, we compute :
\begin{align*}
    g(D_0e_0,e_0)&=\frac{1}{2}e_0g_{00}=-\frac{1}{2}e_0(N^2)=-Ne_0N,\\
    g(D_ie_0,e_0)&=\frac{1}{2}\dr_ig_{00}=-\frac{1}{2}\dr_i(N^2)=-N\dr_iN,\\
    g(D_0e_0,\dr_i)&=-\frac{1}{2}\dr_ig_{00}-g(e_0,[e_0,\dr_i])=N\dr_iN,\\
    g(D_ie_0,\dr_j)&=\frac{1}{2}\left(e_0g_{ij}-g(\dr_i,[e_0,\dr_j])-g(\dr_j,[e_0,\dr_i])\right)=\frac{e^{2\gamma}}{2}\left(2e_0\gamma\delta_{ij}-\dr_i\beta^k\delta_{jk}-\dr_j\beta^k\delta_{ik}\right), \\
    g(D_0\dr_i,e_0)&=\frac{1}{2}\dr_ig_{00}+g(e_0,[e_0,\dr_i])=-N\dr_iN,\\
    g(D_0\dr_i,\dr_j)&=\frac{1}{2}\left(e_0g_{ij}+g(\dr_i,[\dr_j,e_0])+g(\dr_j,[e_0,\dr_i])\right)=\frac{e^{2\gamma}}{2}\left(2e_0\gamma\delta_{ij}+\dr_i\beta^k\delta_{jk}-\dr_j\beta^k\delta_{ik}\right),\\
    g(D_i\dr_j,e_0)&=\frac{1}{2}\left(-e_0g_{ij}-g(\dr_i,[\dr_j,e_0])+g(\dr_j,[e_0,\dr_i])\right)=-\frac{e^{2\gamma}}{2}\left(2e_0\gamma\delta_{ij}-\dr_i\beta^k\delta_{jk}-\dr_j\beta^k\delta_{ik}\right),\\
    g(D_i\dr_j,\dr_k)&=e^{2\gamma}\left(\delta_{ik}\dr_j\gamma+\delta_{jk}\dr_i\gamma-\delta_{ij}\delta^{\ell}_k\dr_{\ell}\gamma \right).\\
\end{align*}
The first two expressions are derived using $X(g(Y,Z))=g(D_XY,Z)+g(Y,D_XZ)$ and the other ones with the Koszul formula : 
\begin{equation*}
    2g(D_VW,X)=Vg(W,X)+Wg(X,V)-Xg(V,W)-g(V,[W,X])+g(W,[X,V])+g(X,[V,W]).
\end{equation*}
From the above calculations, we obtain 
\begin{align}
    D_0e_0 &=\Ll N e_0 +e^{-2\gamma}\delta^{ij}N\dr_iN\dr_j , \label{covariante 1}\\
    D_0\dr_i &=\dr_iN\Ll+\frac{1}{2}\left(2\delta^{j}_ie_0\gamma+\dr_i\beta^j-\delta_{ik}\delta^{j\ell}\dr_{\ell}\beta^k \right)\dr_j ,\label{covariante 2}\\
    D_ie_0 &=\dr_iN\Ll+\frac{1}{2}\left(2\delta^{j}_ie_0\gamma-\dr_i\beta^j-\delta_{ik}\delta^{j\ell}\dr_{\ell}\beta^k \right)\dr_j ,\label{covariante 3}\\
    D_i\dr_j &=\frac{e^{2\gamma}}{2N}\left(2\delta_{ij}e_0\gamma-\left(\dr_i\beta^k \right)\delta_{jk}-\left(\dr_j\beta^k \right)\delta_{ik} \right)\Ll+\left(\delta^k_i\dr_j\gamma+\delta^k_j\dr_i\gamma-\delta_{ij}\delta^{k\ell}\dr_{\ell}\gamma\right)\dr_k.\label{covariante 4}
\end{align}

\subsection{Decomposition of the Ricci tensor}\label{ricci tensor}

\begin{prop}
Given $g$ of the form \eqref{metrique elliptique}, we have the following identities :
\begin{align}
    K_{ij}&=-\frac{\delta_{ij}}{2}\Ll\left( e^{2\gamma}\right)+\frac{e^{2\gamma}}{2N}\left(\dr_i\beta_j+\dr_j\beta_i\right),\label{seconde forme fonda}\\
    H_{ij}&=\frac{e^{2\gamma}}{2N}(L\beta)_{ij},\label{appendix beta}\\
    \tau &=-2\Ll\gamma+\frac{\dive\left(\beta\right)}{N}.\label{appendix tau}
\end{align}
\end{prop}

\begin{proof}
The equation \eqref{seconde forme fonda} follows from \eqref{Kij}, and \eqref{appendix beta} and \eqref{appendix tau} follow from \eqref{seconde forme fonda}.
\end{proof}

\begin{prop}
Given $g$ of the form \eqref{metrique elliptique}, the components of the Ricci tensor in the basis $(e_0,\dr_i)$ are given by
\begin{align}
    R_{ij} &=\delta_{ij}\left(-\Delta\gamma+\frac{\tau^2}{2}e^{2\gamma}-\frac{e^{2\gamma}}{2}\Ll\tau-\frac{\Delta N}{2N}\right) -\Ll H_{ij}-2e^{-2\gamma}H_i^{\;\ell}H_{j\ell}\label{appendix Rij}\\& \quad+\frac{1}{N}\left( \dr_j\beta^kH_{ki}+\dr_i\beta^kH_{kj}\right)-\frac{1}{N}\left(  \dr_i\dr_jN-\frac{1}{2}\delta_{ij}\Delta N-\left( \delta_i^k\dr_j\gamma+\delta_j^k\dr_i\gamma-\delta_{ij}\delta^{\ell k}\dr_{\ell}\gamma \right)\dr_k N  \right)\nonumber,\\
    R_{0j} &= N\left(\frac{1}{2}\dr_j\tau-e^{-2\gamma}\dr^iH_{ij} \right),\label{appendix R0j}\\
    R_{00} &= N\left(e_0\tau-e^{-4\gamma}N\left|H\right|^2-\frac{N\tau^2}{2}+e^{-2\gamma}\Delta N \right)\label{appendix R00}.
\end{align}
Moreover, 
\begin{align}
    \delta^{ij}R_{ij}&= 2\left(-\Delta\gamma+\frac{\tau^2}{2}e^{2\gamma}-\frac{e^{2\gamma}}{2}\Ll\tau-\frac{\Delta N}{2N}\right),\label{appendix trace ricci}\\
    R&=-2\Ll\tau +\frac{3}{2}\tau^2 + e^{-4\gamma}\left|H\right|^2-2e^{-2\gamma}\frac{\Delta N}{N} - 2e^{-2\gamma}\Delta\gamma.\label{appendix R}
\end{align}
\end{prop}

\begin{proof}
From Chapter 6 of \cite{cho09}, we have
\begin{align}
    R_{ij} & = \Bar{R}_{ij}+K_{ij}\tr_{\Bar{g}}K-2K_i^{\;\ell}K_{j\ell}- N^{-1}\left( \mathcal{L}_{e_0}K_{ij}+D_i\dr_jN\right),\label{Rij CB}\\
    R_{0j} & = N\left( \dr_j( \tr_{\Bar{g}}K)-D_{\ell}K^{\ell}_{\;j}\right) ,\label{R0j CB}\\
    R_{00} & = N\left( e_0(\tr_{\Bar{g}}K)-N|K|^2+\Delta_{\Bar{g}} N \right),\label{R00 CB}
\end{align}
where $\Bar{D}$, $\Bar{R}_{ij}$ and $\Delta_{\Bar{g}}$ are defined with respect to $\Bar{g}$. First, by \eqref{def H} and the connection coefficients computations, \eqref{Rij CB} becomes 
\begin{align}
    R_{ij} & =-\delta_{ij}\Delta\gamma+\tau\left( H_{ij}+\frac{1}{2}e^{2\gamma}\delta_{ij}\tau\right)-2e^{-2\gamma}\left( H_i^{\;\ell} +\frac{1}{2}e^{2\gamma}\delta_i^{\ell}\tau\right) \left( H_{j\ell} +\frac{1}{2}e^{2\gamma}\delta_{j\ell}\tau\right) \label{Rij inter}\\&\quad -\frac{1}{N} \left( \mathcal{L}_{e_0}K_{ij}+\dr_i\dr_jN- \left( \delta_i^k\dr_j\gamma+\delta_j^k\dr_i\gamma-\delta_{ij}\delta^{k\ell}\dr_{\ell}\gamma\right) \dr_kN  \right). \nonumber
\end{align}
To proceed, we compute $\mathcal{L}_{e_0}K_{ij}$ by considering $H_{ij}$ and $\tau$ :
\begin{align*}
    \mathcal{L}_{e_0}H_{ij}&=e_0H_{ij}-\dr_j\beta^kH_{ki}-\dr_i\beta^kH_{kj},\\
    \mathcal{L}_{e_0}(\tau\Bar{g}_{ij})&=e^{2\gamma}\delta_{ij}e_0\tau-2N\tau K_{ij}.
\end{align*}
Therefore, using \eqref{def H} and plugging $\mathcal{L}_{e_0}K_{ij}$ into \eqref{Rij inter}, we obtain \eqref{appendix Rij}. 

The expression of $R_{0j}$ in \eqref{appendix R0j} follows from \eqref{R0j CB} and the fact that for any covariant symmetric 2-tensor $A_ij$,
\begin{equation*}
    \Bar{g}^{ik}\Bar{D}_kA_{ij}=e^{-2\gamma}\dr^iA_ij-\dr_j\gamma\tr_{\Bar{g}}A.
\end{equation*}

Using \eqref{R00 CB} and the conformal invariance of the Laplacian we easily get \eqref{appendix R00}. 

To prove \eqref{appendix trace ricci}, we first note that
\begin{equation*}
    \delta^{ij}(\dr_j\beta^k H_{ki}+\dr_i\beta^k H_{kj})=H_{ij}(L\beta)^{ij}.
\end{equation*}
Combining this with \eqref{appendix beta}, we obtain
\begin{equation*}
    \delta^{ij}\left( -2e^{-2\gamma}H_i^{\;\ell}H_{j\ell}+\frac{1}{N}(\dr_j\beta^k H_{ki}+\dr_i\beta^k H_{kj}) \right)=0.
\end{equation*}
Taking the trace of \eqref{appendix Rij} and using this identity yield \eqref{appendix trace ricci}.

Finally, by putting \eqref{metrique elliptique}, \eqref{appendix R00} and \eqref{appendix trace ricci} we easily get \eqref{appendix R}.
\end{proof}

\subsection{The stress-energy-momentum tensor}\label{T mu nu}

Define $T_{\mu\nu}$ by
\begin{equation*}
    T_{\mu\nu}=2\dr_{\mu}\ffi\dr_{\nu}\ffi-g_{\mu\nu}g^{\alpha\beta}\dr_{\alpha}\ffi\dr_{\beta}\ffi+\frac{1}{2}e^{-4\ffi}\left( 2\dr_{\mu}\omega\dr_{\nu}\omega-g_{\mu\nu}g^{\alpha\beta}\dr_{\alpha}\omega\dr_{\beta}\omega\right).
\end{equation*}

\begin{prop}
The following identities are satisfied (with respect to the $(e_0,\dr_i)$ basis) :
\begin{align}
    T_{00}&= (e_0\ffi)^2+e^{-2\gamma}N^2|\nabla\ffi|^2+\frac{1}{4}e^{-4\ffi}\left(  (e_0\omega)^2+e^{-2\gamma}N^2|\nabla\omega|^2\right),\label{T 00}\\
    T_{0j}&= 2e_0\ffi\dr_j\ffi+\frac{1}{2}e^{-4\ffi}e_0\omega\dr_j\omega,\label{T 0j}\\
    T_{ij}&= 2\dr_i\ffi\dr_j\ffi+\frac{e^{2\gamma}}{N^2}(e_0\ffi)^2\delta_{ij}-|\nabla\ffi|^2\delta_{ij}\label{T ij}\\&\qquad+\frac{1}{4}e^{-4\ffi}\left( 2\dr_i\omega\dr_j\omega+\frac{e^{2\gamma}}{N^2}(e_0\omega)^2\delta_{ij}-|\nabla\omega|^2\delta_{ij}\right),\nonumber\\
    \tr_gT &= -g^{\alpha\beta}\dr_{\alpha}\ffi\dr_{\beta}\ffi-\frac{1}{4}e^{-4\ffi}g^{\alpha\beta}\dr_{\alpha}\omega\dr_{\beta}\omega,\\
    T_{00}-g_{00}\tr_gT &= 2\left(e_0\ffi\right)^2+\frac{1}{2}e^{-4\ffi}(e_0\omega)^2,\label{JSP}\\
    T_{ij}-g_{ij}\tr_gT &= 2\dr_{i}\ffi\dr_{j}\ffi+\frac{1}{2}e^{-4\ffi}\dr_{i}\omega\dr_{j}\omega,\label{JSP 2}\\
    \delta^{ij}\left( T_{ij}-g_{ij}\tr_gT\right) & =2\left|\nabla\ffi \right|^2+\frac{1}{2}e^{-4\ffi}\left|\nabla\omega \right|^2,\label{JSP 3}\\
    D^{\mu}T_{\mu\nu}&=2(\Box_g\ffi)\dr_{\nu}\ffi+\frac{1}{2}e^{-4\ffi}(\Box_g\omega)\dr_{\nu}\omega \label{divergence de T}\\&\qquad-e^{-4\ffi}\dr^\mu\ffi \left( 2\dr_{\mu}\omega\dr_{\nu}\omega-g_{\mu\nu}g^{\alpha\beta}\dr_{\alpha}\omega\dr_{\beta}\omega\right) .\nonumber
\end{align}
\end{prop}

\section{Weighted Sobolev spaces}\label{appendix B}
Here are some results about weighted Sobolev spaces on $\R^2$, which are systematically used during the proof. Most of them can be found in the Appendix I of \cite{cho09}. 

\begin{lem}\label{B1}
Let $m\geq 1$, $p\in[1,\infty)$ and $\delta\in\R$, then 
\begin{align*}
    \|\nabla u\|_{W^{m-1,p}_{\delta+1}} & \lesssim \| u\|_{W^{m,p}_{\delta}},\\
    \|\nabla u\|_{C^{m-1}_{\delta}} & \lesssim \| u\|_{C^{m}_{\delta+1}}.
\end{align*}
\end{lem}
We have an easy embedding result, which is a straightforward application of the Hölder’s inequality :
\begin{lem}\label{prop holder 2}
If $1\leq p_1\leq p_2\leq \infty$ and $\delta_2-\delta_1>2\left( \frac{1}{p_1}-\frac{1}{p_2} \right)$, then we have the continuous embedding 
\begin{equation*}
    L^{p_2}_{\delta_2}\xhookrightarrow{}L^{p_1}_{\delta_1}.
\end{equation*}
\end{lem}
Next, we have Sobolev embedding theorems for weighted Sobolev spaces :
\begin{prop}\label{embedding}
Let $s,m\in\N\cup\{0\}$, $1<p<\infty$.
\begin{itemize}
    \item If $s>\frac{2}{p}$ and $\beta\leq \delta+\frac{2}{p}$, then we have the continuous embedding
     \begin{equation*}
        W^{s+m,p}_{\delta}\xhookrightarrow{}C^m_{\beta}.
    \end{equation*}
    \item If $s<\frac{2}{p}$, then we have the continuous embedding
    \begin{equation*}
        W^{s+m,p}_{\delta}\xhookrightarrow{}W^{m,\frac{2p}{2-sp}}_{\delta+s}.
    \end{equation*}
\end{itemize}
\end{prop}
We will also need a product estimate. 
\begin{prop}\label{prop prod}
Let $s,s_1,s_2\in\N\cup\{0\}$, $p\in[1,\infty]$, $\delta,\delta_1,\delta_2\in\R$ such that $s\leq\min(s_1,s_2)$, $s<s_1+s_2-\frac{2}{p}$ and $\delta<\delta_1+\delta_2+\frac{2}{p}$. Then we have the continuous multiplication property 
\begin{equation*}
    W^{s_1,p}_{\delta_1}\times W^{s_2,p}_{\delta_2} \xhookrightarrow{}W^{s,p}_{\delta}. 
\end{equation*}
\end{prop}

The following simple lemma will be useful as well.
\begin{lem}
Let $\alpha\in\R$ and $g\in L^{\infty}_{loc}$ such that $|g(x)|\lesssim \langle x \rangle^{\alpha}$. Then the multiplication by $g$ map $L^2_{\delta+\alpha}$ to $L^2_{\delta}$ with operator norm bounded by $\sup_{x\in\R^2}\frac{|g(x)|}{\langle x\rangle^{\alpha}}$.
\end{lem}
The next result, which is due to McOwen, concerns the invertibility of the Laplacian on weighted Sobolev
spaces. Its proof can be found in \cite{mco79}.
\begin{thm}\label{mcowens 1}
Let $m,s\in\N\cup\{0\}$ and $-1+m<\delta<m$. The Laplace operator $\Delta:H^{s+2}_{\delta}\longrightarrow H^{s}_{\delta+2}$ is an injection with closed range
\begin{equation*}
    \enstq{f\in H^{s}_{\delta+2}}{ \forall v\in \cup_{i=0}^m\mathcal{H}_i,\; \int_{\R^2}fv=0},
\end{equation*}
where $\mathcal{H}_i$ is the set of harmonic polynomials of degree $i$. Moreover, $u$ obeys the estimate 
\begin{equation*}
    \| u\|_{H^{s+2}_{\delta}}\leq C(\delta,m,p)\|\Delta u\|_{H^{s}_{\delta+2}}.
\end{equation*}
\end{thm}
The following is a corollary of Theorem \ref{mcowens 1} :
\begin{coro}\label{mcowens 2}
Let $-1<\delta<0$ and $f\in H^0_{\delta+2}$. Then there exists a solution u of
\begin{equation*}
    \Delta u=f,
\end{equation*}
which can be written 
\begin{equation*}
    u=\frac{1}{2\pi}\left(\int_{\R^2}f\right)\chi(|x|)\ln(|x|)+v,
\end{equation*}
where $\chi$ is as in Section \ref{subsection initial data} and $\|v\|_{H^2_{\delta}}\leq C(\delta)\|f\|_{H^0_{\delta+2}}$.
\end{coro}

We will also use some classical inequalities, which we recall here, even if they are not related to weighted Sobolev spaces. The proof of the next property can be found in Appendix A of \cite{tao06}. 
\begin{prop}\label{littlewood paley}
If $s\in\N$, then 
\begin{equation*}
    \| uv\|_{H^s}\lesssim \|u\|_{H^s}\|v\|_{L^{\infty}}+\|v\|_{H^s}\|u\|_{L^{\infty}}.
\end{equation*}
\end{prop}
We recall the Hardy-Littlewood-Sobolev inequality :
\begin{prop}\label{prop HLS}
If $0<\alpha<2$ and $1<p<r<\infty$ and $\frac{1}{r}=\frac{1}{p}-\frac{\alpha}{2}$, then
\begin{equation*}
    \l u* \frac{1}{|\cdot|^{2-\alpha}} \r_{L^r}\lesssim \l u \r_{L^p}.
\end{equation*}
\end{prop}
We recall the Gagliardo-Nirenberg inequality, for which a proof can be found in \cite{fri69} :
\begin{prop}\label{GN}
Let $1\leq q,r\leq +\infty$, $m\in\N^*$. Let $\alpha\in\R$ and $j\in\N$ such that
\begin{equation*}
\frac{j}{m}\leq\alpha\leq 1.
\end{equation*}
Then :
\begin{equation*}
\l \nabla^j u \r_{L^p} \lesssim \l \nabla ^mu\r_{L^r}^\alpha \l  u \r_{L^q}^{1-\alpha},
\end{equation*}
with 
\begin{equation*}
\frac{1}{p}=\frac{j}{2}+\left(\frac{1}{r}-\frac{m}{2} \right)\alpha+\frac{1-\alpha}{q}.
\end{equation*}
\end{prop}

\section{Third order energy estimate}\label{appendix C}

In this section, we prove Proposition \ref{dernier coro}. We split the proof into two lemmas : their goal is to point out the dependence of $\frac{\d}{\d t}\mathscr{E}_3^\ffi$ and $\frac{\d}{\d t}\mathscr{E}_3^\omega$ on non-linear terms in $\dr\nabla^2 U$.

\begin{lem}
The energy $\mathscr{E}_3^\ffi$ satisfies
\begin{align*}
\frac{\d}{\d t}\mathscr{E}_3^\ffi & = \int_{\R^2}2e^{-4\ffi}e_0 \dr_j\dr_i\ffi \left( -\Ll \dr_j\dr_i\omega \Ll\omega+e^{-2\gamma}\nabla \dr_j\dr_i\omega \cdot\nabla\omega\right)\d x
\\&\quad +\int_{\R^2}2e^{-2\gamma}e^{-4\ffi}\nabla \dr_j\dr_i\ffi\cdot \left( e_0\dr_j\dr_i\omega \nabla\omega-e_0\omega\nabla \dr_j\dr_i\omega \right)\d x +O(\mathscr{R}(t)).
\end{align*}
\end{lem}

\begin{proof}
We split $\mathscr{E}_3^\ffi$ into two parts $A^\ffi+B^\ffi$ :
\begin{equation*}
\mathscr{E}_3^\ffi  =\underbrace{\int_{\R^2}  \frac{2}{N^2}\left(e_0\dr_j\dr_i\ffi+\frac{1}{2}e^{-4\ffi}\dr_j\dr_i\omega e_0\omega \right)^2\d x}_{A^\ffi\vcentcolon=}  +\underbrace{\int_{\R^2}2e^{-2\gamma}\left|\nabla\dr_j\dr_i\ffi+\frac{1}{2}e^{-4\ffi}\dr_j\dr_i\omega \nabla\omega \right|^2\d x}_{B^\ffi\vcentcolon=}.
\end{equation*}
We start with $A^\ffi$, by writing $\dr_t=e_0+\beta\cdot\nabla$. Note that if for some function $f$ we have $\l \nabla g f \r_{L^1}=O(\mathscr{R})$, then by integration by parts we have :
\begin{equation*}
\int_{\R^2}\dr_tf\d x = \int_{\R^2}e_0f\d x+\int_{\R^2}\beta\cdot\nabla f\d x =  \int_{\R^2}e_0f\d x - \int_{\R^2}\dive(\beta)f\d x= \int_{\R^2}e_0f\d x+O(\mathscr{R}).
\end{equation*}
Therefore, in what follows, we can forget about the $\beta\cdot\nabla$-part in $\dr_t$, which only contributes to $O(\mathscr{R})$. We now compute :
\begin{align*}
\frac{\d}{\d t}A^\ffi & = \int_{\R^2}4\left(e_0\dr_j\dr_i\ffi+\frac{1}{2}e^{-4\ffi}\dr_j\dr_i\omega e_0\omega \right)\left( \Ll^2\dr_j\dr_i\ffi+\frac{1}{2}e^{-4\ffi}\Ll \dr_j\dr_i\omega \Ll\omega+\frac{1}{2}e^{-4\ffi}\dr_j\dr_i\omega \Ll^2\omega \right)\d x 
\\&\quad   +O(\mathscr{R}(t)).
\end{align*}
We then replace terms involving $\Ll^2$ according to \eqref{expression de box}, and then replace $\Box_g \dr_j\dr_i\ffi$ according to \eqref{WM dd ffi} ($F^\ffi_{ij}$ and $\dr_j\dr_i\omega\Box_g\omega$ only contributes to $O(\mathscr{R})$) :
\begin{align*}
\frac{\d}{\d t}A^\ffi& = \int_{\R^2}4\left(e_0\dr_j\dr_i\ffi+\frac{1}{2}e^{-4\ffi}\dr_j\dr_i\omega e_0\omega \right) \left[ -\Box_g \dr_j\dr_i\ffi+\frac{e^{-2\gamma}}{N}\dive(N\nabla \dr_j\dr_i\ffi)+\frac{1}{2}e^{-4\ffi}\Ll \dr_j\dr_i\omega \Ll\omega \right.
\\&\qquad\qquad\qquad\qquad\qquad\qquad\qquad \left. -\frac{1}{2}e^{-4\ffi}\dr_j\dr_i\omega \Box_g\omega +\frac{e^{-2\gamma}}{2N}e^{-4\ffi}\dr_j\dr_i\omega \dive(N\nabla\omega) \right]\d x+O(\mathscr{R}(t))
\\& = \int_{\R^2}4\left(e_0\dr_j\dr_i\ffi+\frac{1}{2}e^{-4\ffi}\dr_j\dr_i\omega e_0\omega \right) \left[ -\frac{1}{2}e^{-4\ffi}\Ll \dr_j\dr_i\omega \Ll\omega+e^{-4\ffi}e^{-2\gamma}\nabla \dr_j\dr_i\omega \cdot\nabla\omega\right.
\\&\qquad\qquad\qquad\qquad\qquad\qquad\qquad  \left. +\frac{e^{-2\gamma}}{N}\dive(N\nabla \dr_j\dr_i\ffi)+\frac{e^{-2\gamma}}{2N}e^{-4\ffi}\dr_j\dr_i\omega \dive(N\nabla\omega) \right]\d x+O(\mathscr{R}(t)).
\end{align*}
We integrate by parts the terms with a divergence and expand :
\begin{align*}
\frac{\d}{\d t}A^\ffi & =  \int_{\R^2}4\left(e_0\dr_j\dr_i\ffi+\frac{1}{2}e^{-4\ffi}\dr_j\dr_i\omega e_0\omega \right) \left( -\frac{1}{2}e^{-4\ffi}\Ll \dr_j\dr_i\omega \Ll\omega+e^{-4\ffi}e^{-2\gamma}\nabla \dr_j\dr_i\omega \cdot\nabla\omega\right)\d x 
\\& \quad -\int_{\R^2}4e^{-2\gamma}\nabla \dr_j\dr_i\ffi\cdot \nabla\left(e_0\dr_j\dr_i\ffi+\frac{1}{2}e^{-4\ffi}\dr_j\dr_i\omega e_0\omega \right)\d x
\\& \quad -\int_{\R^2}2e^{-4\ffi}e^{-2\gamma}\nabla\omega\cdot\nabla\left( \dr_j\dr_i\omega \left(e_0\dr_j\dr_i\ffi+\frac{1}{2}e^{-4\ffi}\dr_j\dr_i\omega e_0\omega \right)\right)\d x+O(\mathscr{R}(t))
\\& = \int_{\R^2}4\left(e_0\dr_j\dr_i\ffi+\frac{1}{2}e^{-4\ffi}\dr_j\dr_i\omega e_0\omega \right) \left( -\frac{1}{2}e^{-4\ffi}\Ll \dr_j\dr_i\omega \Ll\omega+e^{-4\ffi}e^{-2\gamma}\nabla \dr_j\dr_i\omega \cdot\nabla\omega\right)\d x
\\& \quad -\int_{\R^2}4e^{-2\gamma}\nabla \dr_j\dr_i\ffi\cdot\nabla e_0 \dr_j\dr_i\ffi \d x -\int_{\R^2}2e^{-4\ffi}e^{-2\gamma}\nabla \dr_j\dr_i\ffi\cdot\nabla (\dr_j\dr_i\omega e_0\omega)\d x
\\& \quad -\int_{\R^2}2e^{-4\ffi}e^{-2\gamma}(\nabla\omega\cdot\nabla \dr_j\dr_i\omega )\left(e_0\dr_j\dr_i\ffi+\frac{1}{2}e^{-4\ffi}\dr_j\dr_i\omega e_0\omega \right)\d x
\\& \quad -\int_{\R^2}2e^{-4\ffi}e^{-2\gamma}\dr_j\dr_i\omega \nabla\omega\cdot\nabla e_0\dr_j\dr_i\ffi\d x - \int_{\R^2}e^{-8\ffi}e^{-2\gamma}\dr_j\dr_i\omega \nabla\omega\cdot\nabla(\dr_j\dr_i\omega e_0\omega)\d x +O(\mathscr{R}(t))
\\& = \int_{\R^2}2e_0\dr_j\dr_i\ffi \left( -e^{-4\ffi}\Ll \dr_j\dr_i\omega \Ll\omega+e^{-4\ffi}e^{-2\gamma}\nabla \dr_j\dr_i\omega \cdot\nabla\omega\right)\d x
\\& \quad  -\int_{\R^2}2e^{-4\ffi}e^{-2\gamma}e_0\omega\nabla \dr_j\dr_i\ffi\cdot\nabla \dr_j\dr_i\omega \d x
 -\int_{\R^2}4e^{-2\gamma}\nabla \dr_j\dr_i\ffi\cdot\nabla e_0 \dr_j\dr_i\ffi \d x \\& \quad-\int_{\R^2}2e^{-4\ffi}e^{-2\gamma}\dr_j\dr_i\omega \nabla\omega\cdot\nabla e_0\dr_j\dr_i\ffi\d x  +O(\mathscr{R}(t))
\end{align*}
We now deal with $B^\ffi$ :
\begin{align*}
\frac{\d}{\d t}B^\ffi & = \int_{\R^2}4e^{-2\gamma} \left( \nabla \dr_j\dr_i\ffi+\frac{1}{2}e^{-4\ffi}\dr_j\dr_i\omega \nabla\omega \right)\cdot e_0\left( \nabla \dr_j\dr_i\ffi+\frac{1}{2}e^{-4\ffi}\dr_j\dr_i\omega \nabla\omega \right)\d x +O(\mathscr{R}(t))
\\& = \int_{\R^2}2e^{-2\gamma}e^{-4\ffi}e_0\dr_j\dr_i\omega \nabla\dr_j\dr_i\ffi\cdot \nabla\omega \d x
\\&\quad+ \int_{\R^2}4e^{-2\gamma}\nabla \dr_j\dr_i\ffi\cdot \nabla e_0 \dr_j\dr_i\ffi \d x +\int_{\R^2}2e^{-2\gamma}e^{-4\ffi}\dr_j\dr_i\omega \nabla\omega\cdot\nabla e_0\dr_j\dr_i\ffi \d x +O(\mathscr{R}(t)).
\end{align*}
We see that the terms which contains $\nabla e_0 \dr_j\dr_i\ffi $ in $A^\ffi$ and $B^\ffi$ cancel each other, and that every terms wich are linear in $\dr \nabla^2 U$ only contribute to $O(\mathscr{R}(t))$, so that :
\begin{align*}
\frac{\d}{\d t}\mathscr{E}_3^\ffi & = \int_{\R^2}2e^{-4\ffi}e_0 \dr_j\dr_i\ffi \left( -\Ll \dr_j\dr_i\omega \Ll\omega+e^{-2\gamma}\nabla \dr_j\dr_i\omega \cdot\nabla\omega\right)\d x
\\&\quad +\int_{\R^2}2e^{-2\gamma}e^{-4\ffi}\nabla \dr_j\dr_i\ffi\cdot \left( e_0\dr_j\dr_i\omega \nabla\omega-e_0\omega\nabla \dr_j\dr_i\omega \right)\d x +O(\mathscr{R}(t)).
\end{align*}

\end{proof}

\begin{lem}
The energy $\mathscr{E}_3^\omega$ satisfies
\begin{align*}
\frac{\d}{\d t}\mathscr{E}_3^\omega & =\int_{\R^2}2e^{-4\ffi}e_0\dr_j\dr_i\omega \left( \Ll\dr_j\dr_i\ffi\Ll\omega -e^{-2\gamma}\nabla\dr_j\dr_i\ffi\cdot\nabla\omega \right) \d x
 \\& \quad +\int_{\R^2}2e^{-2\gamma}e^{-4\ffi}\nabla\dr_j\dr_i\omega \cdot \left( e_0\omega\nabla\dr_j\dr_i\ffi-e_0\dr_j\dr_i\ffi\nabla\omega \right) \d x+O(\mathscr{R}(t)).
\end{align*}
\end{lem}

\begin{proof}
The proof of this lemma is very similar to the one of the previous lemma, except that we also differenciate the coefficient $e^{-4\ffi}$ in the energy $\mathscr{E}_3^\omega$. We split $\mathscr{E}_3^\omega$ into two parts $A^\omega+B^\omega$ :
\begin{align*}
\mathscr{E}_3^\ffi & = \underbrace{\int_{\R^2}  \frac{1}{2N^2}e^{-4\ffi}  \left( e_0\dr_j\dr_i\omega -2\dr_j\dr_i\omega e_0\ffi-2\dr_j\dr_i\ffi e_0\omega \right)^2\d x}_{A^\omega\vcentcolon=}  \\&\qquad\qquad\qquad\qquad+\underbrace{\int_{\R^2}\frac{1}{2}e^{-4\ffi}e^{-2\gamma}\left|\nabla\dr_j\dr_i\omega -2\dr_j\dr_i\omega \nabla\ffi-2\dr_j\dr_i\ffi \nabla\omega \right|^2\d x}_{B^\omega\vcentcolon=}.
\end{align*}
We start by $A^\omega$ :
\begin{align*}
\frac{\d}{\d t}A^\omega & = \int_{\R^2}e^{-4\ffi} \left( e_0\dr_j\dr_i\omega  -2\dr_j\dr_i\omega  e_0\ffi-2 \dr_j\dr_i\ffi  e_0\omega \right)\left( \Ll^2\dr_j\dr_i\omega -2\Ll \dr_j\dr_i\omega \Ll\ffi -2\dr_j\dr_i\omega \Ll^2\ffi\right.\\& \qquad\qquad\qquad\qquad\qquad\qquad\qquad\qquad\qquad\left.-2\Ll  \dr_j\dr_i\ffi \Ll\omega -2 \dr_j\dr_i\ffi \Ll^2\omega \right) \d x
\\&\quad -\int_{\R^2}2e^{-4\ffi}e_0\ffi \left( \Ll \dr_j\dr_i\omega  -2\dr_j\dr_i\omega  \Ll\ffi-2 \dr_j\dr_i\ffi  \Ll\omega \right)^2\d x +O(\mathscr{R}(t))
\\& =  \int_{\R^2}e^{-4\ffi} \left( e_0\dr_j\dr_i\omega  -2\dr_j\dr_i\omega  e_0\ffi-2 \dr_j\dr_i\ffi  e_0\omega \right) \left[ -\Box_g\dr_j\dr_i\omega +\frac{e^{-2\gamma}}{N}\dive(N\nabla \dr_j\dr_i\omega ) + 2\dr_j\dr_i\omega \Box_g\ffi \right. 
\\& \qquad\qquad\qquad\qquad\qquad\qquad\qquad\qquad\qquad \left.- \frac{2e^{-2\gamma}}{N}\dr_j\dr_i\omega \dive(N\nabla\ffi) +2 \dr_j\dr_i\ffi \Box_g\omega \right. \\& \qquad\qquad\qquad\qquad\qquad\qquad\qquad\qquad\qquad\left.- \frac{2e^{-2\gamma}}{N} \dr_j\dr_i\ffi \dive(N\nabla\omega) -2\Ll \dr_j\dr_i\omega \Ll\ffi-2\Ll  \dr_j\dr_i\ffi \Ll\omega  \right]\d x
\\&\qquad\qquad\qquad-\int_{\R^2}2e^{-4\ffi}e_0\ffi \left( \Ll \dr_j\dr_i\omega   \right)^2\d x+O(\mathscr{R}(t)).
\end{align*}
We integrate by parts the terms with a divergence (note that we differenciate the $e^{-4\ffi}$, but the one with $\dive(N\nabla\dr_j\dr_i\omega)$ in front is the only divergence term which gives a main term) :
\begin{align*}
\frac{\d}{\d t}A^\omega & = \int_{\R^2}2e^{-4\ffi} \left( e_0\dr_j\dr_i\omega  -2\dr_j\dr_i\omega  e_0\ffi-2 \dr_j\dr_i\ffi  e_0\omega \right)\left( \Ll \dr_j\dr_i\omega \Ll\ffi+ \Ll  \dr_j\dr_i\ffi \Ll\omega \right. \\&  \qquad\qquad\qquad\qquad\qquad\qquad\qquad\qquad\qquad\left. -2e^{-2\gamma}\nabla \dr_j\dr_i\omega \cdot\nabla\ffi-2e^{-2\gamma}\nabla  \dr_j\dr_i\ffi \cdot\nabla\omega \right) \d x
\\& \quad -\int_{\R^2}e^{-4\ffi}e^{-2\gamma}\nabla \dr_j\dr_i\omega \cdot\nabla\left( e_0\dr_j\dr_i\omega  -2\dr_j\dr_i\omega  e_0\ffi-2 \dr_j\dr_i\ffi  e_0\omega \right)\d x
\\&\quad+\int_{\R^2}4e^{-4\ffi}e^{-2\gamma}\nabla\ffi\cdot\nabla\dr_j\dr_i\omega \left( e_0\dr_j\dr_i\omega  -2\dr_j\dr_i\omega  e_0\ffi-2 \dr_j\dr_i\ffi  e_0\omega \right) \d x
\\&\quad +\int_{\R^2}2e^{-4\ffi}e^{-2\gamma}\nabla\ffi\cdot \nabla\left( \dr_j\dr_i\omega \left( e_0\dr_j\dr_i\omega  -2\dr_j\dr_i\omega  e_0\ffi-2 \dr_j\dr_i\ffi  e_0\omega \right)\right)\d x
\\&\quad +\int_{\R^2}2e^{-4\ffi}e^{-2\gamma}\nabla\omega\cdot\nabla \left(  \dr_j\dr_i\ffi  \left( e_0\dr_j\dr_i\omega  -2\dr_j\dr_i\omega  e_0\ffi-2 \dr_j\dr_i\ffi  e_0\omega \right)\right)\d x \\& \quad-\int_{\R^2}2e^{-4\ffi}e_0\ffi \left( \Ll \dr_j\dr_i\omega   \right)^2\d x+O(\mathscr{R}(t)).
\end{align*}
We now expand all the terms and note again that the linear terms in $\dr\nabla^2 U$ only contribute to $O(\mathscr{R}(t))$ :
\begin{align*}
\frac{\d}{\d t}A^\omega & =  \int_{\R^2}2e^{-4\ffi}  \left( e_0\dr_j\dr_i\omega  -2\dr_j\dr_i\omega  e_0\ffi-2 \dr_j\dr_i\ffi  e_0\omega \right)  \left( \Ll \dr_j\dr_i\omega \Ll\ffi+ \Ll  \dr_j\dr_i\ffi \Ll\omega  \right. \\& \qquad\qquad\qquad\qquad\qquad\qquad\qquad\qquad\qquad\qquad\qquad \left.-2e^{-2\gamma}\nabla \dr_j\dr_i\omega \cdot\nabla\ffi-2e^{-2\gamma}\nabla  \dr_j\dr_i\ffi \cdot\nabla\omega \right) \d x
\\& \quad  +\int_{\R^2}2e^{-4\ffi}e^{-2\gamma}e_0\ffi|\nabla \dr_j\dr_i\omega|^2 \d x +\int_{\R^2}2e^{-4\ffi}e^{-2\gamma}e_0\omega\nabla \dr_j\dr_i\omega \cdot\nabla \dr_j\dr_i\ffi \d x
\\&\quad +\int_{\R^2}6e^{-4\ffi}e^{-2\gamma}(\nabla\ffi\cdot\nabla \dr_j\dr_i\omega )  e_0\dr_j\dr_i\omega    \d x  +\int_{\R^2}2e^{-4\ffi}e^{-2\gamma}(\nabla\omega\cdot\nabla  \dr_j\dr_i\ffi )  e_0\dr_j\dr_i\omega    \d x 
\\&\quad +\int_{\R^2}2e^{-4\ffi}e^{-2\gamma} \dr_j\dr_i\ffi \nabla\omega\cdot\nabla e_0\dr_j\dr_i\omega \d x  -\int_{\R^2}e^{-4\ffi}e^{-2\gamma}\nabla \dr_j\dr_i\omega \cdot\nabla e_0\dr_j\dr_i\omega \d x
\\&\quad
+\int_{\R^2}2e^{-4\ffi}e^{-2\gamma}\dr_j\dr_i\omega \nabla\ffi\cdot\nabla e_0\dr_j\dr_i\omega \d x -\int_{\R^2}2e^{-4\ffi}e_0\ffi \left( \Ll \dr_j\dr_i\omega   \right)^2\d x+O(\mathscr{R}(t))
\\& =  \int_{\R^2}2e^{-4\ffi} e_0\dr_j\dr_i\omega    \left(  \Ll  \dr_j\dr_i\ffi \Ll\omega +e^{-2\gamma}\nabla \dr_j\dr_i\omega \cdot\nabla\ffi-e^{-2\gamma}\nabla  \dr_j\dr_i\ffi \cdot\nabla\omega \right) \d x
\\& \quad  +\int_{\R^2}2e^{-4\ffi}e^{-2\gamma}e_0\ffi|\nabla \dr_j\dr_i\omega|^2 \d x +\int_{\R^2}2e^{-4\ffi}e^{-2\gamma}e_0\omega\nabla \dr_j\dr_i\omega \cdot\nabla \dr_j\dr_i\ffi \d x
\\&\quad +\int_{\R^2}e^{-4\ffi}e^{-2\gamma}(-\nabla \dr_j\dr_i\omega +2\dr_j\dr_i\omega \nabla\ffi+2 \dr_j\dr_i\ffi \nabla\omega)\cdot \nabla e_0\dr_j\dr_i\omega \d x +O(\mathscr{R}(t))
\end{align*}
We now deal with $B^\omega$ :
\begin{align*}
\frac{\d}{\d t}B^\omega & =\int_{\R^2}e^{-4\ffi}e^{-2\gamma}\left(\nabla \dr_j\dr_i\omega -2\dr_j\dr_i\omega \nabla\ffi-2 \dr_j\dr_i\ffi \nabla\omega \right)\cdot e_0\left(\nabla \dr_j\dr_i\omega -2\dr_j\dr_i\omega \nabla\ffi-2 \dr_j\dr_i\ffi \nabla\omega \right)\d x
\\&\quad -\int_{\R^2}2e^{-4\ffi}e^{-2\gamma}e_0\ffi\left|\nabla \dr_j\dr_i\omega -2\dr_j\dr_i\omega \nabla\ffi-2 \dr_j\dr_i\ffi \nabla\omega \right|^2\d x
\\& = \int_{\R^2}e^{-4\ffi}e^{-2\gamma}(\nabla \dr_j\dr_i\omega -2\dr_j\dr_i\omega \nabla\ffi-2 \dr_j\dr_i\ffi \nabla\omega)\cdot \nabla e_0\dr_j\dr_i\omega \d x
\\& \quad -\int_{\R^2}2e^{-4\ffi}e^{-2\gamma}e_0\dr_j\dr_i\omega\nabla \dr_j\dr_i\omega\cdot  \nabla\ffi\d x -\int_{\R^2}2e^{-4\ffi}e^{-2\gamma}e_0 \dr_j\dr_i\ffi\nabla \dr_j\dr_i\omega \cdot  \nabla\omega\d x
\\&\quad- \int_{\R^2}2e^{-4\ffi}e^{-2\gamma}e_0\ffi|\nabla \dr_j\dr_i\omega |^2\d x+O(\mathscr{R}(t)).
\end{align*}
We see that the terms which contains $\nabla e_0 \dr_j\dr_i\omega $ in $A^\omega$ and $B^\omega$ cancel each other, therefore :
\begin{align*}
\frac{\d}{\d t}\mathscr{E}_3^\omega & = \int_{\R^2}2e^{-4\ffi} e_0\dr_j\dr_i\omega    \left(  \Ll  \dr_j\dr_i\ffi \Ll\omega +e^{-2\gamma}\nabla \dr_j\dr_i\omega \cdot\nabla\ffi-e^{-2\gamma}\nabla  \dr_j\dr_i\ffi \cdot\nabla\omega \right) \d x
\\& \quad  +\int_{\R^2}2e^{-4\ffi}e^{-2\gamma}e_0\omega\nabla \dr_j\dr_i\omega \cdot\nabla \dr_j\dr_i\ffi \d x
 -\int_{\R^2}2e^{-4\ffi}e^{-2\gamma}e_0\dr_j\dr_i\omega\nabla \dr_j\dr_i\omega\cdot  \nabla\ffi\d x \\& \quad -\int_{\R^2}2e^{-4\ffi}e^{-2\gamma}e_0 \dr_j\dr_i\ffi\nabla \dr_j\dr_i\omega \cdot  \nabla\omega\d x
 +O(\mathscr{R}(t))
 \\& = \int_{\R^2}2e^{-4\ffi}e_0\dr_j\dr_i\omega \left( \Ll\dr_j\dr_i\ffi\Ll\omega -e^{-2\gamma}\nabla\dr_j\dr_i\ffi\cdot\nabla\omega \right) \d x
 \\& \quad +\int_{\R^2}2e^{-2\gamma}e^{-4\ffi}\nabla\dr_j\dr_i\omega \cdot \left( e_0\omega\nabla\dr_j\dr_i\ffi-e_0\dr_j\dr_i\ffi\nabla\omega \right) \d x+O(\mathscr{R}(t))
 \\& = \int_{\R^2}2e^{-4\ffi}e_0\dr_j\dr_i\omega \left( \Ll\dr_j\dr_i\ffi\Ll\omega -e^{-2\gamma}\nabla\dr_j\dr_i\ffi\cdot\nabla\omega \right) \d x
 \\& \quad +\int_{\R^2}2e^{-2\gamma}e^{-4\ffi}\nabla\dr_j\dr_i\omega \cdot \left( e_0\omega\nabla\dr_j\dr_i\ffi-e_0\dr_j\dr_i\ffi\nabla\omega \right) \d x+O(\mathscr{R}(t)).
\end{align*}

Adding the two previous lemmas, we see that the main parts of $\frac{\d}{\d t}\mathscr{E}_3^\ffi$ and $\frac{\d}{\d t}\mathscr{E}_3^\omega$ cancel each other, and we obtain Proposition \ref{dernier coro}.

\end{proof}

\end{document}